\documentclass[letterpaper, 11pt]{amsart}
\usepackage{mathtools}
\usepackage{amsmath}
\usepackage{amssymb}
\usepackage{yhmath}
\usepackage{graphicx}
\usepackage{mathrsfs}
\usepackage{bbm}
\usepackage{xcolor}
\usepackage{tikz-cd}
\usepackage{tikz}
\usetikzlibrary{patterns}
\usepackage{hyperref}

\usepackage{verbatim}
\usepackage{float}

\DeclareMathAlphabet{\mathpzc}{OT1}{pzc}{m}{it}

\usepackage{thmtools}
\usepackage{thm-restate}

\usepackage{caption}

\newtheorem{theorem}{Theorem}[section]

\newtheorem*{claim*}{Claim}

\newtheorem{lemma}[theorem]{Lemma}

\newtheorem{corollary}[theorem]{Corollary}

\newtheorem{proposition}[theorem]{Proposition}

\theoremstyle{definition}
\newtheorem{definition}[theorem]{Definition}
\newtheorem{Def}[theorem]{Definition}

\theoremstyle{remark}
\newtheorem{remark}[theorem]{Remark}

\newtheorem{Rmk}[theorem]{Remark}

\numberwithin{equation}{section}

\newcommand{\op}{\operatorname}

\newcommand{\be}{\begin{equation}}
\newcommand{\ee}{\end{equation}}
\newcommand{\Ga}{\Gamma}

\newcommand{\R}{\mathbb R}

\newcommand{\N}{\mathbb N}
\newcommand{\ga}{\gamma}

\newcommand{\la}{\lambda}
\newcommand{\La}{\Lambda}

\newcommand{\ba}{\backslash}

\newcommand{\cal}{\mathcal}
\newcommand{\br}{\mathbb R}
\newcommand{\SO}{\op{SO}}
\newcommand{\Isom}{\op{Isom}}

\newcommand{\F}{\cal F}

\newcommand{\bH}{\mathbb H}

\newcommand{\stab}{\op{Stab}}

\newcommand{\G}{\Gamma}
\newcommand{\m}{\mathsf{m}^{\BMS}}

\newcommand{\T}{\op{T}}
\renewcommand{\frak}{\mathfrak}

\newcommand{\e}{\varepsilon}

\newcommand{\BMS}{\op{BMS}}

\newcommand{\fa}{\mathfrak a}

\renewcommand{\i}{\op{i}}

\newcommand{\fg}{\frak g}

\newcommand{\Leb}{\op{Leb}}

\newcommand{\lat}{\La_{\theta}}
\newcommand{\ft}{\F_\theta}

\renewcommand{\epsilon}{\e}

\renewcommand{\d}{\mathsf{d}}

\newcommand{\SL}{\op{SL}}

 \title[Finiteness, Measure of maximal entropy and reparameterization]{Relatively Anosov groups: finiteness,\\measure of maximal entropy, and reparameterization}

\author{Dongryul M. Kim}
\address{Department of Mathematics, Yale University, New Haven, CT 06511}
\email{dongryul.kim@yale.edu}

\author{Hee Oh}
\address{Department of Mathematics, Yale University, New Haven, CT 06511}
\email{hee.oh@yale.edu}

\thanks{Oh is partially supported by the NSF grant No. DMS-1900101.}

\begin{document}

\begin{abstract} For a geometrically finite Kleinian group $\Gamma$, the Bowen-Margulis-Sullivan measure is finite  and is the unique measure of maximal entropy for the geodesic flow, as shown by Sullivan and Otal-Peign\'e respectively. Moreover, it is strongly mixing by a result of Babillot.
We obtain a higher-rank analogue of this theorem.
Given a relatively Anosov subgroup $\Gamma$ of a semisimple real algebraic group,
there is a family of flow spaces parameterized by linear forms tangent to the growth indicator. 
We construct a reparameterization of each flow space by the geodesic flow on the Groves-Manning space of $\Gamma$ which exhibits exponential expansion along unstable foliations.
Using this reparameterization, we prove that the Bowen-Margulis-Sullivan measure of each flow space is finite and is the unique measure of maximal entropy. Moreover, it is strongly mixing.

\end{abstract}

\maketitle

\tableofcontents

\section{Introduction}
For a geometrically finite Kleinian group $\Ga$ of $\SO^\circ(n,1)=\Isom^+(\bH^n)$, $n\ge 2$, it is a classical result of Sullivan  (\cite{Sullivan1984entropy}, see also \cite{Corlette1999limit}) that
the associated Bowen-Margulis-Sullivan measure $m^{\BMS}$ on the unit tangent bundle $\T^1(\Ga\ba \bH^n)$ is finite, and the measure-theoretic entropy of the geodesic flow with respect to $m^{\BMS}$ equals the topological entropy. Hence the Bowen-Margulis-Sullivan measure is the measure of maximal entropy. Moreover, Otal-Peign\'e \cite{OP_variational} showed that this measure is the unique measure of maximal entropy. It is also strongly mixing by
a theorem of Babillot \cite{Babillot_mixing}.

In this paper, we obtain higher-rank analogues of these theorems. Let $G$ be a connected semisimple real algebraic group. Anosov subgroups and relatively Anosov subgroups of $G$ are regarded as higher-rank generalizations of convex cocompact and geometrically finite rank-one groups, respectively. There is an even broader class of discrete subgroups called transverse subgroups, which are viewed as generalizations of rank-one discrete subgroups.
For a transverse subgroup $\Ga$, we have a family of Bowen-Margulis-Sullivan measures $\m_{\psi}$ parameterized by a distinguished collection of linear forms $\psi$. Each such measure $\m_\psi$ lives on a fibered dynamical system over a canonical one-dimensional base flow space $(\Omega_\psi, m_\psi, \phi_t)$ where the fiber is the
kernel of $\psi$ and
$\m_\psi$ is equal to 
 the product measure $m_\psi \otimes \op{Leb}_{\ker\psi}$. We refer to $m_{\psi}$ as the base Bowen-Margulis-Sullivan measure on $\Omega_{\psi}$.

We prove that if $\Ga$ is a relatively Anosov subgroup, then
the base BMS measure $m_{\psi}$ is finite and  is the unique measure of maximal entropy for the flow $\{\phi_t\}$. Moreover, we show that
for any transverse subgroup for which $m_\psi$ is finite, the dynamical system $(\Omega_{\psi}, m_{\psi}, \phi_t)$ is strongly mixing. In particular, both entropy-maximization and strong mixing holds for 
 $(\Omega_{\psi}, m_{\psi}, \phi_t)$ associated with relatively Anosov subgroups.

To formulate these results precisely, we fix a Cartan decomposition $G=KA^+K$, where
$K$ is a maximal compact subgroup of $G$ and $A^+=\exp \fa^+$ is a positive Weyl chamber of a maximal split torus $A$ of $G$.  We denote by $\mu : G \to \fa^+$ the Cartan projection defined by the condition $g\in K\exp \mu(g) K$ for $g \in G$. 
Let $\Pi$ be the set of all simple roots for $(\op{Lie} G, \frak a^+)$.  Given a non-empty subset $\theta\subset \Pi$, 
there is the notion of relatively Anosov and transverse subgroup. 
Let $\F_\theta=G/P_\theta$ where
$P_\theta$ is the standard parabolic subgroup associated with $\theta$. 
Let $\G<G$ be a discrete subgroup and let $\La_\theta$ denote the limit set of $\Ga$ in $\F_{\theta}$ as defined in \eqref{def.limitset}, which we assume contains at least $3$ points, that is, $\Ga$ is non-elementary.
In the rest of the introduction, we assume that $\Ga$ is
a {\it $\theta$-transverse} (or simply, transverse) subgroup. This means that $\Ga$ satisfies
\begin{itemize}
    \item {\it regularity}:
$ \liminf_{\ga\in \Ga} \alpha(\mu({\ga}))=\infty $ for all $\alpha\in \theta$; 
\item {\it antipodality}:
any $\xi\ne  \eta\in \La_{\theta\cup \i(\theta)}$ 
are in general position (see \eqref{gp}).   
\end{itemize}
Here $\i = -\op{Ad}_{w_0}:\Pi \to \Pi $ denotes the opposition involution
where $w_0$ is the longest Weyl element.

\subsection*{Fibered dynamical systems}
Let $\fa_\theta= \bigcap_{\alpha \in \Pi - \theta} \ker \alpha$ and $A_\theta=\exp \fa_\theta$. The centralizer of $A_\theta$ is a Levi subgroup of $P_\theta$ which is a direct product $A_\theta S_\theta$ where $S_\theta$ is a compact central extension of a semisimple algebraic subgroup.
The right translation action of $A_\theta$ on the
quotient space $G/S_\theta$ is equivariantly conjugate to the $\fa_\theta$-translation action on $\F_\theta^{(2)}\times \fa_\theta$ where $\F_\theta^{(2)}$ consists of all pairs
$(\xi, \eta)\in \F_\theta\times \F_{\i(\theta)}$ in general position.
The left $\Ga$-action on $G/S_{\theta}$ is not properly discontinuous in general.
On the other hand, if we set $\La_\theta^{(2)}=(\La_\theta\times \La_{\i(\theta)})\cap \F_\theta^{(2)}$, then  it is shown in \cite[Theorem 9.1]{KOW_indicators} that $\Ga$ acts  properly discontinuously on the following
space:
$$\tilde{\Omega}_{\Ga}:= \La_\theta^{(2)}\times \fa_\theta \simeq  \{ g S_{\theta} \in G/S_\theta: gP_\theta\in \La_\theta, g w_0 P_{\i(\theta)}\in \La_{\i(\theta)}\}.$$ 
Hence $$\Omega_{\Ga}:=\Ga\ba \tilde \Omega_{\Ga}.$$
is a  second countable locally compact Hausdorff space on which $\fa_\theta$ acts by translations.
Moreover, for each  $(\Ga, \theta)$-proper\footnote{
$\psi$ is called $(\Ga, \theta)$-proper if $\psi \circ \mu : \Ga \to [-\varepsilon, \infty)$ is a proper map for some $\e>0$.} linear form $\psi \in \fa_{\theta}^*$, the space ${\Omega}_{\Ga}$ fibers over a one-dimensional flow space
$\Omega_\psi := \Ga\ba (\La_\theta^{(2)}\times \br)$.

More precisely,
via the projection $(\xi, \eta, v) \mapsto (\xi, \eta, \psi(v))$, the $\Ga$-action on $\tilde \Omega_{\Ga}$ descends to a proper discontinuous action on $ \tilde \Omega_{\psi} := \La_\theta^{(2)}\times \br$ \cite[Theorem 9.2]{KOW_indicators}. 
Therefore $\Omega_{\psi} := \Ga \ba \tilde \Omega_{\psi} $ is a second countable locally compact Hausdorff space over which $\Omega_{\Ga}$ is a trivial $\ker\psi$-bundle:  $$\begin{tikzcd}[column sep = tiny]
    {\color{white} \simeq \Omega_{\psi} \times \ker \psi} \hspace{-1.3em}  & (\Omega_{\Ga},\fa_\theta) \arrow[d]  & \hspace{-1.3em} \simeq \Omega_{\psi} \times \ker \psi \\
    & (\Omega_{\psi}, \br) &
\end{tikzcd}$$
The translation flow  $\phi_t(\xi, \eta, s) = (\xi, \eta, s + t)$ on $ \tilde \Omega_{\psi} = \La_\theta^{(2)}\times \br $ descends to a translation flow  on $\Omega_{\psi}$ which we also denote by $\{\phi_t\}$ by abuse of notation.
 The $(\Ga, \theta)$-properness of $\psi \in \fa_{\theta}^*$ is crucial for the proper discontinuity of the $\Ga$-action on $\tilde \Omega_{\psi}$.  See Remark \ref{rmk.exampleproper} for examples.

 For a pair of a $(\Ga, \psi)$-Patterson-Sullivan measure $\nu$ on $\La_{\theta}$ and a $(\Ga, \psi \circ \i)$-Patterson-Sullivan measure $\nu_{\i}$ on $\La_{\i(\theta)}$, 
 we denote by $\m_\psi={\mathsf m}^{\BMS}_{\nu,\nu_{\i}}$ the associated $A_\theta$-invariant Bowen-Margulis-Sullivan measure on $\Omega_{\Ga}$, 
 locally equivalent to the product $\nu\otimes\nu_{\i}\otimes \op{Leb}_{\fa_\theta}$.
Similarly, we denote by $m_\psi$ the associated $\{\phi_t\}$-invariant {\it Bowen-Margulis-Sullivan measure} on $\Omega_\psi$, 
 locally equivalent to the product $\nu\otimes\nu_{\i}\otimes \op{Leb}_{\br}$.
Then $\m_\psi=m_\psi \otimes \op{Leb}_{\ker\psi} .$
As we are not assuming the uniqueness of $\nu$ and $\nu_{\i}$ for a given $\psi$, $\m_\psi$ and $m_\psi$ are not necessarily determined by $\psi$. Nevertheless, it is convenient to refer
to them as BMS measures associated to $\psi$.

\subsection*{Relatively Anosov groups}  
A transverse subgroup $\Ga<G$ is called {\it relatively Anosov} (more precisely relatively $\theta$-Anosov) if $\Ga$ is a relatively hyperbolic group and there exists a $\Ga$-equivariant homeomorphism between the Bowditch boundary of $\Ga$ and the limit set $\La_{\theta}$. When $\Ga$ is hyperbolic, its Bowditch boundary is the Gromov boundary of $\Ga$, and in this case, the relatively Anosov subgroup $\Ga$ is simply an Anosov subgroup.
When $G$ has rank-one, relatively Anosov subgroups coincide with geometrically finite Kleinian groups.
Recall that for a geometrically finite Kleinian group $\Ga$, there exists a unique Patterson-Sullivan measure of dimension equal to the critical exponent $\delta_\Ga$.
In higher-rank, we consider the growth indicator $\psi_\Ga^\theta$ of $\Ga$, a generalization of the critical exponent (see \eqref{eqn.defofGI} for the definition). A linear form $\psi$ is said to be {\it tangent} to $\psi_\Ga$ if $\psi \ge \psi_{\Ga}^{\theta}$ and equality holds at some non-zero $u \in \fa_{\theta} $.
For a relatively Anosov subgroup $\Ga$ and a $(\Ga, \theta)$-proper linear form $\psi\in \fa_\theta^*$ tangent to $\psi_\Ga^\theta$, there exists a unique $(\Ga, \psi)$-Patterson-Sullivan measure on $\La_\theta$, and hence a unique
BMS measure $m_\psi$ associated with $\psi$ (see \cite{LO_invariant}, \cite{sambarino2022report} for Anosov groups and \cite{CZZ_relative} for relatively Anosov groups). 

For Anosov subgroups, the associated base space $\Omega_\psi$ is known to be homeomorphic to the Gromov geodesic flow space and is compact (\cite{CS_local}, \cite{BCLS_gafa}, \cite{sambarino2022report}). In fact, for a transverse subgroup, $\Ga$ is Anosov if and only if
 $\Omega_\psi$ is compact \cite{KOW_indicators}.
In particular, $\Omega_\psi$ is non-compact for relatively Anosov subgroups that are not Anosov. 
Analogous to the classical result on the finitness of the Bowen-Margulis-Sullivan measure for a geometrically finite Kleinian group, we prove the following: 

\begin{theorem}[Finiteness and mixing]  \label{thm.bmsfinite}
    Let $\Ga$ be a relatively Anosov subgroup of $G$.
    For any $(\Ga, \theta)$-proper linear form $\psi \in \fa_{\theta}^*$  tangent to the growth indicator of $\Ga$,  the BMS measure $m_\psi$ is finite:
  $$  
    |m_\psi|<\infty. $$
Moreover, the system $(\Omega_\psi, m_\psi, \phi_t)$ is strongly mixing.
\end{theorem}

In fact, we establish strong mixing in a broader setting of transverse subgroups, which can be regarded as a higher-rank analogue of Babillot's mixing theorem (see Theorem \ref{thm.mixing}).

Given the finiteness of $m_{\psi}$, the metric entropy $h_{m_\psi}(\{\phi_t\})$ of the normalized measure  $m_{\psi}/|m_{\psi}|$  is well-defined.  For a $(\Ga, \theta)$-proper linear form $\psi \in \fa_{\theta}^*$, the associated $\psi$-critical exponent 
 is given by $$\delta_\psi=\limsup_{T\to \infty}\frac{\log \#\{\ga\in \G: \psi(\mu(\ga))<T\}}{T}\in (0, \infty)$$
and one has $\delta_{\psi} = 1$ if and only if $\psi$ is tangent to $\psi_{\Ga}^{\theta}$ (\cite[Theorem 10.1]{CZZ_relative}, \cite[Theorem 4.5]{KOW_indicators}).

\begin{theorem}[Unique measure of maximal entropy] \label{thm.mmemain}
Let $\Ga$ be a relatively Anosov subgroup of $G$.
    For any $(\Ga, \theta)$-proper linear form $\psi \in \fa_{\theta}^*$  tangent to the growth indicator of $\Ga$, 
    $$m_{\psi} \text{ is the unique measure of maximal entropy for $(\Omega_\psi, \{\phi_t\})$}$$
  and the entropy $h_{m_\psi}(\{\phi_t\})$ is equal to  $\delta_{\psi} = 1$.
\end{theorem}

For Anosov subgroups, this theorem is due to Sambarino (\cite{Sambarino_quant}, \cite{sambarino2022report}), as a consequence
of thermodynamic formalism. Our proof, by constrast, does not use the thermodynamic formalism and thus provides an alternative argument even in the Anosov case.

\begin{remark}
    The identity $\delta_{\psi} = 1$ follows from the normalization that $\psi$ is tangent to $\psi_{\Ga}^{\theta}$.  In rank-one, $\phi_t$
    corresponds to the time-changed geodesic flow $g_{t/\delta_\Ga}$ 
    and $m_{\delta_\Ga}$ is the unique measure of maximal entropy for $g_t$, satisfying
      $h_{m_{\delta_\Ga}} (\{g_t\})=\delta_\Ga$. Hence
    $h_{m_{\delta_\Ga}} (\{\phi_t \}) = h_{m_{\delta_\Ga}} (\{g_t\})/\delta_\Ga =1 $.
\end{remark}

A key technical ingredient of Theorems \ref{thm.bmsfinite} and \ref{thm.mmemain} is  the following coarse reparameterization theorem, which is also of independent interest.
 Let $(X_{GM}, d_{GM})$ denote the Groves-Manning cusp space of $\Ga$ and let $\cal G$ denote the space of all parameterized bi-infinite geodesics in the Groves-Manning cusp space \cite{GM_relhyp}. Define the geodesic flow
 $\varphi_s : \cal G \to \cal G$ by $(\varphi_s \sigma)(\cdot) = \sigma(\cdot + s)$.

\begin{theorem}[Reparameterization] \label{thm.repar0} \label{thm.reparflow0}
  There exists a continuous, surjective, proper $\Ga$-equivariant map 
    $$\tilde \Psi : \cal G \to \tilde \Omega_{\psi}$$
together with  a continuous cocycle $\tilde{\mathsf t}:\cal G\times
    \br\to \br$  such that  for all $\sigma\in \cal G$ and $ s\in \br$, 
    \begin{enumerate}
        \item 
   $\tilde \Psi(\varphi_s \sigma) = \phi_{\tilde{\mathsf t}(\sigma, s)} \tilde \Psi(\sigma)$;
\item      $\tilde{\mathsf t}(\sigma, s)= - \tilde{\mathsf t}(\varphi_s \sigma, -s)$;
 \item  there exists an absolute constant  $B > 0$ such that  $$a|s| - B \le \tilde{\mathsf t} (\sigma, |s|)  \le  a'|s| + B$$
    where 
  $$0< a:= \liminf_{\ga \in \Ga} \frac{\psi(\mu(\ga))}{d_{GM}(e, \ga)}\;\; \text{ and } \;\; a' := 3\limsup_{\ga \in \Ga} \frac{\psi(\mu(\ga))}{d_{GM}(e, \ga)}<\infty ;$$
  \item  all fibers $\{\sigma(0)\in X_{GM}: \sigma\in \tilde \Psi^{-1}(x)\} $, $x \in \tilde \Omega_{\psi}$, have uniformly bounded diameter.
      \end{enumerate}
     Moreover, the flow $\phi_t$ is  exponentially expanding along unstable foliations
 of $\tilde \Omega_{\psi} =\La_\theta^{(2)}\times \br $, as described in Theorem \ref{thm.nicemetrics}.
 \end{theorem}

The map $\Psi : \Ga \ba \cal G \to \Omega_{\psi}$, induced from $\tilde \Psi$,
provides a thick-thin decomposition of $\Omega_{\psi}$ that plays a crucial role
in the proof of the finiteness
of $m_\psi$ (Theorem \ref{thm.bmsfinite}). This decomposition is used in conjunction with the work of Canary-Zhang-Zimmer \cite{CZZ_relative}, which analyzes
the critical exponents of peripheral subgroups of $\Ga$.  The exponentially expanding property of $\phi_t$
is essential in constructing a measurable partition of $\tilde \Omega_{\psi}$ subordinated to unstable foliations (Proposition \ref{prop.prop1}), a key step
in the proof of Theorem \ref{thm.mmemain} concerning the uniqueness of the measure of maximal entropy.

\begin{Rmk} 
Recently, Blayac-Canary-Zhu-Zimmer \cite{BCZZ} showed that for $\theta$-transverse $\Ga$ and $\psi \in \fa_{\theta}^*$, if there exists a $(\Ga, \theta)$-Patterson-Sullivan measure on $\La_{\theta}$, then $\psi$ must be $(\Ga, \theta)$-proper. This result implies that
the $(\Ga, \theta)$-properness condition is not a genuinely restrictive assumption when studying dynamics associated to Bowen-Margulis-Sullivan measures.
\end{Rmk}

\subsection*{Acknowledgements}
We were informed that the ongoing work of Blayac, Canary, Zhu and Zimmer \cite{BCZZ} contains a different proof of Theorem \ref{thm.bmsfinite}.

\section{Preliminaries}
We review some basic facts about Lie groups, following \cite[Section. 2] {KOW_indicators} which we refer for more details.
Throughout the paper, let $G$ be a connected semisimple real algebraic group. 
 Let $P<G$ be a minimal parabolic subgroup with a fixed Langlands decomposition $P=MAN$ where $A$ is a maximal real split torus of $G$, $M$ is the maximal compact subgroup of $P$ commuting with $A$ and $N$ is the unipotent radical of $P$.
Let $\fg$ and $\fa$ respectively denote the Lie algebras of $G$
and $A$. Fix a positive Weyl chamber $\fa^+<\fa$ so that
$\log N$ consists of positive root subspaces and
set $A^+=\exp \fa^+$. We fix a maximal compact subgroup $K< G$ such that the Cartan decomposition $G=K A^+ K$ holds. We denote by 
$\mu : G \to \fa^+$ the Cartan projection defined by the condition $g\in K\exp \mu(g) K$ for $g \in G$. Let $X = G/K$ be the associated Riemannian symmetric space and $o=[K]\in X$.  Fix a $K$-invariant norm $\| \cdot \|$ on $\fg$. This induces the left $G$-invariant Riemannian metric $d$ on $X$.

Let $\Phi=\Phi(\fg, \fa)$ denote the set of all roots,  $\Phi^{+}\subset \Phi$
the set of all positive roots, and $\Pi \subset \Phi^+$  the set of all simple roots.
Fix a Weyl element $w_0\in K$ of order $2$ in the normalizer of $A$ representing the longest Weyl element so that $\op{Ad}_{w_0}\mathfrak a^+= -\mathfrak a^+$. 
The map $$\i= -\op{Ad}_{w_0}:\fa\to \fa$$ is called the opposition involution.
It induces an involution $\Phi\to \Phi$ preserving $\Pi$, for which we use the same notation $\i$, such that $\i (\alpha ) \circ  \op{Ad}_{w_0}  =-\alpha $ for all $\alpha\in \Phi$.

Henceforth, we fix a  non-empty subset $\theta\subset \Pi$. 
Let $ P_\theta$ denote a standard parabolic subgroup of $G$ corresponding to $\theta$; that is,
$P_{\theta}$ is generated by $MA$ and all root subgroups $U_\alpha$,
where $\alpha$ ranges over all positive roots which are not $\mathbb Z$-linear combinations of $\Pi-\theta$. Hence $P_\Pi=P$. Let 
\begin{equation*}
\mathfrak{a}_\theta =\bigcap_{\alpha \in \Pi-\theta} \ker \alpha, \quad  \quad \fa_\theta^+  =\fa_\theta\cap \fa^+, \end{equation*}
\begin{equation*} A_{\theta}  = \exp \fa_{\theta}, \quad \text{and} \quad     A_{\theta}^+  = \exp \fa_{\theta}^+. \end{equation*}
Let $ p_\theta:\mathfrak{a}\to\mathfrak{a}_\theta$ denote  the projection invariant under all Weyl elements fixing $\fa_\theta$ pointwise. We write $\mu_{\theta} := p_{\theta} \circ \mu :G\to \fa_\theta^+.$
The space $\fa_\theta^*=\op{Hom}(\fa_\theta, \br)$ can be identified with the subspace of $\fa^*$ which is $p_\theta$-invariant: $\fa_\theta^*=\{\psi\in \fa^*: \psi\circ p_\theta=\psi\}$.
We have the Levi-decomposition
 $P_\theta=L_\theta N_\theta$ where  $L_\theta$ is the centralizer of $A_{\theta}$ and $N_\theta=R_u(P_\theta)$ is the unipotent radical of $P_\theta$. We set $M_{\theta} = K \cap P_{\theta}\subset L_\theta$.

\subsection*{Limit set $\La_\theta$} 
We set $$\F_\theta=G/P_{\theta} .$$
The subgroup $K$ acts transitively on $\F_\theta$, and hence
 $\F_\theta\simeq K/ M_\theta.$ 
\begin{definition} \label{fc} For a sequence $g_i\in G$  and $\xi\in \ft$, we write $\lim_{i\to \infty} g_i =\xi$ and
 say $g_i $ {\it converges} to $\xi$ if \begin{itemize}
     \item for each $\alpha\in \theta$, $\alpha(\mu(g_i)) \to \infty$ as $g_i\to \infty$; 
\item $\lim_{i\to\infty} \kappa_{i}\xi_\theta= \xi$ in $\F_\theta$ for some $\kappa_{i}\in K$ such that $g_i\in \kappa_{i} A^+ K$.
 \end{itemize}         
\end{definition}

 The {\it $\theta$-limit set} of a discrete subgroup $\Ga$ can be defined as follows:
\be \label{def.limitset} \lat=\lat(\Ga):=\{\lim {\ga}_i\in \F_\theta: {\ga}_i\in \Ga\}\ee  where $\lim \ga_i$ is defined as in Definition \ref{fc}.
If $\Ga$ is Zariski dense, this is the unique $\Ga$-minimal subset of $\F_{\theta}$ (\cite{Benoist1997proprietes}, \cite{Quint2002Mesures}).

\subsection*{Jordan projections} 

	Any $g\in G$ can be written as the commuting product $g=g_hg_e g_u$ where $g_h$ is hyperbolic, $g_e$ is elliptic and $g_u$ is unipotent. 
	The hyperbolic component $g_h$ is conjugate to a unique element $\exp \lambda(g) \in A^+$ and $\lambda(g)$ is called 
	the {\it Jordan projection} of $g$. We write $\la_{\theta} := p_{\theta} \circ \la$.

\begin{theorem}\cite{Benoist_properties2} \label{thm.Benoistdense}
   For any Zariski dense subgroup $\Ga<G$, the subgroup generated by
   $\{\lambda(\ga)\in \fa^+:\ga\in \Ga\}$ is dense in $\fa$.
   \end{theorem}

\subsection*{Busemann map and Gromov product}
The {\it $\frak a$-valued Busemann map} $\beta: \cal F_\Pi \times G \times G \to\frak a $ is defined as follows: for $\xi\in \cal F$ and $g, h\in G$,
$$  \beta_\xi ( g, h):=\sigma (g^{-1}, \xi)-\sigma(h^{-1}, \xi)$$
where  $\sigma(g^{-1},\xi)\in \fa$ 
is the unique element such that $g^{-1}k \in K \exp (\sigma(g^{-1}, \xi)) N$ for any $k\in K$ with $\xi=kP$.
For $(\xi,g,h)\in \cal F_\theta\times G\times G$, we define
 \be\label{Bu} \beta_{\xi}^\theta (g, h): = 
p_\theta ( \beta_{\xi_0} (g, h)) \ee 
for any $\xi_0\in \F_\Pi$ projecting to $\xi$.
This is well-defined independent of the choice of $\xi_0$ 
\cite[Lemma 6.1]{Quint2002Mesures}. Moreover, since product map $ K\times A \times N\to G$ is a diffeomorphism, Busemann maps are continuous.

Two points $\xi \in \F_{\theta}$ and $\eta \in \F_{\i(\theta)}$ are said to be {\it in general position} if 
\be \label{gp} \xi=gP_{\theta}\text{ and } \eta = gw_0 P_{\i(\theta)} \text{ for some $g\in G$}.\ee 
We set
\be\label{fgp} \F_\theta^{(2)}=\{(\xi,\eta)\in \F_\theta\times \F_{\i(\theta)}:  \text{$\xi, \eta$ are in general position} \}\ee 
which is the unique open $G$-orbit in $\F_\theta\times \F_{\i(\theta)}$ under the diagonal $G$-action.

For $(\xi, \eta) \in \F_{\theta}^{(2)}$, we define the {\it $\fa_{\theta}$-valued Gromov product } as 
\be \label{eqn.gromovproductdef}
\langle \xi, \eta \rangle = \beta_{\xi}^{\theta}(e, g) + \i (\beta_{\eta}^{\i(\theta)}(e, g))
\ee
where $g \in G$ satisfies $(gP_\theta, gw_0P_{\i(\theta)}) = (\xi, \eta)$. This does not depend on the choice of $g$ \cite[Lemma 9.13]{KOW_indicators}.

\subsection*{Patterson-Sullivan measures}
For $\psi\in \fa_\theta^*$, a {\it $(\Gamma, \psi)$-conformal measure} is a Borel probability measure on $\F_\theta$ such that 
\be \label{eqn.psmeas}
\frac{d \gamma_*\nu}{d\nu}(\xi)=e^{\psi(\beta_\xi^\theta(e,\gamma))} \quad \text{for all $\gamma \in \Gamma$ and $ \xi \in \F_\theta$}
\ee
 where ${\ga}_* \nu(D) = \nu(\ga^{-1}D)$ for any Borel subset $D\subset \F_\theta$ and $\beta_\xi^\theta$ denotes the $\fa_\theta$-valued Busemann map defined in \eqref{Bu}. A $(\Ga, \psi)$-conformal measure supported on $\La_\theta$ is called a {\it $(\Ga, \psi)$-Patterson Sullivan measure}.

\subsection*{Growth indicator} Let $\Ga<G$ be a $\theta$-discrete subgroup, that is,  $\mu_\theta|_\Ga$ is a proper map.
The {\it $\theta$-growth indicator} $\psi_\Ga^{\theta}:\fa_\theta\to [-\infty, \infty) $ is a higher-rank version of the critical exponent, which is defined as follows: If $u \in \fa_\theta$ is non-zero,
\be \label{eqn.defofGI}
\psi_\Ga^{\theta}(u)=\|u\| \inf_{u\in \cal C}
\tau^\theta_{\mathcal C}\ee 
where $ \tau^{\theta}_{\cal C}$ is the abscissa of convergence of the series $\sum_{\ga\in \Ga, \mu_\theta(\ga)\in \mathcal C} e^{-s\|\mu_\theta(\ga)\|}$ and $\cal C\subset \fa_\theta$ ranges over all open cones containing $u$. Set $\psi_{\Ga}^{\theta}(0) = 0$. This definition was given in \cite{KOW_indicators}, extending Quint's growth indicator  \cite{Quint2002divergence}  to a general $\theta$. 

For $\Ga$ transverse and $\psi$ $(\Ga, \theta)$-proper, it is proved in \cite{KOW_indicators} that if there exists a $(\Ga, \psi)$-conformal measure on $\F_\theta$, then $$\psi \ge \psi_{\Ga}^{\theta}.$$
We say that $\psi\in \fa_\theta^*$  is {\it tangent} to $\psi_{\Ga}^{\theta}$ if $\psi\ge \psi_\Ga^\theta$ and $\psi(u) =\psi_{\Ga}^{\theta}(u)$ for some $u \in \fa_{\theta} - \{0\}$. 
In the rank-one case, if 
$\delta_\Ga$ is the critical exponent of the Poincar\'e series $\sum_{\ga\in \Ga } e^{-s d(o, \ga o)}$ and 
$v\in \fa^+$
is the unique vector with $d(o, \exp v o)=1$, then $\psi_\Ga^{\Pi}$ on $\fa^+=\br_+v$ is given by $\psi_\Ga^{\Pi} (tv)=\delta_\Ga t$. As $\psi_{\Ga}^{\Pi}$ itself is the restriction of a linear form to $\fa^+$, it is the unique linear form tangent to itself.
In higher-rank, $\psi_\Ga^{\theta}$ is typically non-linear but concave and
there are abundant tangent linear forms in general.
As in the rank-one setting, interesting geometry and dynamics occur for tangent linear forms.

\section{Vector bundle structure of the non-wandering set $\Omega_\Ga$} \label{sec.bundle}
We fix a non-empty subset $\theta$ of $\Pi$.
In this section, we assume that $\Ga<G$ is a non-elementary {\it $\theta$-transverse} subgroup, that is, $\Ga$ satisfies
\begin{itemize}
\item (non-elementary): $\#\La_\theta \ge 3$;
    \item  (regularity):
$ \liminf_{\ga\in \Ga} \alpha(\mu({\ga}))=\infty $ for all $\alpha\in \theta$; and
\item  (antipodality): any two distinct $\xi, \eta\in \La_{\theta\cup \i(\theta)}$ 
are in general position as in \eqref{gp}.
\end{itemize}

We will define a locally compact Hausdorff space $\Omega_\Ga$ which is the non-wandering set for the action of $A_\theta$. Recall that the centralizer of $A_\theta$ is the direct product $A_\theta S_\theta$ where $S_\theta$ is a compact central extension of a connected semisimple real algebraic subgroup. Note that
$S_\theta$ is compact if and only if $\theta=\Pi$.

The homogeneous space $G/S_{\theta}$ can be identified with the space $\F_{\theta}^{(2)} \times \fa_{\theta}$ via the map $$g S_{\theta} \mapsto (gP_\theta, gw_0P_{\i(\theta)}, \beta_{gP_{\theta}}^{\theta}(e, g)),$$
recalling that $w_0 \in K$ is the longest Weyl element,
    and the left $G$-action on $\F_{\theta}^{(2)} \times \fa_{\theta}$ given by $$g(\xi, \eta, v) = (g \xi, g \eta, v + \beta_{\xi}^{\theta}(g^{-1}, e))$$ makes the above identification $G$-equivariant. Since $S_{\theta}$ commutes with $A_{\theta}$, the diagonal subgroup $A_{\theta}$ acts on $G/S_{\theta}$ on the right, and this action is conjugate to the action of $\fa_{\theta}$ on $\F_{\theta}^{(2)} \times \fa_{\theta}$ by the translation on the last component.
Since $S_\theta$ is not compact in general, the action of $\Ga$ on $\F_\theta^{(2)}\times \fa_\theta$ is not properly discontinuous. However for $\Ga$ transverse,
the $\Ga$-action restricted to the subspace $\tilde{\Omega}_{\Ga}:= \La_\theta^{(2)}\times \fa_\theta $ turns out to be properly discontinuous where $\La_\theta^{(2)}=\F_\theta^{(2)}\cap (\La_\theta\times \La_{\i(\theta)})$
\cite[Theorem 9.1]{KOW_indicators}.
Hence we obtain the locally compact second countable Hausdorff space $$\Omega_{\Ga}:=\Ga\ba \tilde \Omega_{\Ga},$$
which is the non-wandering set for the right $A_\theta$-action.
  
For each $(\Ga, \theta)$-proper form $\psi \in \fa_{\theta}^*$, $\Omega_\Ga$ admits a $\ker\psi$-bundle structure over a non-wandering set $\Omega_\psi$ for a one-dimensional flow. More precisely, 
\begin{theorem}\cite[Theorem 9.2]{KOW_indicators}\label{op2}
The $\Ga$-action on the space $\tilde \Omega_{\psi}:=\La_\theta^{(2)}\times \br$ given by $$\ga (\xi, \eta, s) = (\ga \xi, \ga \eta, s + \psi(\beta_{\xi}^{\theta}(\ga^{-1}, e)))$$
     is properly discontinuous.
     Thus the space \be\label{op} \Omega_{\psi} := \Ga \ba \tilde \Omega_{\psi} =\Ga\ba (\La_\theta^{(2)}\times \br) \ee  is a locally compact second countable Hausdorff space equipped with the translation flow 
     $\{\phi_t\}$ on the $\br$-component.
     \end{theorem}

\begin{Rmk} \label{rmk.exampleproper} Any linear form which is positive on $\fa^+\cap \fa_\theta-\{0\}$, e.g., any non-negative linear combination of the fundamental weights $\omega_\alpha$, $\alpha\in \theta$, is $(\Ga, \theta)$-proper. On the other hand, a linear form which takes negative values on some part of the $\theta$-limit cone is never $(\Ga, \theta)$-proper (see \cite{KOW_indicators}).
\end{Rmk}
     Explicitly, the {\it translation flow} $\{\phi_t\}$ is defined as follows: for $t \in \R$ and $(\xi, \eta, s) \in \tilde{\Omega}_{\psi}$, $$\phi_t(\xi, \eta, s) = (\xi, \eta, s + t).$$
     This flow $\{\phi_t\}$ on $\tilde \Omega_{\psi}$ commutes with the $\Ga$-action, and hence induces the one-dimensional flow on $\Omega_{\psi}$ which we also denote by $\phi_t$ by abusing notations.

Consider the projection $\Omega_{\Ga} \to \Omega_{\psi}$  induced by the $\Ga$-equivariant projection
$\tilde{\Omega}_{\Ga} \to \tilde \Omega_{\psi}$ given by $(\xi, \eta, v)\mapsto (\xi, \eta, \psi(v))$. This is a principal $\ker\psi$-bundle, which
is trivial since $\ker\psi$ is a vector space. It follows that there exists a $\ker\psi$-equivariant homeomorphism
between $\Omega_\Ga$ and $\Omega_\psi \times \ker \psi$.
    $$\begin{tikzcd}[column sep = tiny]
    {\color{white} \simeq \Omega_{\psi} \times \ker \psi} \hspace{-1.3em}  & \Omega_{\Ga} \arrow[d]  & \hspace{-1.3em} \simeq \Omega_{\psi} \times \ker \psi \\
    & \Omega_{\psi} &
\end{tikzcd}$$

Let $\nu$ and $\nu_{\i}$ be a pair of $(\Ga, \psi)$ and $(\Ga, \psi \circ \i)$-Patterson-Sullivan measures on $\La_{\theta}$ and $\La_{\i(\theta)}$ respectively.
The Bowen-Margulis-Sullivan measure $\m_{\psi}$ on $\Omega_{\Ga}$ associated with the pair $(\nu, \nu_{\i})$ is the $A_\theta$-invariant measure induced by the $\Ga$-invariant measure  $d \tilde{\mathsf{m}}^{\BMS}_{\psi}(\xi, \eta, v) := e^{\psi(\langle \xi, \eta \rangle)} d \nu(\xi) d \nu_{\i}(\eta) d \Leb_{\fa_\theta}$
on $\tilde{\Omega}_{\Ga}$, where $\langle \cdot, \cdot \rangle$ denotes the Gromov product \eqref{eqn.gromovproductdef} and $d \Leb_{\fa_\theta}$ denotes the Lebesgue measure on $\fa_{\theta}$. 

We also have a $\{\phi_t\}$-invariant Radon measure $m_\psi$
on $\Omega_\psi$ induced by the $\Ga$-invariant measure \be\label{mp} d \tilde m_{\psi}(\xi, \eta, s) := e^{\psi(\langle \xi, \eta \rangle)} d \nu(\xi) d \nu_{\i}(\eta) ds\ee 
on $\tilde \Omega_{\psi}$ where $ds$ denotes the Lebesgue measure on $\br$. The measure $m_{\psi}$ is also referred to as {\it Bowen-Margulis-Sullivan measure} on $\Omega_{\psi}$ associated with the pair $(\nu, \nu_{\i})$. By the $\ker\psi$-equivariant homeomorphism
 $\Omega_\Ga\simeq \Omega_\psi \times \ker \psi$, $\m_\psi$ disintegrates over the measure $m_\psi$ with conditional measure being the Lebesgue measure $\Leb_{\ker \psi}$ so that
$$\m_{\psi} = m_{\psi} \otimes \Leb_{\ker \psi} .$$

\section{Strong mixing for transverse groups with finite BMS measure}

Let $\Ga < G$ be a non-elementary $\theta$-transverse subgroup.
Fix a $(\Ga, \theta)$-proper form $\psi \in \fa_{\theta}^*$ and a pair $(\nu, \nu_{\i})$ of $(\Ga, \psi)$ and $(\Ga, \psi \circ \i)$-Patterson-Sullivan measures on $\La_{\theta}$ and $\La_{\i(\theta)}$ respectively. Let $\Omega_\psi$ be as in Theorem \ref{op2} and
  $m_\psi=m_\psi(\nu, \nu_{\i})$ denote a BMS measure on $\Omega_\psi$ associated to a pair $(\nu, \nu_{\i})$.

  This section is devoted to the proof of the following:
\begin{theorem} \label{thm.mixing}
     If $|m_\psi|<\infty$, then $(\Omega_{\psi}, m_{\psi}, \phi_t)$ is strongly mixing. That is, for any $f_1, f_2\in L^2(\Omega_\psi, m_{\psi})$,
   $$\lim_{|t|\to \infty}
   \int f_1 ( \phi_t (x)) f_2 (x) \ dm_{\psi} (x) = \frac{1}{|m_{\psi}|} \int f_1 \ d m_{\psi} \int f_2 \ d m_{\psi} .$$
\end{theorem}

We begin by observing the ergodicity of $m_\psi$:

\begin{theorem} \label{thm.ergodic}
   If $|m_\psi|<\infty$, then $(\Omega_{\psi}, m_{\psi}, \phi_t)$ is ergodic.
\end{theorem}

\begin{proof}
    By the Poincar\'e recurrence theorem, the dynamical system $(\Omega_{\psi}, m_{\psi}, \phi_t)$ is conservative. Hence it follows from the higher-rank Hopf-Tsuji-Sullivan dichotomy \cite[Theorem 10.2]{KOW_indicators} that $(\Omega_{\psi}, m_{\psi}, \phi_t)$ is ergodic.
\end{proof}
Although the flow space $\Omega_\psi$ was not considered,
Theorem \ref{thm.ergodic} can also be deduced from \cite{CZZ_transverse} once $\Omega_{\psi}$ is shown to make sense.
See also \cite{LO1} and \cite{sambarino2022report} for Anosov cases.
\subsection*{$\theta$-transitivity subgroups}

For $g\in G$, we set
$g^+ := gP_{\theta} \in \F_{\theta}$ and $g^- := g w_0 P_{\i(\theta)} \in \F_{\i(\theta)}$.
Set $N_\theta^+=w_0 N_{\i(\theta)} w_0^{-1}$.
We use the following notion of $\theta$-transitivity subgroup:
\begin{definition}
    For $g \in G$ with  $(g^+, g^-) \in \La_{\theta}^{(2)}$, we define the subset $\cal H_{\Ga}^{\theta}(g) $ of $ A_{\theta}$ as follows: for $a\in A_\theta$,
    $a \in \cal H_{\Ga}^{\theta}(g)$ if and only if there exist $\ga \in \Ga$, $s \in S_{\theta}$ and a sequence $n_1, \cdots, n_k \in N_{\theta} \cup N_{\theta}^+$, such that

    \begin{enumerate}
    \item $((gn_1 \cdots n_r)^+ , (gn_1 \cdots n_r)^-) \in \La_{\theta}^{(2)}$ for all $1 \le r \le k$; and
    \item $g n_1 \cdots n_k = \ga gas$.
    \end{enumerate}
    It is not hard to see that $\cal H_{\Ga}^{\theta}(g)$ is a subgroup (cf. \cite[Lemma 3.1]{Winter_mixing}). We call $\cal H_{\Ga}^{\theta}$ the {\it $\theta$-transitivity subgroup} for $\Ga$.
   
\end{definition}

In the following, we prove that the $\theta$-transitivity subgroup $\cal H_{\Ga}^{\theta}$ contains  $\exp \la_{\theta}(\Ga_0)$ for some Zariski dense Schottky subgroup $\Ga_0 < \Ga$.

\begin{proposition} \label{prop.densetranse}
    For any $g \in G$ such that $(g^+, g^-) \in \La_{\theta}^{(2)}$, the subgroup $\psi(\log \cal H_{\Ga}^{\theta}(g))$ is dense in $\R$.
\end{proposition}

\begin{proof}
It was shown in \cite[Proposition 8.3]{KOW_ergodic} that if $\Ga$ is a Zariski dense $\theta$-transverse subgroup and if $g \in G$ is such that $(g^+, g^-) \in \La_{\theta}^{(2)}$, then the subgroup $\cal H_{\Ga}^{\theta}(g)$ is dense in $A_{\theta}$, by proving that for a Schottky subgroup $\Ga_0 < \Ga$, the set of Jordan projections $\la_{\theta}(\Ga_0)$ is contained in $\log \cal H_{\Ga}^{\theta}(g)$. The Zariski dense hypothesis was used to guarantee that $\Ga_0$ can be taken to be Zariski dense, and hence $\la_{\theta}(\Ga_0)$ generates a dense subgroup in $\fa_{\theta}$ (\cite{Benoist_properties2}, Theorem \ref{thm.Benoistdense}).

In general, let $H$ be the Zariski closure of $\Ga$ and
consider the Levi decomposition of $H$: $H=LU$ where $L$ is a reductive algebraic subgroup and $U$ the unipotent radical of $H$. Moreover, we have a Cartan decomposition $G=KA^+K$ so that $L=(K\cap L) (A^+\cap L) (K\cap L)$ 
by \cite{Mostow_reductive}.
If $\pi:H\to L$ denotes the projection, then $\pi(\Ga)$ is Zariski dense in $L$ and hence
its Jordan projection generates a dense subgroup of $\fa\cap \op{Lie} L$.
 This allows the same proof of \cite[Proposition 8.3]{KOW_ergodic} to work within $L$, and hence the claim follows.
\end{proof}

\subsection*{Contractions by flow on $\Omega_{\psi}$}
For $g \in G$, we write
$$[g] := (g^+, g^-, \psi( \beta_{g^+}^{\theta}(e, g))) \in \F_{\theta}^{(2)} \times \R.$$
We mainly consider the case when $[g] \in \tilde{\Omega}_{\psi} = \La_{\theta}^{(2)} \times \R$, that is, when $(g^+, g^-) \in \La_{\theta}^{(2)}$. For $[g] \in \tilde{\Omega}_{\psi}$, we denote by $\Ga[g] \in \Omega_{\psi}$ the element of $\Omega_{\psi}$ obtained as the projection of $[g]$ by $\tilde{\Omega}_{\psi} \to \Omega_{\psi}$.

We set for $g \in G$ such that $[g] \in \tilde{\Omega}_{\psi}$, 
\be \label{ww}
\begin{aligned}
    \tilde W^+ ([g]) & := \{ [g n] \in \tilde \Omega_{\psi} : n \in N_{\theta}^+ \}; \\
    \tilde W^- ([g]) & := \{ [g n] \in \tilde \Omega_{\psi} : n \in N_{\theta} \}.
\end{aligned}
\ee
The elements of $\tilde{W}^{\pm}([g])$ can be described as follows:
\begin{lemma}\cite[Lemma 8.4]{KOW_ergodic} \label{lem.transbyhoro}
    Let $g \in G$, $n \in N_{\theta}^+$, and  $n' \in N_{\theta}$. Then $$\begin{aligned}
        [g n] & = \left((gn)^+, g^-, \psi\left(\beta_{g^+}^{\theta}(e, g) + \langle(g n)^+, g^-\rangle - \langle(g^+, g^-)\rangle\right)\right); \\
        [gn'] & = \left(g^+, (gn')^-, \psi\left(\beta_{g^+}^{\theta}(e, g)\right)\right). \\
    \end{aligned}$$
\end{lemma}
These are leaves of foliations $\tilde W^{\pm} := \{ \tilde W^+([g]) : [g] \in \tilde \Omega_{\psi}\}$.
For $z\in \Omega_\psi$,
we set 
\be \label{ww2} W^+(z) := \Ga \ba \tilde W^+([g]), \quad \text{and} \quad W^-(z) :=  \Ga \ba \tilde W^-([g])\ee 
where $g\in G$ is such that $\Gamma [g]=z$.
The following proposition says that we may consider
$W^+ := \{ W^+(z): z\in \Omega_\psi\}$ and $W^- := \{W^-(z):z\in \Omega_\psi\}
$ as {\it unstable and stable foliations } for the flow $\phi_t$ in $\Omega_\psi$: note that since $\Omega_\psi$ is a locally compact second countable Hausdorff space by Theorem \ref{op2}, so is its one-point compactification $\Omega_\psi^*$. Hence $\Omega_\psi^*$ is metrizable. Therefore, we can choose a metric $\d$ on $\Omega_{\psi}$ which is a restriction of a metric on $\Omega_\psi^*$. That we can use this kind of metric $\d$ to prove the following proposition was first observed in \cite{BCZZ}.

\begin{proposition} \cite[Proposition 8.6]{KOW_ergodic} \label{prop.admissiblemetric} 
    Let $ z \in \Omega_{\psi}$. We have
    \begin{enumerate}
     \item if $x,y\in  W^+ (z)$, then $$\d(\phi_{-t}  (x), \phi_{-t} ( y) )\to 0 \quad \text{as } t \to +\infty.$$
     
        \item if $x,y\in    W^- (z)$, then
        $$\d(\phi_t (x ), \phi_t (y)) \to 0 \quad \text{as } t \to +\infty.$$
    \end{enumerate}
    Moreover, the convergence is uniform on compact subsets.
\end{proposition}

\subsection*{Proof of Theorem \ref{thm.mixing}}

We are now ready to prove the strong mixing. We recall the following lemma proved by Babillot:

\begin{lemma} \cite[Lemma 1]{Babillot_mixing} \label{lem.babillot1}
    Let $(\cal X, m, \{T_t\}_{t \in \R})$ be a probability measure-preserving system. Let $f \in L^2(\cal X, m)$ be such that $\int f dm = 0$. Suppose that $f \circ T_{t_i} \not\to 0$ weakly\footnote{$f_n \to 0$ weakly if and only if $\int f_n g \,dm \to 0$ for all $g \in L^2(\cal X, m)$} for some $t_i \to \infty$. Then there exists a non-constant function $F$ such that by passing to a subsequence, $$f \circ T_{t_i} \to F \quad \mbox{and} \quad f \circ T_{-t_i} \to F \quad \mbox{weakly} \quad \text{as } i \to \infty.$$
\end{lemma}

The following is an easy observation in measure theory:

\begin{lemma} \label{lem.babillot2}
  Let $(\cal X, m)$ be a probability measure space.  If $f_i \to F$ weakly in $L^2(\cal X, m)$, then there exists a subsequence $f_{i_j}$ such that the Cesaro average converges: $${1 \over \ell^2} \sum_{j = 1}^{\ell^2} f_{i_j} \to F \quad m\text{-a.e.}$$
\end{lemma}

Now going back to our setting, let $f_1, f_2 \in L^2(\Omega_{\psi}, m_{\psi})$. We may assume that $m_{\psi}$ is a probability measure. By replacing $f_1$ with $f_1 - \int f_1 d m_{\psi}$, it suffices to show that for any $f \in L^2(\Omega_{\psi}, m_{\psi})$ with $\int f dm_{\psi} = 0$, we have $f \circ \phi_t \to 0$ weakly as $|t|\to \infty$.
Since $C_c(\Omega_{\psi}) $ is dense in $ L^2(\Omega_{\psi}, m_{\psi})$, we may assume without loss of generality that $f$ is a  continuous function  with compact support on $\Omega_{\psi}$. Suppose that $f \circ \phi_t \not\to 0$ weakly as $t \to \infty$. By Lemma \ref{lem.babillot1} and Lemma \ref{lem.babillot2}, there exists a non-constant function $F : \Omega_{\psi} \to \R$ and a subsequence $t_i \to \infty$ such that \be \label{eqn.provemixing}
    {1 \over \ell^2} \sum_{i = 1}^{\ell^2} f \circ \phi_{t_i} \to F \quad \mbox{and} \quad {1 \over \ell^2} \sum_{i = 1}^{\ell^2} f \circ \phi_{-t_i} \to F \quad m_{\psi}\text{-a.e. as } \ell \to \infty.
    \ee We claim that $F$ is invariant under the flow $\phi_t$; this  yields a contradiction to the ergodicity of $(\Omega_{\psi}, m_{\psi}, \phi_t)$ obtained in Theorem \ref{thm.ergodic}.

Let $W_0 = \{x \in \Omega_{\psi} : \eqref{eqn.provemixing} \mbox{ holds}\}$, which is $m_{\psi}$-conull. Since $f$ is uniformly continuous, it follows from Proposition \ref{prop.admissiblemetric} that if 
$g \in G$ and $n \in N_{\theta} \cup N_{\theta}^+$ are such that $[g], [gn] \in \tilde{\Omega}_{\psi}$ and $\Ga[g], \Ga[gn] \in W_0$, then $$F(\Ga [g]) = F(\Ga [gn]).$$ 
Denote by $\tilde{W}_0$ and $\tilde{F}$ the $\Ga$-invariant lifts of $W_0$ and $F$ to $\tilde{\Omega}_{\psi}$ respectively.  We set $$W_1 := \{(\xi, \eta) : (\xi, \eta, t) \in \tilde{W}_0 \text{ for } \Leb\text{-a.e. } t\}.$$ We also set $$W = \{(\xi, \eta) \in W_1 : (\xi, \eta'), (\xi', \eta) \in W_1 \text{ for } \nu\text{-a.e. } \xi' \text{ and } \nu_{\i}\text{-a.e. } \eta'\}.$$ 

Recall that we also denote by $\{\phi_t\}$ the translation flow on $\tilde \Omega_{\psi}$.
For any $\varepsilon  > 0$ and $x \in \tilde{\Omega}_{\psi}$, let $$F_{\varepsilon}(x) := {1 \over \varepsilon} \int_{-\varepsilon}^{\varepsilon} \tilde{F}( \phi_s(x)) \ ds.$$ Then $F_{\varepsilon}$ is continuous on each $\{\phi_t\}$-orbit and as $\varepsilon \to 0$, we have the convergence $F_{\varepsilon} \to \tilde{F}$ $m_{\psi}$-a.e. Hence it suffices to show that $F_{\varepsilon}$ is invariant under the flow $ \phi_t$.

By the definition of $W$ and the observation on $W_0$ made above, we have that if $g \in G$ and $n \in N_{\theta} \cup N_{\theta}^+$ are such that $[g], [gn] \in W \times \R \subset \tilde{\Omega}_{\psi}$, then $F_{\varepsilon}([g]) = F_{\varepsilon}([gn])$. Fix $g \in G$ such that $[g] \in W \times \R$ and let $t_0 \in \psi(\log \cal H_{\Ga}^{\theta}(g))$ and $a \in \cal H_{\Ga}^{\theta}(g)$ such that $\psi(\log a) = t_0$. We then have $\phi_{t_0}([g]) = [ga]$. By the definition of the $\theta$-transitivity subgroup, there exist $\ga \in \Ga$, $s \in S_{\theta}$, and a sequence $n_1, \cdots, n_k \in N_{\theta} \cup N_{\theta}^+$, such that
    \begin{enumerate}
    \item $((gn_1 \cdots n_r)^+ , (gn_1 \cdots n_r)^-) \in \La_{\theta}^{(2)}$ for all $1 \le r \le k$;
    \item $g n_1 \cdots n_k = \ga gas$.
    \end{enumerate}

    As in the proof of \cite[Proposition 8.8]{KOW_ergodic}, there exist a sequence $a_j \in A_{\theta}$ and a sequence of $k$-tuples $(n_{1, j}, \cdots, n_{k, j}) \in \prod_{i = 1}^k N_{\theta} \cup N_{\theta}^+$ converging to $a$ and $(n_1, \cdots, n_k)$ respectively as $j \to \infty$, and such that for each $j \ge 1$, we have $$[gn_{1, j} \cdots n_{r, j}] \in W \times \R \quad \text{for all } 1 \le r \le k \quad \text{and} \quad [g n_{1, j} \cdots n_{k, j}] = [\ga ga_j].$$
    Therefore, we have for each $j \ge 1$ that 
    $$\begin{aligned}
    F_{\varepsilon}([g]) & = F_{\varepsilon}([gn_{1, j}]) = \cdots = F_{\varepsilon}([gn_{1, j} \cdots n_{k-1, j}]) = F_{\varepsilon}([g n_{1, j} \cdots n_{k, j}]) \\
    & = F_{\varepsilon}([\ga g a_j]) = F_{\varepsilon}([ga_j]).
    \end{aligned}$$
    Taking the limit $j \to \infty$, it follows from the continuity of $F_{\varepsilon}$ on each $\{\phi_t\}$-orbit that $$F_{\varepsilon}([g]) = F_{\varepsilon}([ga]) = (F_{\varepsilon} \circ  \phi_{t_0})([g]).$$
    Since $\psi( \log \cal H_{\Ga}^{\theta}(g))$ is dense in $\R$ by Proposition \ref{prop.densetranse}, this implies that $$F_{\varepsilon}([g]) = (F_{\varepsilon} \circ \phi_t)([g]) \quad \text{for all } t \in \R.$$
    Since  $[g] \in W \times \R$ is arbitrary and $(\nu \otimes \nu_{\i})(W) = 1$, this completes the proof.
    \qed

\section{Relatively Anosov groups}

Relatively Anosov groups are relatively hyperbolic groups as abstract groups, which we now define. Let $\Ga$ be a countable group acting on a compact metrizable space $\cal X$ by homeomorphisms. 
This action is called a {\it convergence group action} if for any sequence of distinct elements $\ga_n \in \Ga$, there exist  a subsequence $\ga_{n_k}$ and $a, b \in \cal X$ such that as $k \to \infty$, 
$\ga_{n_k}(x) $ converges to $ a $ for all $x\in \cal X-\{b\}$, uniformly on compact subsets. 
An element $\ga \in \Ga$ of infinite order fixes either exactly two points in $\cal X$ or exactly one point in $\cal X$. In the former case, we call $\ga$ {\it loxodromic}, and  {\it parabolic} otherwise. An infinite subgroup $P < \Ga$ is called {\it parabolic} if $P$ fixes some point in $\cal X$ and every infinite order element of $P$ is parabolic. 

A point $\xi \in \cal X$ is called a {\it conical limit point} if there exist a sequence of distinct elements $\ga_n \in \Ga$ and distinct points $a, b \in \cal X$ such that as $n\to \infty$, $\ga_n \xi \to a$ and $\ga_n^{-1} \eta \to b$ for all $\eta \in \cal X - \{\xi\}$. A point $\xi \in \cal X$ is called a {\it parabolic limit point} if $\xi$ is fixed by a parabolic subgroup of $\Ga$. We say that a parabolic limit point $\xi \in \cal X$ is bounded if $\stab_{\Ga}(x) \ba (\cal X - \{\xi\})$ is compact. The action of $\Ga$ on $\cal X$ is called a {\it geometrically finite convergence group action} if every point of $\cal X$ is either conical or bounded parabolic limit point.
A typical example of geometrically finite convergence group action is the action of a geometrically finite Kleinian group on its limit set.

Let $\Ga$ be a finitely generated group and $\cal P$ a finite collection of finitely generated infinite subgroups of $\Ga$. We say that $\Ga$ is {\it hyperbolic relative to $\cal P$} (or that $(\Ga, \cal P)$ is {\it relatively hyperbolic}), if $\Ga$ admits a geometrically finite convergence group action on some compact perfect metrizable space $\cal X$ and the collection of maximal parabolic subgroups is $$\cal P^{\Ga} := \{ \ga P \ga^{-1} : P \in \cal P, \ga \in \Ga \}.$$
Bowditch \cite{Bowditch_relhyp} showed that for $\Ga$ hyperbolic relative to $\cal P$, the space $\cal X$ satisfying the above hypothesis is unique up to a $\Ga$-equivariant homeomorphism. Hence this space is called {\it Bowditch boundary} and denoted by $\partial (\Ga, \cal P)$.

\subsection*{The Groves-Manning cusp space}

Let $\Ga$ be a hyperbolic group relative to $\cal P$.
The {\it Groves-Manning cusp space} for $(\Ga, \cal P)$ is a proper geodesic Gromov hyperbolic space constructed by Groves-Manning \cite{GM_relhyp} on which $\Ga$ acts properly discontinuously and by isometries. 
We briefly review the construction of the Groves-Manning cusp space. We first need a notion of combinatorial horoballs: for a graph $Y$ equipped with a simplicial distance $d_Y$, the combinatorial horoball $\cal H(Y)$ is the graph 
with the vertex set $Y^{(0)} \times \N$ and two types of edges: vertical edges between vertices $(y, n)$ and $(y, n+1)$ for $y \in Y$ and $n \in \N$, and horizontal edges between vertices $(y_1, n)$ and $(y_2, n)$ for $y_1, y_2 \in Y$ and $n \in \N$ if $d_Y(y_1, y_2) \le 2^{n-1}$. We also equip $\cal H(Y)$ with the simplicial distance.

Now fix a finite generating set $S$ of $\Ga$ such that for each $P \in \cal P$, $S \cap P$ generates $P$. We denote by $\cal C(\Ga, S)$ and $\cal C(P, S \cap P)$ the Cayley graphs of $\Ga$ and $P$ with respect to $S$ and $S \cap P$ respectively. For each $\ga \in \Ga$ and $P \in \cal P$, we glue the horoball $\cal H(\ga \cal C(P, S \cap P))$ to $\cal C(\Ga, S)$, by identifying $\ga \cal C (P, S \cap P) \subset \cal C(\Ga, S)$ with $\ga \cal C (P, S \cap P) \times \{1\} \subset \cal H (\ga \cal C (P, S \cap P))$. The resulting graph equipped with the simplicial distance is called the Groves-Manning cusp space for $(\Ga, \cal P)$ and $S$, which we denote by $X_{GM}(\Ga, \cal P, S)$.

\begin{theorem} \cite[Theorem 3.25]{GM_relhyp} \label{thm.GM_relhyp}
    The space $X_{GM}(\Ga, \cal P, S)$ is a proper geodesic Gromov hyperbolic space.
\end{theorem}

From the construction, the natural action of $\Ga$ on the Cayley graph $\cal C(\Ga, S)$ induces the isometric action of $\Ga$ on  $X_{GM}(\Ga, \cal P, S)$ which is properly discontinuous. Hence the induced $\Ga$-action on the Gromov boundary $\partial X_{GM}(\Ga, \cal P, S)$ is a convergence group action \cite[Lemma 2.11]{Bowditch1999convergence}, and moreover is a geometrically finite convergence group action by the construction of $X_{GM}(\Ga, \cal P, S)$. Therefore the Gromov boundary of $X_{GM}(\Ga, \cal P, S)$ is the Bowditch boundary: $$\partial X_{GM}(\Ga, \cal P, S) = \partial (\Ga, \cal P).$$

\subsection*{Relatively Anosov subgroups}

Let $\Ga < G$ be a finitely generated non-elementary $\theta$-transverse subgroup with the limit set $\La_{\theta}$ and $\cal P$ a finite collection of finitely generated infinite subgroups of $\Ga$.

\begin{Def} 
We say that $\Ga$ is {\it $\theta$-Anosov relative to $\cal P$} if $\Ga$ is hyperbolic relative to $\cal P$ and there exists a $\Ga$-equivariant homeomorphism $\partial (\Ga, \cal P) \to \La_{\theta}$.
\end{Def}
Let $\Ga$ be a $\theta$-Anosov relative to $\cal P$ in the rest of the section. We denote by $X_{GM} := X_{GM}(\Ga, \cal P, S)$ the associated Groves-Manning cusp space for some fixed generating set $S$. We then have the $\Ga$-equivariant homeomorphism $$f : \partial X_{GM} \to \La_{\theta},$$
which has the following property:
Noting that the action of $\Ga$ is faithful on $X_{GM}$, we have a well-defined map $\Ga x\to \Ga o$ given by $\ga x \mapsto \ga o$
for any $x\in X_{GM}$.

\begin{proposition} \cite[Proposition 4.3]{CZZ_relative}  \label{prop.dynamicspreserve} Let $x \in X_{GM}$. Then the map $\Ga x \to \Ga o$ extends continuously to a unique $\Ga$-equivariant homeomorphism $f : \partial X_{GM} \to \La_{\theta}$.
\end{proposition}

By the antipodality of $\Ga$, the canonical projections $\pi_{\theta} : \La_{\theta \cup \i(\theta)} \to \La_{\theta}$ and $ \pi_{\i(\theta)}: \La_{\theta \cup \i(\theta)} \to \La_{\i(\theta)}$ are $\Ga$-equivariant homeomorphisms. This implies that being relatively $\theta$-Anosov implies being relatively $\theta \cup \i(\theta)$-Anosov as well as relatively $\i(\theta)$-Anosov.
In particular, setting the composition $f_{\i} := \pi_{\i(\theta)} \circ \pi_{\theta}^{-1} \circ f$, two maps $$f : \partial X_{GM} \to \La_{\theta} \quad \text{and} \quad f_{\i} : \partial X_{GM} \to \La_{\i(\theta)}$$ have the property that if $\xi, \eta \in \partial X_{GM}$ are distinct, then $(f(\xi), f_{\i}(\eta)) \in \F_{\theta}^{(2)}$.

\subsection*{Compatibility of shadows}

We first define the shadows in the symmetric space $X$:
for $p \in X$ and $R > 0$, let $B(p, R) $ denote the metric ball $ \{ x \in X : d(x, p) < R\}$. For $q \in X$, the {\it $\theta$-shadow} $O_R^{\theta}(q, p) \subset \F_{\theta}$ of $B(p, R)$ viewed from $q$ is defined as
$$
    O_R^{\theta}(q, p)   = \{ gP_{\theta} \in \F_{\theta} : g \in G, \ go = q, \ gA^+o \cap B(p, R) \neq \emptyset \}.
    $$
The following two lemmas will be useful:

\begin{lemma} \cite[Lemma 5.7]{LO_invariant} \label{lem.buseandcartan}
        There exists $\kappa > 0$ such that for any $g, h \in G$ and $R>0$, we have $$\sup_{\xi \in O^\theta_R(go, ho)}  \| \beta_{\xi}^{\theta}(g, h) - \mu_{\theta}(g^{-1}h) \| \le \kappa R.$$
    \end{lemma}

\begin{lemma} \cite[Lemma 9.9]{KOW_indicators} \label{lem.eventualshadow}
    Let $g_n \in G$ and $\xi_n\in \F_\theta$  be sequences both  converging to some $\xi \in \F_{\theta}$. Suppose that 
    there exists a sequence $\eta_n\in \F_{\i(\theta)}$ converging to some $\eta\in \F_{\i(\theta)}$  such that $(\xi, \eta) \in \F_{\theta}^{(2)}$ and the sequence $g_n^{-1}(\xi_n, \eta_n)$ is precompact in $\F_{\theta}^{(2)}$. Then there exists $R > 0$ such that $$\xi_n \in O_R^{\theta}(o, g_n o) \quad \text{for all } n \ge 1.$$
\end{lemma}

We also consider shadows in Groves-Manning cusp space.
Let $d_{GM}$ be the simplicial distance on $X_{GM}$. 

The following theorem is obtained in  \cite[Theorem 10.1]{CZZ_relative}; although it stated only the lower bound, the upper bound also follows from its proof: \begin{theorem} \label{thm.czzcartan}
    For any $(\Ga,\theta)$-proper linear form
$\psi \in \fa_{\theta}^*$,
there exists positive constants $c, {c}'$ and $C $ such that 
     for all $\ga \in \Ga$,
    $$  c \, d_{GM}(e, \ga) - C \le \psi(\mu_{\theta}(\ga)) \le {c'} \, d_{GM}(e, \ga) + C. $$
\end{theorem}

For $y \in X_{GM}$ and $R > 0$, we denote the $R$-ball centered at $y$ by
$$B_{GM}(y, R) := \{z \in X_{GM} : d_{GM}(y, z) < R\}.$$ 
For $x, y \in X_{GM}$ and $R > 0$, we define the shadow of $B_{GM}(y, R)$ viewed from $x$ as follows:
$$O_R^{GM}(x, y) := \left\{ \xi \in \partial X_{GM} : \begin{matrix}
    \text{there exists a geodesic ray from } x \text{ to } \xi\\
    \text{passing through } B_{GM}(y, R)
\end{matrix}\right\}.$$
Note that $\xi \in \partial X_{GM}$ is a conical limit point if and only if there exists $R > 0$ such that $\xi \in O_R^{GM}(o, \ga_n o)$ for an infinite sequence $\ga_n \in \Ga$. 

We prove the following compatibility of shadows under $f : \partial X_{GM} \to \La_{\theta}$:

\begin{proposition} \label{prop.shadowcompare}
    Let $x \in X_{GM}$ and $o \in X$. For all sufficiently large $R > 1$, there exist $r_1=r_1(R), r_2 = r_2(R) > 0$ such that for any $\ga \in \Ga$, we have $$ O_{r_1}^{\theta}(o, \ga o)  \cap \La_{\theta} \subset f(O_R^{GM}(x, \ga x)) \subset O_{r_2}^{\theta}(o, \ga o) \cap \La_{\theta} .$$
   Moreover, we can take $r_1(R)\to \infty $ as $R\to \infty$.
   
\end{proposition}

We begin with some lemmas:

\begin{lemma} \label{lem.nonemptyshadow}
    For any $x \in X_{GM}$, there exists $R_0 > 0$ such that $O_{R_0}^{GM}(x, \ga x) \neq \emptyset$ for any $\ga \in \Ga$.
\end{lemma}

\begin{proof}
    Suppose not. Then there exists an infinite sequence $\ga_n \in \Ga$ so that $O_n^{GM}(x, \ga_n x) = \emptyset$, and hence $O_n^{GM}(\ga_n^{-1}x, x) = \emptyset$ for all $n \ge 1$. This forces $\partial X_{GM}$ to be a singleton, which contradicts the perfectness of $\partial X_{GM}$.
\end{proof}

\begin{lemma} \label{lem.shadowgiveslimit}
    Let $x \in X_{GM}$ and $R > 0$. Let $\ga_n \in \Ga$ and $\xi_n \in \partial X_{GM}$ be sequences such that $\xi_n \in O_R^{GM}(x, \ga_n x)$  for all $n \ge 1$.
    If  $\ga_n x\to \xi \in \partial X_{GM}$ as $n \to \infty$, then $\xi_n \to \xi$ as $n \to \infty$.
\end{lemma}

\begin{proof}
    Suppose to the contrary that the sequence $\xi_n$, after passing to a subsequence, converges to $\xi' \in \partial X_{GM}$ distinct from $\xi$.
    Since $\ga_n x \to \xi$ as $n \to \infty$ and $X_{GM}$ is Gromov hyperbolic (Theorem \ref{thm.GM_relhyp}), this implies that there exist a constant $R' > 0$ and a sequence of geodesic rays $[\ga_n x, \xi_n]$ from $\ga_n x$ to $\xi_n$ such that $d_{GM}(x, [\ga_n x, \xi_n]) < R'$ for all $n \ge 1$. On the other hand, since $\xi_n \in O_R^{GM}(x, \ga_n x)$, there exists a geodesic ray $[x, \xi_n]$ from $x$ to $\xi_n$ and a point $c_n \in [x, \xi_n]$ such that $d_{GM}(c_n, \ga_n x) < R$ for all $n \ge 1$. Since the distance between $\ga_n x$ and $c_n$ is uniformly bounded, the Hausdorff distance between two geodesic rays $[\ga_n x, \xi_n]$ and $[c_n, \xi_n] \subset [x, \xi_n]$ is uniformly bounded, by the Gromov hyperbolicity of $X_{GM}$ (Theorem \ref{thm.GM_relhyp}).  Since the distance $d_{GM}(x, [\ga_n x, \xi_n]) $ is uniformly bounded, this implies that the distance $d_{GM}(x, [c_n, \xi_n])$ is uniformly bounded as well. Since $[c_n, \xi_n]$ is the geodesic ray contained in the geodesic ray $[x, \xi_n]$, we have that $d_{GM}(x, c_n)= d_{GM}(x, [c_n, \xi_n])$ is uniformly bounded. Therefore, it follows from  the uniform boundedness of $d_{GM}(c_n, \ga_n x)$ that $d_{GM}(x, \ga_n x)$ is uniformly bounded, which contradicts the hypothesis that $\ga_n x \to \xi$ as $n \to \infty$. This finishes the proof.
\end{proof}

\subsection*{Proof of Proposition \ref{prop.shadowcompare}}

Note that the first inclusion and the last claim follow once we show that for any $c > 0$, there exists $C > 0$ such that $O_c^{\theta}(o, \ga o) \subset f(O_C^{GM}(x, \ga x))$ for all $\ga \in \Ga$.
Suppose not.  Then there exist sequences $\ga_n \in \Ga$ and $\xi_n \in \partial X_{GM} - O_n^{GM}(x, \ga_n x)$ such that $f(\xi_n) \in O_c^{\theta}(o, \ga_n o)$ for all $n \ge 1$. After passing to a subsequence, we may assume that the sequence $\ga_n^{-1} x$ converges to some  point $\eta \in \partial X_{GM}$ as $n \to \infty$. Since $\ga_n^{-1} \xi_n \notin O_n^{GM}(\ga_n^{-1} x, x)$ for all $n \ge 1$, we have that 
\be \label{eqn.sep10}
\lim_{n \to \infty} \ga_n^{-1} \xi_n = \eta.
\ee
On the other hand, by Proposition \ref{prop.dynamicspreserve}, we have $\lim_{n \to \infty} \ga_n^{-1} = f_{\i}(\eta) \in \La_{\i(\theta)}$. Since $f(\ga^{-1}_n \xi_n) \in O_c^{\theta}(\ga_n^{-1} o,  o)$ for all $n \ge 1$ and $\lim_{n \to \infty} \ga_n^{-1} = f_{\i}(\eta)$, it follows from \eqref{eqn.sep10} and the continuity of higher-rank shadows on viewpoints \cite[Proposition 3.4]{KOW_ergodic} that $f(\eta) = \lim_{n \to \infty} f(\ga_n^{-1}\xi_n) \in \La_{\theta}$ is in general position with $f_{\i}(\eta)$. This yields contradiction.

We now prove the second inclusion.
Let $R_0>0$ be as given by Lemma \ref{lem.nonemptyshadow} and fix $R > R_0$.  Let $x \in X_{GM}$ and $o \in X$. Suppose on the contrary that there exists a sequence $\ga_n \in \Ga$ such that $$f(O_R^{GM}(x, \ga_n x)) \not\subset O_n^{\theta}(o, \ga_n o) \quad \text{for all } n \ge 1.$$
This means that there exists a sequence $\xi_n \in O_R^{GM}(x, \ga_n x)$ such that $f(\xi_n) \not\in O_n^{\theta}(o, \ga_n o)$ for all $n \ge 1$. After passing to a subsequence, we may assume that the sequence $\ga_n x$ converges to a point $\xi \in \partial X_{GM}$. By Proposition \ref{prop.dynamicspreserve}, we have
\be \label{eqn.shadowcompare1}
\ga_n \to f(\xi) \quad \text{as } n \to \infty.
\ee

In addition, it follows from Lemma \ref{lem.shadowgiveslimit} that $\xi_n \to \xi$ as $n \to \infty$. For each $n \ge 1$, we choose a point $\eta_n \in O_R^{GM}(\ga_n x, x)$ which is possible by Lemma \ref{lem.nonemptyshadow}. We may assume that the sequence $\eta_n$ converges to $\eta \in \partial X_{GM}$, after passing to a subsequence. Since $\ga_n x \to \xi$ as $n \to \infty$ and $\eta_n \in O_R^{GM}(\ga_n x, x)$ for all $n \ge 1$, we have $\xi \neq \eta$. Therefore, we have the following convergence of the sequence in $\F_{\theta}^{(2)}$:
\be \label{eqn.shadowcompare2}
(f(\xi_n), f_{\i}(\eta_n)) \to (f(\xi), f_{\i}(\eta)) \in \F_{\theta}^{(2)} \quad \text{as }n \to \infty.
\ee

On the other hand, we also have $\ga_n^{-1}\xi_n \in O_R^{GM}(\ga_n^{-1} x, x)$ and $\ga_n^{-1} \eta_n \in O_R^{GM}(x, \ga_n^{-1} x)$ for all $n \ge 1$. Together with the $\Ga$-equivariance of $f$ and $f_{\i}$, a similar argument as above implies that
\be \label{eqn.shadowcompare3}
\text{the sequence } \ga_n^{-1}(f(\xi_n), f_{\i}(\eta_n)) \text{ is precompact in } \F_{\theta}^{(2)}.
\ee
By \eqref{eqn.shadowcompare1}, \eqref{eqn.shadowcompare2}, and \eqref{eqn.shadowcompare3}, we apply Lemma \ref{lem.eventualshadow} and deduce that there exists $R' > 0$ so that $f(\xi_n) \in O_{R'}^{\theta}(o, \ga_n o)$ for all $n \ge 1$. This contradicts to the choice of the sequence $\xi_n$ that $f(\xi_n) \notin O_n^{\theta}(o, \ga_n o)$ for all $n \ge 1$. This completes the proof.
\qed

\begin{lemma} \label{lem.gromovproductbdd}
    Let $x \in X_{GM}$ and $R > 0$. Then there exists a compact subset $Q \subset \fa_{\theta}$ satisfying the following: if $\xi, \eta \in \partial X_{GM}$ are such that $d_{GM}(x, [\xi, \eta]) < R$ for some bi-infinite geodesic $[\xi, \eta]$, then $$\langle f(\xi), f_{\i}(\eta) \rangle \in Q$$
    where $\langle \cdot, \cdot \rangle$ is the Gromov product defined in \eqref{eqn.gromovproductdef}.
\end{lemma}

\begin{proof}
    Suppose not. Then there exists  a sequence of bi-infinite geodesics $[\xi_n, \eta_n]$ for some $\xi_n,\eta_n\in \partial X_{GM}$ such that we have
    $\sup_n d_{GM}(x, [\xi_n, \eta_n]) < R$ and  the Gromov products $\langle f(\xi_n), f_{\i}(\eta_n) \rangle $ escape every compact subset of $\fa_{\theta}$ as $n \to \infty$. After passing to a subsequence, we may assume that  $\xi_n\to \xi$ and $\eta_n\to \eta $  in $\partial X_{GM}$. The hypothesis $\sup_n d_{GM}(x, [\xi_n, \eta_n]) < R$  implies $\xi \neq \eta$, since $X_{GM}$ is Gromov hyperbolic (Theorem \ref{thm.GM_relhyp}).
    Therefore  $(f(\xi), f_{\i}(\eta))\in \La_\theta^{(2)}$  and hence
    $\langle f(\xi), f_{\i}(\eta) \rangle\in \fa_\theta$ is well-defined.
    On the other hand, by the continuity of the Gromov product,
    we have $\langle f(\xi_n), f_{\i}(\eta_n) \rangle \to \langle f(\xi), f_{\i}(\eta) \rangle \in \fa_{\theta}$ as $n \to \infty$. This yields a contradiction.
\end{proof}

\section{Reparameterization for relatively Anosov groups} \label{sec.reparconstruct}
Let $\Ga<G$ be a $\theta$-Anosov subgroup relative to $\cal P$ and $X_{GM} = X_{GM}(X, \cal P, S)$ the associated Groves-Manning cusp space for a fixed generating set $S$. 
Fix a $(\Ga,\theta)$-proper linear form
$\psi \in \fa_{\theta}^*$. 
Recall from section \ref{sec.bundle} the space $\tilde \Omega_{\psi} := \La_{\theta}^{(2)} \times \R$ equipped with the $\Ga$-action given by
$$
\ga (\xi, \eta, s) = (\ga \xi, \ga \eta, s + \psi(\beta_{\xi}^{\theta}(\ga^{-1}, e))).
$$
As stated in Theorem \ref{op2}, the space
$$\Omega_{\psi} := \Ga \ba \tilde \Omega_{\psi}$$
is a locally compact second countable Hausdorff space. The translation flow $\{\phi_t\}$ on the $\R$-component of $\tilde \Omega_{\psi}$ commutes with the $\Ga$-action, and hence it induces the translation flow on $\Omega_{\psi}$ which we also denote by 
     $\{\phi_t\}$.
We will relate $\tilde \Omega_{\psi}$ and  $\Omega_\psi$
with the Groves-Manning cusp space $X_{GM}$ in this section. More precisely, let $$\cal G := \{ \sigma : \R \to X_{GM}: \text{ bi-infinite geodesic}\}.$$
The space $\cal G$ admits the geodesic flow $\varphi_s : \cal G \to \cal G$ defined by $(\varphi_s \sigma)(\cdot) = \sigma( \cdot + s)$ for $s \in \R$, and the inversion $I : \cal G \to \cal G$ defined by $(I \sigma)(s) = \sigma(-s)$ for $s \in \R$.
The canonical isometric action of $\Ga$ on $\cal G$ commutes with the geodesic flow and $I$, and is properly discontinuous. Hence we can also consider the locally compact Hausdorff space $\Ga \ba \cal G$. This section is devoted to the proof of the following reparameterization theorem:

  Set
  \be\label{aa} a= \liminf_{\ga \in \Ga} \frac{\psi(\mu_{\theta}(\ga))}{d_{GM}(e, \ga)}\quad\text{and}\quad a' = 3 \limsup_{\ga \in \Ga} \frac{\psi(\mu_{\theta}(\ga))}{d_{GM}(e, \ga)} .\ee 
  By Theorem \ref{thm.czzcartan},
  we have $0<a\le a' <\infty$.
  
\begin{theorem}[Reparameterization, Theorem \ref{thm.repar0}(1)-(3)] \label{thm.repar} \label{thm.reparflow}
    There exists a continuous, surjective, proper $\Ga$-equivariant map
    $$\tilde \Psi : \cal G \to \tilde \Omega_{\psi}.$$
    Moreover,  we have a continuous cocycle $\tilde{\mathsf t} : \cal G\times \br\to \br$   such that  for all $\sigma\in \cal G$ and $ s\in \br$, 
    \begin{enumerate}
        \item 
   $\tilde \Psi(\varphi_s \sigma ) = \phi_{\tilde{\mathsf t}(\sigma, s)} \tilde \Psi(\sigma)$;
\item      $\tilde{\mathsf t}(\sigma, s)= - \tilde{\mathsf t}(\varphi_s \sigma, -s)$;
 \item  there exists an absolute constant  $B > 0$ such that  $$a|s| - B \le \tilde{\mathsf t} (\sigma, |s|)  \le a'|s| + B.$$
  
     \end{enumerate}
     
\end{theorem}

In the above theorem, $\tilde{\mathsf t} : \cal G\times \br\to \br$ being a continuous cocycle means that it is continuous and for all $\sigma \in \cal G$ and $s_1, s_2 \in \R$,
$$
\tilde{\mathsf t}(\sigma, s_1 + s_2) = \tilde{\mathsf t} ( \sigma, s_1) + \tilde{\mathsf t}(\varphi_{s_1} \sigma, s_2).
$$
Since $\tilde \Psi : \cal G \to \tilde \Omega_{\psi}$ in Theorem \ref{thm.repar} is $\Ga$-equivariant, this descends to the map $\Psi : \Ga \ba \cal G \to \Omega_{\psi}$. The following is immediate from Theorem \ref{thm.repar}.

\begin{corollary}[Reparameterization] \label{cor.reparquotient}
    There exists a continuous, surjective, proper map $$\Psi : \Ga \ba \cal G \to \Omega_{\psi}.$$
    Moreover,  we have a continuous cocycle $\mathsf t:\Ga \ba \cal G\times
    \br\to \br$  such that  for all $\sigma\in \cal G$ and $ s\in \br$, 
    \begin{enumerate}
        \item 
   $\Psi([\varphi_s \sigma ]) = \phi_{\mathsf t(\sigma, s)} \Psi([\sigma])$;
\item      $\mathsf{t}(\sigma, s)= - \mathsf{t}(\varphi_s \sigma, -s)$;
 \item  there exists an absolute constant  $B > 0$ such that  $$a|s| - B \le \mathsf t (\sigma, |s|)  \le a'|s| + B.$$
  \end{enumerate}
\end{corollary}

\subsection*{Thick-thin decomposition of $\cal G$}
For $P \in \cal P$, let $\xi_P \in \partial X_{GM}$ be the bounded parabolic limit point fixed by $P$. We consider the open {\it horoball} $H_P \subset X_{GM}$ based at $\xi_P$ invariant under $P$, defined as follows: let $H_P' \subset X_{GM}$ be the subgraph induced by the vertices $\{(g, n) :  g \in P, n \ge 2\}$ and $\hat{H}_P \subset X_{GM}$ be the subgraph induced by the vertices $\{(g, 2) : g \in P\}$. We then set $$H_P := H_P' - \hat{H}_P.$$
For $\ga \in \Ga$, we also set $$H_{\ga P \ga^{-1}} := \ga H_P$$ which is the open horoball based at $\xi_{\ga P \ga^{-1}} := \ga \xi_P$ and invariant under $\ga P \ga^{-1} \in \cal P^{\Ga}$. The boundary $\partial H_{\ga P \ga^{-1}}$ consists of the vertices $\ga \{ (g, 2) : g \in P \}$. 
We then have the $\Ga$-invariant family $\{ H_P : P \in \cal P^{\Ga} \}$ of open horoballs with disjoint closures. 

We define the following subsets of $\cal G$: for $P \in \cal P^{\Ga}$, let
$$\begin{aligned}
    \cal G_P & := \{ \sigma \in \cal G : \sigma(0) \in H_P\}; \\
    \partial \cal G_P & := \{ \sigma \in \cal G : \sigma(0) \in \partial H_P \}.
    \end{aligned}$$
    We  have the {\it thick-thin decomposition} of $\cal G$: $$\cal G_{thin} := \bigcup_{P \in \cal P^{\Ga}} \cal G_P \quad \text{and} \quad \cal G_{thick} := \cal G - \cal G_{thin}.$$
    Since the Groves-Manning cusp space $X_{GM}$ is constructed by attaching combinatorial horoballs to the Cayley graph of $\Ga$, the $\Ga$-action on $X_{GM} - \bigcup_{P \in \cal P^{\Ga}} H_P$ is cocompact. Hence the $\Ga$-action on $\cal G_{thick}$ which consists of bi-infinite geodesics based at $X_{GM} - \bigcup_{P \in \cal P^{\Ga}} H_P$ is also cocompact.

We also introduce the following subsets of $\partial \cal G_P$ for each $P \in \cal P^{\Ga}$:
$$\begin{aligned}
    \partial^+ \cal G_P & := \{ \sigma \in \partial \cal G_P : \sigma(t) \in H_P \text{ for all sufficiently small } t > 0 \}; \\
    \partial^- \cal G_P & := \{ \sigma \in \partial \cal G_P : \sigma(-t) \in H_P \text{ for all sufficiently small } t > 0 \}.
\end{aligned}$$
Note that $\partial^+ \cal G_P \cap \partial^- \cal G_P = \emptyset$. For $\sigma \in \partial^+ \cal G_P$, we set $$T_{\sigma}^+ := \min \{ t \in (0, \infty] : \sigma(t) \notin H_P\},$$
and for $\sigma \in \partial^- \cal G_P$, we set  $$T_{\sigma}^- := \max \{ t \in [-\infty, 0) : \sigma(t) \notin H_P\},$$
which are the escaping times for the horoball $H_P$. We then have $$\cal G_P = \left( \bigcup_{\sigma \in \partial^+ \cal G_P} \bigcup_{t \in (0, T_{\sigma}^+)} \varphi_t\sigma \right) \cup \left( \bigcup_{\sigma \in \partial^- \cal G_P} \bigcup_{t \in (T_{\sigma}^-, 0)} \varphi_t\sigma \right).$$

\subsection*{Construction of the reparameterization}

To construct the reparameterization, we consider the trivial bundle $$\cal G \times \R_{+} \to \cal G.$$ Given $\sigma \in \cal G$, we denote by $\sigma^+ = \sigma(\infty) \in \partial X_{GM}$ and $\sigma^- = \sigma(-\infty) \in \partial X_{GM}$ the forward and backward endpoint of the bi-infinite geodesic $\sigma$. Noting that we have $\Ga$-equivariant homeomorphisms $f : \partial X_{GM} \to \La_{\theta}$ and $f_{\i} : \partial X_{GM} \to \La_{\i(\theta)}$, we identify $\partial X_{GM}$, $\La_{\theta}$, and $\La_{\i(\theta)}$ in this section via the homeomorphisms. We define the $\Ga$-action on $\cal G \times \R_{+}$ as follows: for $\ga \in \Ga$ and $(\sigma, v) \in \cal G \times \R_{+}$, $$\ga(\sigma, v) = \left(\ga \sigma, v e^{\psi\left(\beta_{\sigma^+}^{\theta}(\ga^{-1}, e)\right)}\right).$$ This action makes the following projection $\Ga$-equivariant: $$\begin{aligned}
    \Psi_0 : \cal G \times \R_{+} & \longrightarrow \qquad \tilde{\Omega}_{\psi} \\
     (\sigma, v) \quad & \longmapsto  (\sigma^+, \sigma^-, \log v).
\end{aligned}$$
We construct the reparameterization $\Psi : \Ga \ba \cal G \to \Omega_{\psi}$ in Theorem \ref{thm.repar} by constructing a nice $\Ga$-equivariant section $u : \cal G \to \cal G \times \R_{+}$ of the trivial bundle so that we obtain a $\Ga$-equivariant map $\tilde \Psi : \cal G \to \tilde{\Omega}_{\psi}$ as follows, with the desired properties:
$$\begin{tikzcd}
    \cal G \times \R_+ \arrow[rd, "\Psi_0"] \arrow[d] & \\
    \cal G \arrow[r, dashed, "\tilde{\Psi}"'] \arrow[u, bend left=50, dashed, "u"] & \tilde{\Omega}_{\psi}
\end{tikzcd}$$

\subsection*{Norms on fibers} To construct a section of the trivial bundle $\cal G \times \R_{+} \to \cal G$, we define a  continuous family of $\Ga$-equivariant norms on fibers. More precisely, we define a $\Ga$-invariant continuous function $$\| \cdot \| : \cal G \times \R_{+} \to \R_{+}$$ such that for each $\sigma \in \cal G$, $\| (\sigma, \cdot) \|$ is the restriction of a norm on $\br$ to $\R_{+}$. We simply write $$\| \cdot \|_{\sigma} := \| (\sigma, \cdot)\|\quad \text{ for each $\sigma \in \cal G$. }$$ Once we define the norm, we will define a section $u :  \cal G \to \cal G \times \R_{+}$ by $u(\sigma) = (\sigma, v_{\sigma})$ where $v_{\sigma} \in \R_{+}$ is the unique unit vector with respect to the norm $\| \cdot \|_{\sigma}$, i.e., $\|v_{\sigma} \|_{\sigma} = 1$. The $\Ga$-equivariance and the continuity of the norms imply that the section $u$ is also $\Ga$-equivariant and continuous. To make the reparameterization $\tilde{\Psi} = \Psi_0 \circ u$ satisfy the conditions in Theorem \ref{thm.repar}, our norms should have
a  property that the contraction rate along the geodesic flow is bounded from both {\it above} and {\it below} by uniform exponential functions.

Our construction of the family of norms is motivated by \cite{ZZ_relatively} which considered flat bundles for relatively Anosov subgroups of $\SL(n, \R)$ with respect to a maximal parabolic subgroup. Our proof of the contraction property is motivated by (\cite{CZZ_cusped},  \cite{ZZ_relatively}) where the upper bound of the contraction rate of norms on flat bundles for relatively Anosov subgroups of $\SL(n, \R)$ with respect to a maximal parabolic subgroup was proved. We also remark that
the contraction property was earlier studied in (\cite{Tsouvalas_Anosov}, \cite{CS_local}) for Anosov subgroups.

We now define a family of norms as follows (compare to a similar construction in \cite{ZZ_relatively}): first we fix a continuous family of $\Ga$-equivariant norms $\| \cdot \|_{\sigma}$ for $\sigma \in \cal G_{thick}$ such that $\| \cdot \|_{\sigma} = \| \cdot \|_{I\sigma}$ for all $\sigma \in \cal G_{thick}$. Let $\sigma \in \cal G_{thin}$. Then $\sigma \in \cal G_P$ for some $P \in \cal P^{\Ga}$. Let 
\be \label{eqn.defofc}
c > 0
\ee be the constant given by Theorem \ref{thm.czzcartan}. 
There are two cases indicated by the Figures \ref{fig.case1} and \ref{fig.case2}:

\medskip
{\noindent \bf Case 1.}   If $\sigma = \varphi_t\sigma_0$ for some $\sigma_0 \in \partial^+ \cal G_P$ and $t \in (0, T_{\sigma_0}^+)$, we write $T := T_{\sigma_0}^+$ and 
\begin{itemize}
    \item if $t \in \left(0, \tfrac{1}{3}T\right]$, we set $$\|\cdot \|_{\sigma} := e^{-c t} \| \cdot \|_{\sigma_0}.$$

    \item if $t \in \left[\tfrac{2}{3}T, T\right)$, we set $$\|\cdot \|_{\sigma} := e^{c(T - t)} \| \cdot \|_{\varphi_T\sigma_0}.$$

    \item if $t \in \left(\tfrac{1}{3}T, \tfrac{2}{3}T\right)$, we set $$\| \cdot \|_{\sigma} := \| \cdot \|_{\varphi_{T/3}\sigma_0}^{2 - \tfrac{3}{T} t} \| \cdot \|_{\varphi_{2T/3}\sigma_0}^{\tfrac{3}{T}t - 1}.$$
\end{itemize}
\begin{figure}[h]
\begin{tikzpicture}
    \draw (0, 0) circle(2);
    \fill[fill=gray!10] (0, 0.5) circle(1.5);

    \draw (1.2, 0) node {\tiny $H_P$};

    \filldraw (0, 2) circle(1pt);
    \draw (0, 2) node[above] {$\xi_P$};

    \draw (0, 2) arc(0:-30:7.45);

    \filldraw[blue] (-0.27, 0) circle(1pt);
    \draw[->, thick, blue] (-0.27, 0) -- (-0.15, 0.4);
    \draw[blue] (-0.15, 0.4) node[right] {$\sigma$};

    \filldraw[red] (-0.58, -0.88) circle(1pt);
    \draw[->, thick, red] (-0.58, -0.88) -- (-0.42, -0.52);
    \draw[red] (-0.42, -0.52) node[right] {$\sigma_0$};

    \draw[dashed] (-0.68, -0.78) .. controls (-0.78, -0.45) and (-0.57, 0) .. (-0.37, 0);
    \draw (-0.6, -0.2) node[left] {$t$};
    
\end{tikzpicture}
\quad 
\begin{tikzpicture}
    \draw (0, 0) circle(2);
    \fill[fill=gray!10] (0, 0.5) circle(1.5);
    
    \draw (1.2, 0) node {\tiny $H_P$};

    \filldraw (0, 2) circle(1pt);
    \draw (0, 2) node[above] {$\xi_P$};

    \draw (0, -2) arc(180:135:4.84);

    \filldraw[blue] (0.69, 0.5) circle(1pt);
    \draw[->, thick, blue] (0.69, 0.5) -- (0.95, 0.9);
    \draw[blue] (0.95, 0.9) node[left] {$\sigma$};

    \filldraw[red] (0.11, -1) circle(1pt);
    \draw[->, thick, red] (0.1, -1) -- (0.2, -0.5);
    \draw[red] (0.2, -0.5) node[left] {$\sigma_0$};

    \draw[dashed] (0.2, -1) .. controls (0.4, -1) and (0.85, 0.1) .. (0.75, 0.4);
    \draw (0.5, -0.5) node[right] {$t$};
    
\end{tikzpicture}
\caption{Two possible configurations of $\sigma \in \cal G_P$ in Case 1 depending on whether $T_{\sigma_0}^+ = \infty$ or not. Only the first item in Case 1 applies to the left figure.} \label{fig.case1}
\end{figure}

\medskip
{\noindent \bf Case 2.} If $\sigma = \varphi_s\tilde{\sigma}_0$ for some $\tilde{\sigma}_0 \in \partial^- \cal G_P$ and $s \in (T_{\tilde{\sigma}_0}^-, 0)$, we write $T := T_{\tilde{\sigma}_0}^-$ and
\begin{itemize}
    \item if $s \in \left[ \tfrac{1}{3}T, 0 \right)$, we set $$\| \cdot \|_{\sigma} := e^{-c s} \| \cdot \|_{\tilde{\sigma}_0}.$$

    \item if $s \in \left( T, \tfrac{2}{3}T \right]$, we set $$\| \cdot \|_{\sigma} := e^{c(T - s)} \| \cdot \|_{\varphi_T\tilde{\sigma}_0}.$$

    \item if $s \in \left( \tfrac{2}{3}T, \tfrac{1}{3}T \right)$, we set $$\| \cdot \|_{\sigma} := \| \cdot \|_{\varphi_{2T/3}\tilde{\sigma}_0}^{\tfrac{3}{T}s - 1} \| \cdot \|_{\varphi_{T/3}\tilde{\sigma}_0}^{2 - \tfrac{3}{T} s}.$$
\end{itemize}
\begin{figure}[h]
\begin{tikzpicture}
\begin{scope}[xscale=-1]
    \draw (0, 0) circle(2);
    \fill[fill=gray!10] (0, 0.5) circle(1.5);

    \draw (1.2, 0) node {\tiny $H_P$};

    \filldraw (0, 2) circle(1pt);
    \draw (0, 2) node[above] {$\xi_P$};

    \draw (0, 2) arc(0:-30:7.45);

    \filldraw[blue] (-0.27, 0) circle(1pt);
    \draw[->, thick, blue] (-0.27, 0) -- (-0.39, -0.4);
    \draw[blue] (-0.27, 0) node[left] {$\sigma$};

    \filldraw[red] (-0.58, -0.88) circle(1pt);
    \draw[->, thick, red] (-0.58, -0.88) -- (-0.74, -1.24);
    \draw[red] (-0.58, -0.88) node[left] {$\tilde \sigma_0$};

    \draw[dashed] (-0.68, -0.78) .. controls (-0.78, -0.45) and (-0.57, 0) .. (-0.37, 0);
    \draw (-0.6, -0.2) node[right] {$s$};
    \end{scope}
\end{tikzpicture}
\quad 
\begin{tikzpicture}
\begin{scope}[xscale=-1]
    \draw (0, 0) circle(2);
    \fill[fill=gray!10] (0, 0.5) circle(1.5);
    
    \draw (1.2, 0) node {\tiny $H_P$};

    \filldraw (0, 2) circle(1pt);
    \draw (0, 2) node[above] {$\xi_P$};

    \draw (0, -2) arc(180:135:4.84);

    \filldraw[blue] (0.69, 0.5) circle(1pt);
    \draw[->, thick, blue] (0.69, 0.5) -- (0.43, 0.1);
    \draw[blue] (0.69, 0.5) node[right] {$\sigma$};

    \filldraw[red] (0.11, -1) circle(1pt);
    \draw[->, thick, red] (0.1, -1) -- (0, -1.5);
    \draw[red] (0.1, -1)  node[right] {$\tilde \sigma_0$};

    \draw[dashed] (0.2, -1) .. controls (0.4, -1) and (0.85, 0.1) .. (0.75, 0.4);
    \draw (0.5, -0.5) node[left] {$s$};
    \end{scope}
\end{tikzpicture}
\caption{Two possible configurations of $\sigma \in \cal G_P$ in Case 2 depending on whether $T_{\tilde \sigma_0}^- = - \infty$ or not. Only the first item in Case 2 applies to the left figure.} \label{fig.case2}
\end{figure}

Note that both cases can happen at the same time, and in that case two definitions coincide. The resulting family of norms is continuous and $\Ga$-equivariant.

\subsection*{Contraction rate along geodesic flow}
For $\sigma \in \cal G$, there exists a unique $v_{\sigma} \in \R_{+}$ such that $\|v_{\sigma} \|_{\sigma} = 1$. For  $t \in \R$, we define 
\be\label{kappat} \kappa_{t}(\sigma) := \| v_{\sigma} \|_{\varphi_t\sigma};\ee 
this measures the contraction rates of  norms under the geodesic flow. It is easy to see that for $\sigma \in \cal G$ and $t, s \in \R$, we have 
\be \label{eqn.cocycle}
v_{\varphi_t \sigma} = \frac{v_{\sigma}}{\|v_{\sigma}\|_{\varphi_t\sigma}}  \quad \text{and} \quad \kappa_{t + s}(\sigma) = \kappa_s(\varphi_t\sigma) \kappa_t(\sigma).
\ee
Moreover, $\kappa_t(\cdot)$ is $\Ga$-invariant.
\begin{lemma} \label{lem.invariancecocycle}
    For $\sigma \in \cal G$, $t \in \R$, and $\ga \in \Ga$, we have $$\kappa_t(\ga \sigma) = \kappa_t(\sigma).$$
\end{lemma}

\begin{proof}
    By the $\Ga$-equivariance of the norm, we have $$1 = \| v_{\sigma} \|_{\sigma} = \left\| v_{\sigma} e^{\psi(\beta_{\sigma^+}^{\theta}(\ga^{-1}, e))} \right\|_{\gamma \sigma}.$$
    This implies 
    \be \label{eqn.vgammasigma}
    v_{\ga \sigma} = v_{\sigma} e^{\psi(\beta_{\sigma^+}^{\theta}(\ga^{-1}, e))}.
    \ee
    Since $\varphi_t \ga \sigma = \ga\varphi_t\sigma$, we have $$\begin{aligned}
        \kappa_t(\ga \sigma) & = \| v_{\ga \sigma} \|_{\varphi_t\ga \sigma} = \|v_{\sigma}\|_{\ga \varphi_t\sigma} e^{\psi(\beta_{\sigma^+}^{\theta}(\ga^{-1}, e))}\\
        & = \left\| v_{\sigma} e^{\psi(\beta_{\ga \sigma^+}^{\theta}(\ga, e))} \right\|_{\varphi_t\sigma}e^{\psi(\beta_{\sigma^+}^{\theta}(\ga^{-1}, e))} \\
        & = \| v_{\sigma}\|_{\varphi_t\sigma} = \kappa_t(\sigma)
    \end{aligned}$$
    as desired.
\end{proof}

The following is the desired estimate on the contraction rate:

\begin{theorem} \label{thm.expdecay}
    There exists $ b>1 $ such that for all $\sigma \in \cal G$ and $t \ge 0$, we have $$ \frac{1}{b} e^{-a't} \le  \kappa_t(\sigma) \le b e^{-a t}$$
    where 
  $a= \liminf_{\ga \in \Ga} \frac{\psi(\mu_{\theta}(\ga))}{d_{GM}(e, \ga)}$ and $a' = 3 \limsup_{\ga \in \Ga} \frac{\psi(\mu_{\theta}(\ga))}{d_{GM}(e, \ga)} $.
\end{theorem}

We begin by observing that the recurrence to a compact subset implies the exponential contraction:

\begin{lemma} \label{lem.cptpart}
    For any compact subset $Q \subset X_{GM}$, there exists $C_Q > 1$ such that if $\sigma \in \cal G$, $t \ge 0$, and $\ga \in \Ga$ satisfy $\sigma(0), \ga^{-1}\sigma(t) \in Q$, then $$\frac{1}{C_Q} e^{-\psi(\mu_{\theta}(\ga))} \le \kappa_t(\sigma) \le C_Q e^{-\psi(\mu_{\theta}(\ga))}.$$
\end{lemma}

\begin{proof}
    Suppose not. Then there exist sequences $\sigma_n \in \cal G$, $t_n \ge 0$, and $\ga_n \in \Ga$ such that $\sigma_n(0), \ga_n^{-1} \sigma_n(t_n) \in Q$ for all $n \ge 1$ while the sequence 
    \be \label{eqn.cptpart1}
    \log \left( \kappa_{t_n}(\sigma_n) e^{\psi(\mu_{\theta}(\ga_n))} \right) = \psi(\mu_{\theta}(\ga_n)) + \log \kappa_{t_n}(\sigma_n) \quad \text{is unbounded.}
    \ee
    In particular, $\ga_n$ is an infinite sequence and $t_n \to \infty$ as $n \to \infty$.
    
    By the hypothesis that $\sigma_n(0), \ga_n^{-1} \sigma_n(t_n) \in Q$, there exist $q \in Q$ and $R > 0$ depending on $Q$ so that we have $\sigma_n^+ \in O_{R}^{GM}(q, \ga_n q)$ for all $n \ge 1$. It follows from Proposition \ref{prop.shadowcompare} that for some $r > 0$, we have $\sigma_n^+ \in O_r^{\theta}(o, \ga_n o)$ for all $n \ge 1$. By Lemma \ref{lem.buseandcartan}, we deduce from \eqref{eqn.cptpart1} that the sequence 
    \be \label{eqn.cptpart2}
    \psi \left(\beta_{\sigma_n^+}^{\theta}(e, \ga_n) \right) + \log \kappa_{t_n} (\sigma_n) \quad \text{is unbounded.}
    \ee

    On the other hand, by the $\Ga$-equivariance of the norms $\|\cdot \|$, we have 
    $$\begin{aligned}
        \kappa_{t_n}(\sigma_n) & = \| v_{\sigma_n} \|_{\varphi_{t_n}\sigma_n} = \left\| v_{\sigma_n} e^{\psi\left(\beta_{\sigma_n^+}^{\theta}(\ga_n, e)\right)} \right\|_{\ga_n^{-1} \varphi_{t_n}\sigma_n} \\
        &  = e^{\psi\left(\beta_{\sigma_n^+}^{\theta}(\ga_n, e)\right)} \left\| v_{\sigma_n}  \right\|_{\ga_n^{-1} \varphi_{t_n}\sigma_n}
    \end{aligned}$$
    and therefore 
    \be \label{eqn.cptpart3}
    \psi\left(\beta_{\sigma_n^+}^{\theta}(e, \ga_n)\right) + \log \kappa_{t_n}(\sigma_n) = 
 \log \left\| v_{\sigma_n}  \right\|_{\ga_n^{-1} \varphi_{t_n}\sigma_n}.
    \ee
    Since both $\sigma_n(0)$ and $ \ga_n^{-1}\sigma_n(t_n) = (\ga_n^{-1}\varphi_{t_n}\sigma_n)(0)$ belong to the compact subset $ Q$ for all $n \ge 1$, there exists a compact subset of $\cal G$ containing $\sigma_n$ and $\ga_n^{-1} \varphi_{t_n}\sigma_n$ for all $n \ge 1$. Therefore, the sequence \eqref{eqn.cptpart3} is uniformly bounded, which contradicts \eqref{eqn.cptpart2}. Hence the claim follows.
\end{proof}

We obtain the following estimate of the contraction rate between the entrance and  exit of a horoball.

\begin{corollary} \label{cor.entexit}
    There exists a constant $c_0 \ge 1$ such that if $\sigma \in \partial^+ \cal G_P$ for some $P \in \cal P^{\Ga}$ with $T_{\sigma}^+ < \infty$, then $$\frac{1}{c_0} e^{- c' T_{\sigma}^+} \le \kappa_{T_{\sigma}^+}(\sigma) \le c_0 e^{-c T_{\sigma}^+}$$ where $c$ and $c'$ are given by Theorem \ref{thm.czzcartan}.
\end{corollary}

\begin{proof}
    Let $P \in \cal P^{\Ga}$ and $\sigma \in \partial^+ \cal G_P$ with $T_{\sigma}^+ < \infty$.
    By Lemma \ref{lem.invariancecocycle}, we may assume that $P \in \cal P$ and $\sigma(0) = (e, 2)$ in the combinatorial horoball attached to a Cayley graph of $P$. We then have $\sigma(T_{\sigma}^+) = (\ga, 2)$ for some $\ga \in P$. Setting $Q = \overline{B_{GM}(e, 1)}$ which is a compact subset of $X_{GM}$, we have $\sigma(0), \ga^{-1}\sigma(T_{\sigma}^+) \in Q$. Hence by Lemma \ref{lem.cptpart}, we have $$ \frac{1}{C_Q} e^{-\psi(\mu_{\theta}(\ga))} \le \kappa_{T_{\sigma}^+}(\sigma) \le C_Q e^{-\psi(\mu_{\theta}(\ga))}$$ where $C_Q$ is the constant therein.
    On the other hand, it follows from Theorem \ref{thm.czzcartan} that 
    $$\begin{aligned}
        \psi(\mu_{\theta}(\ga)) & \ge c d_{GM}(e, \ga) - C \\
        & \ge c (d_{GM}((e, 2), (\ga, 2)) - 2) - C \\
        & = c T_{\sigma}^+ - (2c + C)
    \end{aligned}$$ with the constants $c, C$ in Theorem \ref{thm.czzcartan}. Therefore, we have $$\kappa_{T_{\sigma}^+}(\sigma) \le C_Q e^{2c + C} e^{-c T_{\sigma}^+}.$$ 
    Similarly, we have $$\begin{aligned}
    \psi(\mu_{\theta}(\ga)) & \le c' d_{GM}(e, \ga) + C \\
    & \le c' ( d_{GM}((e, 2), (\ga, 2)) + 2) + C \\
    & = c' T_{\sigma}^+ + (2 c' + C)
    \end{aligned}$$
    where $c'$ is given in Theorem \ref{thm.czzcartan}.
    Therefore, we have $$\kappa_{T_{\sigma}^+}(\sigma) \ge \frac{1}{C_Q} e^{- (2 c' + C)} e^{- c' T_{\sigma}^+}.$$
   This finishes the proof.
\end{proof}

We now estimate the contraction rate in the thin part.

\begin{lemma} \label{lem.cusppart}
    There exists a constant $c_1 \ge 1$ with the following property: if $\sigma \in \cal G_{thin}$ is such that $\varphi_s \sigma \in \cal G_{thin}$ for all $0 \le s \le t$, then $$ c_1^{-1} e^{-(3 c' - 2c) t} \le \kappa_t(\sigma) \le c_1 e^{ - c t}$$
    where $c\le {c}'$ are given by Theorem \ref{thm.czzcartan}.
\end{lemma}

\begin{proof}
    We fix $\sigma \in \cal G_{thin}$ such that $\varphi_s\sigma \in \cal G_{thin}$ for all $0 \le s \le t$. Then there exists $P \in \cal P^{\Ga}$ so that $\varphi_s\sigma \in \cal G_P$ for all $0 \le s \le t$. There are three cases to consider:

    \medskip
    {\bf \noindent Case 1.} Suppose that $\sigma([0, \infty)) \subset \cal G_P$. Then $\sigma = \varphi_s\sigma_0$ for some $\sigma_0 \in \partial^+ \cal G_P$ and $s > 0$. In this case, by the definition of the norm, we have $$\|\cdot \|_{\varphi_t\sigma} = \| \cdot \|_{\varphi_{t + s}\sigma_0} = e^{-c(t + s)} \| \cdot \|_{\sigma_0} = e^{-ct} \| \cdot \|_{\sigma}.$$
    This implies $\kappa_t (\sigma) = e^{-ct}$.

    \medskip
    {\bf \noindent Case 2.} Suppose that $\sigma((-\infty, 0]) \subset \cal G_P$. Then $\sigma = \varphi_s\tilde{\sigma}_0$ for some $\tilde{\sigma}_0 \in \partial^- \cal G_P$ and $s < 0$. We then have $$\| \cdot \|_{\varphi_t\sigma} = e^{-c(s + t)} \| \cdot \|_{\tilde{\sigma}_0} = e^{-c t} \| \cdot \|_{\sigma},$$ and hence $\kappa_t (\sigma) = e^{-ct}$.

    \medskip
    {\bf \noindent Case 3.} Suppose that neither $\sigma([0, \infty)) \subset \cal G_P$ nor $\sigma((-\infty, 0]) \subset \cal G_P$ holds. In this case, we have $\sigma = \varphi_s\sigma_0$ for some $s > 0$ and $\sigma_0 \in \partial^+ \cal G_P$ such that $T_{\sigma_0}^+ < \infty$. We simply write $T := T_{\sigma_0}^+$ and $\sigma_1 = \varphi_T\sigma_0$. We first consider the following three subcases:
    \begin{itemize}
        \item if $s, s + t \in \left(0, \tfrac{1}{3}T \right]$, then $$\| \cdot \|_{\varphi_t\sigma} = \| \cdot \|_{\varphi_{s + t}\sigma_0} = e^{-c(s + t)} \| \cdot \|_{\sigma_0} = e^{-ct} \| \cdot \|_{\sigma},$$
        and hence $\kappa_t (\sigma) = e^{-ct}$.
        \medskip

        \item if $s, s + t \in \left[ \tfrac{2}{3}T, T \right)$, then $$\|\cdot\|_{\varphi_t\sigma} = e^{c(T - (t + s))} \| \cdot \|_{\sigma_1} = e^{-ct} \| \cdot \|_{\sigma},$$
        and hence $\kappa_t(\sigma) = e^{-ct}$.
        \medskip

        \item if $s, s + t \in \left[ \tfrac{1}{3}T, \tfrac{2}{3}T \right]$, then we first observe that $$\begin{aligned}
            \|\cdot \|_{\sigma} & = \| \cdot \|_{\varphi_{T/3}\sigma_0}^{2 - \tfrac{3}{T}s} \| \cdot \|_{\varphi_{2T/3}\sigma_0}^{\tfrac{3}{T} s - 1} \\
            & = \left( e^{-c \tfrac{T}{3}} \| \cdot \|_{\sigma_0} \right)^{2 - \tfrac{3}{T}s} \left( e^{c \tfrac{T}{3}} \| \cdot \|_{\sigma_1} \right)^{\tfrac{3}{T} s - 1} \\
            & = e^{c (2s - T)} \| \cdot \|_{\sigma_0}^{2 - \tfrac{3}{T}s} \| \cdot \|_{\sigma_1}^{\tfrac{3}{T} s - 1}
        \end{aligned}$$
        and similarly that $$\| \cdot \|_{\varphi_t\sigma} = e^{c (2(s + t) - T)} \| \cdot \|_{\sigma_0}^{2 - \tfrac{3}{T}(s+t)} \| \cdot \|_{\sigma_1}^{\tfrac{3}{T} (s + t) - 1}.$$
        Combining the above two computations, we obtain $$\| \cdot \|_{\varphi_t\sigma} = \| \cdot \|_{\sigma} e^{2 c t} \| \cdot \|_{\sigma_0}^{ - \tfrac{3}{T}t} \| \cdot \|_{\sigma_1}^{\tfrac{3}{T} t }.$$
        Evaluating at $v_{\sigma_0}$, the above becomes $$\kappa_{t + s}(\sigma_0) = \kappa_s(\sigma_0) e^{2ct} \kappa_T(\sigma_0)^{\tfrac{3}{T} t}.$$ Since $\kappa_{t + s}(\sigma_0) = \kappa_t(\sigma) \kappa_s(\sigma_0)$ by \eqref{eqn.cocycle}, it follows from Corollary \ref{cor.entexit} and $0 \le t \le \tfrac{T}{3}$ that  $$\begin{aligned}
            \kappa_t(\sigma)  &= e^{2ct} \kappa_T(\sigma_0)^{\tfrac{3}{T}t} \\
            & \le e^{2 c t}( c_0 e^{-cT})^{\tfrac{3}{T}t} =  e^{2ct} c_0^{\tfrac{3}{T}t} e^{-3ct} \\
            & \le \max(1, c_0) e^{-ct}.
            \end{aligned}$$
    \end{itemize}

    Similarly, we also obtain from Corollary \ref{cor.entexit} and $0 \le t \le \frac{T}{3}$ that
    $$\begin{aligned}
        \kappa_t(\sigma) & = e^{2ct} \kappa_T(\sigma_0)^{\frac{3}{T}t} \\
        & \ge e^{2ct} ( c_0^{-1} e^{-c' T})^{\frac{3}{T}t} = e^{2ct} c_0^{\frac{-3}{T} t} e^{-3 c' t} \\
        & \ge \min(1, c_0^{-1}) e^{-(3 c' - 2 c)t}.
    \end{aligned}$$

    We now set $c_1 := \max(1, c_0)$. Note also that $c' \ge c$ and hence $e^{-(3 c' - 2c)t} \le e^{-c t}$ for all $t \ge 0$. In general, we consider the following three consecutive subintervals $$[s, s+t] \cap \left(0, \tfrac{1}{3}T \right], \quad [s, s+t] \cap \left[\tfrac{1}{3}T, \tfrac{2}{3}T \right], \quad \text{and} \quad [s, s+t] \cap \left[\tfrac{2}{3}T, T \right], $$ and then apply the each of the above three subcases to each subintervals. Then by \eqref{eqn.cocycle}, we get $$ c_1^{-1} e^{- (3 c' - 2c)t} \le \kappa_t(\sigma) \le c_1 e^{-ct}$$ as desired.
\end{proof}

We now combine estimates on the thick and thin parts and prove Theorem \ref{thm.expdecay}. We give proofs of the lower bound and the upper bound separately:

\subsection*{Proof of the lower bound in Theorem \ref{thm.expdecay}}
Let $\sigma \in \cal G$ and $t \ge 0$. If $\varphi_s\sigma \in \cal G_{thin}$ for all $0 \le s \le t$, then by Lemma \ref{lem.cusppart}, we have 
\be \label{eqn.nottoomuch0}
\kappa_t(\sigma) \ge c_1^{-1} e^{-(3 c' - 2c) t}
\ee where constants $c_1, c', c$ are given in Lemma \ref{lem.cusppart}.
Now suppose that $\varphi_s\sigma \in \cal G_{thick}$ for some $s \in [0, t]$ and set 
$$
    \begin{aligned}
        s_1 & := \min \{ s \in [0, t] : \varphi_s\sigma \in \cal G_{thick} \}; \\
        s_2 & := \max \{ s \in [0, t] : \varphi_s\sigma \in \cal G_{thick} \}
    \end{aligned}
    $$
    which are well-defined.
    It follows from \eqref{eqn.cocycle} and Lemma \ref{lem.cusppart} that
    \be \label{eqn.nottoomuch1}
    \begin{aligned}
        \kappa_t(\sigma) & = \kappa_{t - s_2} (\varphi_{s_2}\sigma) \kappa_{s_2}(\sigma) \\
        & = \kappa_{t - s_2}(\varphi_{s_2}\sigma)\kappa_{s_2 - s_1}(\varphi_{s_1}\sigma)\kappa_{s_1}(\sigma)\\
        & \ge c_1^{-1} e^{-(3 c' - 2 c)(t - s_2)} \kappa_{s_2 - s_1}(\varphi_{s_1}\sigma)  c_1^{-1} e^{-(3 c' - 2c)s_1} \\
        & = c_1^{-2} e^{-(3 c' - 2c)t} e^{(3 c' - 2c)(s_2 - s_1)} \kappa_{s_2 - s_1} (\varphi_{s_1}\sigma).
    \end{aligned}
    \ee

    To estimate $\kappa_{s_2 - s_1}(\varphi_{s_1}\sigma)$, we fix a compact fundamental domain $Q \subset X_{GM} - \bigcup_{P \in \cal P^{\Ga}} H_P$ for the $\Ga$-action. We may assume that $e \in Q$. By the definition of $s_1$ and $s_2$, there exist $\ga_1, \ga_2 \in \Ga$ such that  $(\varphi_{s_1}\sigma)(0) \in \ga_1 Q$ and $(\varphi_{s_2}\sigma)(0) \in \ga_2 Q$. In other words, we have $(\ga_1^{-1} \varphi_{s_1}\sigma)(0) \in Q$ and $(\ga_1^{-1} \varphi_{s_2}\sigma)(0) \in \ga_1^{-1} \ga_2 Q$. Since $(\ga_1^{-1}\varphi_{s_1}\sigma)(0) = \ga_1^{-1} \sigma(s_1)$ and $(\ga_1^{-1} \varphi_{s_2}\sigma)(0) = \ga_1^{-1} \sigma(s_2)$,
    this implies that for some constant $q > 0$ depending on $Q$, we have
    $
    |d_{GM}(e, \ga_1^{-1} \ga_2) - (s_2 - s_1)| \le q.
    $
    Setting $\ga := \ga_1^{-1} \ga_2$, this is rephrased as 
    \be \label{eqn.nottoomuch2}
    |d_{GM}(e, \ga) - (s_2 - s_1)| \le q.
    \ee
    Moreover, noting that $(\varphi_{s_2}\sigma)(0) = (\varphi_{s_1}\sigma)(s_2 - s_1)$, we have $$(\ga_1^{-1} \varphi_{s_1}\sigma)(0), \ga^{-1}(\ga_1^{-1} \varphi_{s_1}\sigma)(s_2 - s_1) \in Q.$$
    Hence, by Lemma \ref{lem.invariancecocycle} and Lemma \ref{lem.cptpart}, we have
    $$\begin{aligned}
        \kappa_{s_2 - s_1}(\varphi_{s_1}\sigma) & = \kappa_{s_2 - s_1}(\ga_1^{-1} \varphi_{s_1}\sigma) \\
        & \ge \frac{1}{C_Q} e^{-\psi(\mu_{\theta}(\ga))}
    \end{aligned}$$ with the constant $C_Q$ given by Lemma \ref{lem.cptpart}.
    By Theorem \ref{thm.czzcartan} and \eqref{eqn.nottoomuch2}, we deduce
    $$
    \kappa_{s_2 - s_1}(\varphi_{s_1}\sigma) \ge \frac{1}{C_Q} e^{- c' d_{GM}(e, \ga) - C} \ge \frac{e^{- c' q - C}}{C_Q} e^{-c' (s_2 - s_1)}.
    $$
    Together with \eqref{eqn.nottoomuch1}, we have
    \be \label{eqn.nottoomuch3}
    \begin{aligned}
    \kappa_t(\sigma) & \ge c_1^{-2} e^{-(3c' - 2c)t} e^{(3 c' - 2c)(s_2 - s_1)} \frac{e^{- c' q - C}}{C_Q} e^{-c' (s_2 - s_1)}\\
    & =  \frac{c_1^{-2}  e^{- c' q - C}}{C_Q} e^{-(3 c' - 2c)t} e^{(2 c' - 2c)(s_2 - s_1)} \ge \frac{c_1^{-2}  e^{- c' q - C}}{C_Q} e^{-(3 c' - 2c)t} 
    \end{aligned}
    \ee
    where the last inequality is due to $c' \ge c$ and $s_2 \ge s_1$.

  Now note that $a' \ge  3 c' - 2c$ by Theorem \ref{thm.czzcartan} and  choose $b > 1$ such that $b^{-1} \le \min \left(  c_1^{-1}, \frac{ c_1^{-2}  e^{- c' q - C}}{C_Q} \right)$. Then it follows from \eqref{eqn.nottoomuch0} and \eqref{eqn.nottoomuch3} that $$\kappa_t(\sigma) \ge \frac{1}{b} e^{-a' t}$$
    as desired.
\qed

\subsection*{Proof of the upper bound in Theorem \ref{thm.expdecay}}

Let $\sigma \in \cal G$ and $t \ge 0$. If $\varphi_s\sigma \in \cal G_{thin}$ for all $0 \le s \le t$, then by Lemma \ref{lem.cusppart}, we have
\be \label{eqn.contractsimple0}
\kappa_t (\sigma) \le c_1 e^{-ct}
\ee
where $c_1$ and $c$ are constants given in Lemma \ref{lem.cusppart}. We now assume that $\varphi_s\sigma \in \cal G_{thick}$ for some $s \in [0, t]$. As in the proof of the lower bound, we set 
$$
    \begin{aligned}
        s_1 & := \min \{ s \in [0, t] : \varphi_s\sigma \in \cal G_{thick} \}; \\
        s_2 & := \max \{ s \in [0, t] : \varphi_s\sigma \in \cal G_{thick} \}
    \end{aligned}
    $$
We then have from \eqref{eqn.cocycle} and Lemma \ref{lem.cusppart} that
    \be \label{eqn.contractsimple1}
    \begin{aligned}
        \kappa_t(\sigma) & = \kappa_{t - s_2}(\varphi_{s_2}\sigma)\kappa_{s_2 - s_1}(\varphi_{s_1}\sigma)\kappa_{s_1}(\sigma)\\
        & \le c_1^2 e^{-ct} e^{c(s_2 - s_1)} \kappa_{s_2 - s_1}(\varphi_{s_1}\sigma).
    \end{aligned}
    \ee
    By the similar argument as in the proof of the lower bound, we have $$\kappa_{s_2 - s_1}(\varphi_{s_1}\sigma) \le C_Q e^{-\psi(\mu_{\theta}(\ga))}$$
    where $Q \subset X_{GM} - \bigcup_{P \in \cal P^{\Ga}} H_P$ is a compact fundamental domain for the $\Ga$-action, $C_Q$ is the constant given by Lemma \ref{lem.cptpart}, and $\ga \in \Ga$ is such that $|d_{GM}(e, \ga) - (s_2 - s_1) | \le q$ for some constant $q \ge 0$ depending only on $Q$. By Theorem \ref{thm.czzcartan}, this implies $$\kappa_{s_2 - s_1}(\varphi_{s_1}\sigma) \le C_Q e^{- c d_{GM}(e, \ga) + C} \le C_Q e^{cq + C} e^{-c(s_2 - s_1)}$$
    with the constant $C$ therein.
    Plugging this into \eqref{eqn.contractsimple1}, we have 
    \be \label{eqn.contractsimple2}
    \kappa_t(\sigma) \le c_1^2 C_Q e^{cq + C} e^{-ct}.
    \ee
    We then choose $b \ge \max\left( c_1, c_1^2 C_Q e^{cq + C} \right)$. By \eqref{eqn.contractsimple0} and \eqref{eqn.contractsimple2}, we finally obtain $$\kappa_t(\sigma) \le b e^{-ct}.$$
   Since $a=c$ by Theorem \ref{thm.czzcartan}, this
   completes the proof.
    \qed

\subsection*{Proof of Theorem \ref{thm.repar}}
As described above, we define the $\Ga$-equivariant continuous section $u : \cal G \to \cal G \times \R_{+}$ by setting $u(\sigma) = (\sigma, v_{\sigma})$, and set $\bar \Psi = \Psi_0 \circ u$ so that we have the following commutative diagram: $$\begin{tikzcd}
    \cal G \times \R_{+} \arrow[rd, "\Psi_0"] \arrow[d] & \\
    \cal G \arrow[r, dashed, "\tilde{\Psi}"'] \arrow[u, bend left=50, dashed, "u"] & \tilde{\Omega}_{\psi}
\end{tikzcd}$$
In other words, $\tilde{\Psi}(\sigma) = (\sigma^+, \sigma^-, \log v_{\sigma})$. 

We first prove that $\tilde \Psi$ is proper, from which the properness of $\Psi$ follows.
Suppose not. Then there exists a sequence $\sigma_n \in \cal G$ such that $\sigma_n$ escapes every compact subset of $\cal G$ as $n \to \infty$ while $\tilde{\Psi}(\sigma_n) = (\sigma_n^+, \sigma_n^-, \log v_{\sigma_n})$ converges in $\tilde{\Omega}_{\psi}$. Since the sequence $(\sigma_n^+, \sigma_n^-)$ converges in $\La_{\theta}^{(2)}$, two sequences $\sigma_n^+$ and $\sigma_n^-$ converge to two distinct points in $\partial X_{GM}$. This implies that there exist a sequence $t_n \in \R$ and a compact subset $Q \subset \cal G$ so that $\varphi_{t_n}\sigma_n \in Q$ for all $n \ge 1$. Moreover, since the sequence $\tilde{\Psi}(\sigma_n) = (\sigma_n^+, \sigma_n^-, \log v_{\sigma_n})$ converges in $\tilde{\Omega}_{\psi}$, the sequence $v_{\sigma_n}$ converges in $\R_{+}$. This implies that, after passing to a subsequence,
\be \label{eqn.proper1}
\text{the sequence } \|v_{\sigma_n} \|_{\varphi_{t_n}\sigma_n} \text{ converges to a positive number.}
\ee

On the other hand, since the sequence $\sigma_n$ escapes any compact subset of $\cal G$ as $n \to \infty$, we have either $t_n \to \infty$ or $t_n \to -\infty$ as $n \to \infty$, after passing to a subsequence. Suppose first that $t_n \to \infty$ as $n \to \infty$.  It follows from Theorem \ref{thm.expdecay} that for all sufficiently large $n\ge 1$, $$\|v_{\sigma_n}\|_{\varphi_{t_n}\sigma_n} = \kappa_{t_n}(\sigma_n) \le b e^{-at_n} \to 0 \quad \text{as } n \to \infty.$$
This contradicts \eqref{eqn.proper1}.
We now assume that $t_n \to -\infty$ as $n \to \infty$.
Then for all sufficiently large $n\ge 1$, we have $$\|v_{\sigma_n}\|_{\varphi_{t_n}\sigma_n} = \frac{1}{\|v_{\varphi_{t_n}\sigma_n}\|_{\sigma_n}} = \frac{1}{\kappa_{-t_n}(\varphi_{t_n}\sigma_n)} \ge b^{-1} e^{-a t_n}$$ by \eqref{eqn.cocycle} and Theorem \ref{thm.expdecay}. Therefore,  $\|v_{\sigma_n}\|_{\varphi_{t_n}\sigma_n} \to \infty$ as $n \to \infty$, contradicting \eqref{eqn.proper1}. This proves the properness.

We now prove items (1), (2), and (3).
Since the $\Ga$-action on $\cal G$ and $\tilde \Omega_{\psi}$ commute with flows on $\cal G$ and $\tilde \Omega_{\psi}$, it suffices to prove the statement for $\tilde \Psi : \cal G \to \tilde \Omega_{\psi}$.
For $(\sigma,s)\in \cal G\times \br\to \br$, define a continuous function
$$\tilde{\mathsf t}(\sigma, s) := \log v_{\varphi_s\sigma} -\log  v_{\sigma}.$$
By \eqref{eqn.cocycle}, we have $$v_{\varphi_s\sigma} = \frac{v_{\sigma}}{\| v_{\sigma} \|_{\varphi_s\sigma}} = \frac{v_{\sigma}}{\kappa_s(\sigma)}.$$
Therefore
\be \label{eqn.reparspeed}
\tilde{\mathsf{t}}(\sigma, s) = - \log \kappa_s (\sigma),
\ee
 is $\Ga$-invariant (Lemma \ref{lem.invariancecocycle}) and hence induces a continuous map $\mathsf{t} : \Ga \ba \cal G \times \R \to \R$. The cocycle property of $\tilde {\mathsf t}$ follows from \eqref{eqn.cocycle}. By the definition of $\tilde \Psi$,
we have $$
\tilde \Psi (\varphi_s \sigma) =  \phi_{\tilde{\mathsf t}(\sigma, s) }\Psi(\sigma),
$$ from which (1) follows. 
This also implies (2), noting that $$
\phi_{-\tilde{\mathsf t}(\sigma, s)} \tilde \Psi (\varphi_s \sigma) = \tilde \Psi(\sigma) = \tilde \Psi(\varphi_{-s} \varphi_{s} \sigma) = \phi_{\tilde{\mathsf t}(\varphi_s \sigma, -s)}(\varphi_s \sigma).
$$
Moreover, by Theorem \ref{thm.expdecay} and \eqref{eqn.reparspeed}, we have that for all $s \ge 0$, 
\be \label{eqn.ineqitem3}
as - \log b \le \tilde{\mathsf{t}}(\sigma, s) \le a's + \log b
\ee
where $a, a' > 0$ and $b \ge 1$ are given in Theorem \ref{thm.expdecay}. This shows (3).

To see the surjectivity of $\Psi$, note first that for each $(\xi, \eta, t_0) \in \tilde \Omega_{\psi}$, there exists $\sigma \in \cal G$ with $\sigma^+ = \xi$ and $\sigma^- = \eta$ as $X_{GM}$ is a proper geodesic Gromov hyperbolic space. For $s_0 \ge 0$, it follows from \eqref{eqn.ineqitem3} that 
$$
    \tilde{\mathsf{t}}(\sigma, s_0)  \ge as_0 - \log b \quad \text{and} \quad 
    \tilde{\mathsf{t}}(\varphi_{-s_0} \sigma, s_0) \ge as_0 - \log b.
$$
Since $\tilde{\mathsf{t}}(\varphi_{-s_0} \sigma, s_0) = - \tilde{\mathsf{t}}(\sigma, -s_0)$ due to the cocycle property \eqref{eqn.cocycle}, we have $$
    \tilde{\mathsf{t}}(\sigma, s_0)  \ge as_0 - \log b \quad \text{and} \quad 
    \tilde{\mathsf{t}}(\sigma, -s_0) \le - as_0 + \log b.
$$
Since $\tilde \Psi$ is continuous, this implies that the image of $\tilde \Psi$ restricted on $\{ \varphi_{s} \sigma : -s_0 \le s \le s_0\}$ contains $\{ \phi_t \tilde \Psi(\sigma) : - as_0 + \log b \le t \le as_0 - \log b \}$. 
Since $\sigma^+ = \xi$ and $\sigma^- = \eta$, $\tilde \Psi(\sigma) = (\xi, \eta, t_1)$ for some $t_1 \in \R$. We then take $s_0$ large enough so that $$
-a s_0 + \log b + t_1 \le t_0 \le a s_0 - \log b + t_1.
$$ Then $(\xi, \eta, t_0) \in \{ \phi_t \tilde \Psi(\sigma) : - as_0 + \log b \le t \le as_0 - \log b \}$, and hence $(\xi, \eta, t_0)$ belongs to the image of $\tilde \Psi$. Therefore, $\tilde \Psi$ is surjective. This completes the proof.
\qed

\section{Uniformity of fibers of reparameterization}

Recall the reparameterization $\tilde \Psi : \cal G \to \tilde \Omega_{\psi}$ constructed in section \ref{sec.reparconstruct}. The main goal of this section is to establish a uniform bound on the diameters of the fibers of $\tilde \Psi$:

\begin{theorem}[Theorem \ref{thm.repar0}(4)] \label{thm.uniformpreimage}
    The fibers of $\tilde \Psi$ have  uniformly bounded diameter. That is, there exists $C > 0$ such that for any $\sigma, \sigma' \in \cal G$,
    $$
    \tilde \Psi(\sigma) = \tilde \Psi(\sigma') \Longrightarrow d_{GM}(\sigma(0), \sigma'(0)) < C.
    $$
\end{theorem}

We prove this result by analyzing the explicit form of our reparameterization. For $\sigma \in \cal G$,
$$\tilde \Psi(\sigma) = (\sigma^+, \sigma^-, \log v_{\sigma} )$$
where $v_{\sigma} \in \R_{+}$ is the unit vector with respect to the norm $\| \cdot \|_{\sigma}$, as constructed in section \ref{sec.reparconstruct}. Thus, Theorem \ref{thm.uniformpreimage} follows from the next proposition:

\begin{proposition} \label{prop.nearpoint}
    There exists a constant $C_0 > 0$ such that the following holds: for any $\sigma, \sigma' \in \cal G$ with $\sigma^{\pm} = \sigma'^{\pm}$, there exists $s \in \R$ such that
    $$d_{GM}(\sigma(0), \sigma'(s)) < C_0 \quad \text{and} \quad | \log v_{\sigma} - \log v_{\varphi_{s} \sigma'} | < C_0.$$
    Moreover, the shift parameter $s$ satisfies:
    \begin{itemize}
    \item if $s \ge 0$, then
    $$
    \frac{ (\log v_{\sigma} - \log v_{\sigma'}) - C_0 - B}{a'} \le s \le \frac{ (\log v_{\sigma} - \log v_{\sigma'}) + C_0 + B}{a}.
    $$

    \item if $s < 0$, then
    $$
    \frac{ (\log v_{\sigma} - \log v_{\sigma'}) - C_0 - B}{a} \le s \le \frac{ (\log v_{\sigma} - \log v_{\sigma'}) + C_0 + B}{a'}.
$$
\end{itemize}
Here $0 < a < a'$ and $B > 0$ are the constants appearing in Theorem \ref{thm.repar}.
\end{proposition}

To prove Proposition \ref{prop.nearpoint}, we require several preparatory lemmas. We begin by recalling  the definition of the {\it Gromov product} on $X_{GM} \cup \partial X_{GM}$. For $x, y, z \in X_{GM}$, define
$$
(y | z )_x := \frac{1}{2} ( d_{GM}(x, y) + d_{GM}(x, z) - d_{GM}(y, z)).
$$
For $y, z \in X_{GM} \cup \partial X_{GM}$, define $$
( y | z )_x := \sup \liminf_{i, j \to \infty} ( y_i | z_j )_x
$$
where the supremum is taken over all sequences $y_i, z_j \in X_{GM}$ converging to $y, z$, respectively. By the Gromov hyperbolicity of $X_{GM}$ (Theorem \ref{thm.GM_relhyp}), the Gromov product $( y | z )_x$ estimates the distance from $x$ to a geodesic $[y, z]$, up to a uniformly bounded additive error. 

\begin{lemma}\label{nogeo}
Let $\sigma_n \in \cal G$ be a sequence such that
$\{\sigma_n(0) \in X_{GM} :n\ge 1 \} $ is uniformly bounded.
    Then there do not exist sequences $T_n, S_n > 0$ tending to $\infty$ such that
   both $\sigma_n(T_n)$ and $ \sigma_n(-S_n) $ lie in the same horoball
    $ \overline{H_P}$ for some $P\in \cal P$.
\end{lemma}
\begin{proof}
Suppose such sequences exist. Then since $\sigma_n^{\pm}$ belong to the shadows $  O_{1}^{GM}(\sigma_n(0), \sigma_n(T_n))$ and $\sigma_n^- \in O_{1}^{GM}(\sigma_n(0), \sigma_n(-S_n))$, and  $\{ \sigma_n(0) : n \ge 1\}$ is bounded,
we must have  $\lim_{n \to \infty} \sigma_n^{\pm}=\xi_P$.  On the other hand,
the boundedness of $\sigma_n(0)$ implies that $\{\sigma_n\}$ is relatively compact, yielding a contradiction.
\end{proof}

 It is a standard fact in the Gromov hyperbolic geometry (cf. \cite[Theorem III.H.1.7]{Bridson1999metric}) that there exists a constant $c_0>0$ such that any two geodesics with same endpoints have Hausdorff distance at most $c_0$.

\begin{lemma} \label{lem.deepsegment}
    There exists $T_{\mathsf h}' > 0$ such that for each $P\in \cal P^\Ga$ and $\sigma \in \partial^+ \cal G_P$  with $T_{\sigma}^+ > 3 T_{\sf h}'$,  the $c_0$-neighborhood of the segment $ \sigma([T_{\sf h}', T_{\sigma}^+ - T_{\sf h}'])$ is entirely contained in $H_P$.
\end{lemma}

\begin{proof} 
    Suppose not. Since $\cal P$ is finite,  there exist $P \in \cal P$ and sequences $\sigma_n \in \partial^+ \cal G_P$ with $T_{\sigma_n}^+ > 3 n$ and $t_n \in [n, T_{\sigma_n}^+ - n]$ such that $\sigma_n(t_n)$ is not contained in the $c_0$-neighborhood of $H_P$. Hence there exists $p_n \in P$ such that $d_{GM}(\sigma_n(t_n), (p_n, 2)) < c_0$. Replacing $\sigma_n$ with $p_n^{-1} \sigma_n$, we may assume that $p_n = e$, so $\sigma_n(t_n)$ lies in a fixed bounded neighborhood of $(e, 2)$.
Applying Lemma \ref{nogeo} to $\varphi_{t_n} \sigma_n$  with $T_n = T_{\sigma_n}^+ - t_n$ and $S_n = t_n$ yields a contradiction.
\end{proof}

\begin{lemma} \label{lem.forever}
    There exists $\tilde T > 0$ such that for any $P \in \cal P^{\Ga}$ and $\sigma \in \partial^+ \cal G_P$ with $\sigma^+ = \xi_P$, we have $ \sigma(t) \in H_P$ for all $t > \tilde T$.
\end{lemma}

\begin{proof}
    Suppose not. As in the proof of Lemma \ref{lem.deepsegment}, for some $P \in \cal P$, there exist $\sigma_n \in \partial^+ \cal G_P$ with  $\sigma_n^+ = \xi_P$ and $t_n > n$ such that $\sigma_n(t_n) = (e, 2)$.
    Since $\sigma_n^+=\xi_P$, there exist $T_n>n+t_n$ such that $\sigma_n(T_n)\in H_P$ and $\sigma_n(0)\in \partial H_P$. Applying Lemma \ref{nogeo} to $\varphi_{t_n}\sigma_n$  gives a contradiction.
\end{proof}

Let $T_{\sf h}', \tilde T > 0$ be constants given in Lemma \ref{lem.deepsegment} and Lemma \ref{lem.forever} respectively.
\begin{lemma} \label{lem.horoballtraveling}
    There exists $T_{\sf h} > T_{\sf h}' + \tilde T + c_0 + 2$ with the following property: let $P \in \cal P^{\Ga}$, $\sigma \in \partial^+ \cal G_P$ with $T_{\sigma}^+ > 5 T_{\sf h}$, and $t \in [2T_{\sf h}, T_{\sigma}^+ - 2T_{\sf h}]$. Suppose $\sigma' \in \partial^+ \cal G_P$ satisfies
    $\sigma'^{\pm} = \sigma^{\pm}$ and 
    $d_{GM}(\sigma'([0, T_{\sigma'}^+]), \sigma(t)) < c_0$. Then 
    \begin{enumerate}
        \item $d_{GM}(\sigma(0), \sigma'(0)) < T_{\sf h};$
\item $T_{\sigma}^+ < \infty$ if and only if $T_{\sigma'}^+ < \infty$, and in this case, $$
    d_{GM}(\sigma(T_{\sigma}^+), \sigma'(T_{\sigma'}^+)) < T_{\sf h}.
    $$ \end{enumerate}
\end{lemma}

\begin{proof} Suppose that there exist $P \in \cal P$, $\sigma_n, \sigma_n' \in \partial^+ \cal G_P$ with $T_{\sigma_n}^+ > 5n$ and $\sigma_n(0) = (e, 2)$, $\sigma_n^{\pm} = \sigma_n'^{\pm}$, and $t_n \in [2n, T_{\sigma_n}^+ - 2n]$, $s_n \in [0, T_{\sigma_n'}^+]$ such that $$d_{GM}(\sigma_n(t_n), \sigma_n'(s_n)) < c_0 \quad \text{and} \quad d_{GM}(\sigma_n(0), \sigma_n'(0)) > n.$$ 

Since $\sigma_n(t_n)\in H_P$, $\sigma_n(0) = (e, 2)$, and $d_{GM}(\sigma_n(t_n), \sigma_n(0)) = t_n \to \infty$,
we have $\sigma_n(t_n) \to \xi_P$ as $n \to \infty$.  Write $\sigma_n'(0) = (p_n, 2)$ with $p_n \in P$. We claim that 
\be \label{eqn.claim1205}
d_{GM}(\sigma_n(t_n), \sigma_n'(0)) \to \infty;
\ee
if not, the sequence $p_n^{-1}\sigma_n(t_n)$ is contained in a fixed compact subset. Since $p_n^{-1} \sigma_n(T_{\sigma_n}^+), p_n^{-1} \sigma_n(0) \in \partial H_P$ and $T_{\sigma_n}^+ - t_n, t_n \to + \infty$, this contradicts Lemma \ref{nogeo}.
    
    Let $s_n' \in \R$ be such that $d_{GM}(\sigma_n(0), \sigma_n'(s_n')) < c_0$, which exists by the Gromov hyperbolicity. 
  
   We now divide the argument into two cases:

   \smallskip
   
\noindent{\bf Case 1: $s_n' \ge  0$ for infinitely many $n$.}
    Then the Gromov product $( \sigma_n'(0) | \sigma_n'^+ )_{\sigma_n(0)}$ is uniformly bounded, passing to a subsequence. Since $\sigma_n(0) = (e, 2)$, it follows that after passing to a subsequence, $\sigma_n'(0)\to \xi$ and $\sigma_n'^+\to \xi'$ with $\xi\ne \xi'$. But $\sigma_n'(0) = (p_n, 2)$ with $p_n\in P$, and since $d_{GM}(\sigma_n(0), \sigma_n'(0)) > n$, we conclude that
     $p_n\to \infty$ in $P$, hence $\sigma_n'(0) \to \xi_P$. On the other hand, since $\sigma_n'^+ = \sigma_n^+ \in O_1^{GM}((e, 2), \sigma_n(T_{\sigma_n}^+))$ and $\sigma_n(T_{\sigma_n}^+) = (q_n, 2)$ with $q_n \to \infty$ in $P$, it follows from Lemma \ref{lem.shadowgiveslimit} that $\sigma_n'^+ \to \xi_P$, This contradicts the distinctness $\xi\ne \xi'$.

     \smallskip

   \noindent{\bf Case 2:  $s_n' < 0$ for all but finitely many $n \ge 1$.}
   In this case, two geodesic segments $\sigma_n([0, t_n])$ and $\sigma_n'([s_n', s_n])$ have $c_0$-close endpoints. Hence, by Gromov hyperbolicity, there exists $t_n' \in [0, t_n]$ such that $\sigma_n(t_n')$ is uniformly close to $\sigma_n'(0)$. This implies that
     the Gromov product $( \sigma_n(0) | \sigma_n(t_n) )_{\sigma_n'(0)} = ( p_n^{-1} \sigma_n(0) | p_n^{-1}\sigma_n (t_n) )_{p_n^{-1}\sigma_n'(0)}$ is uniformly bounded. It follows from $p_n\to \infty$
    that  $p_n^{-1} \sigma_n(0) = (p_n^{-1}, 2)$ converges to $\xi_P$, after passing to a subsequence. 
     Since $p_n^{-1}\sigma_n'(0) = (e, 2)$, $p_n^{-1} \sigma_n(t_n)$ must converge to a point distinct from $\xi_P$. On the other hand, we have $p_n^{-1}\sigma_n(t_n) \in H_P$, and from \eqref{eqn.claim1205}, we know it diverges from $(e,2)$, thus converging to  $\xi_P$ again, which is a contradiction.

    \smallskip
    
    Now let $T_{\sf h} > 0$ be the constant obtained from the first part. Let $P \in \cal P$,  $\sigma \in \partial^+ \cal G_P$ with $T_{\sigma}^+ > 5 T_{\sf h}$, and $t \in [2T_{\sf h}, T_{\sigma}^+ - 2T_{\sf h}]$. Let $\sigma' \in \partial^+ \cal G_P$ satisfy $\sigma'^{\pm} = \sigma^{\pm}$, and suppose that there exists  $s \in [0, T_{\sigma'}^+]$ such that $d_{GM}(\sigma'(s), \sigma(t)) < c_0$. If $\sigma^+ = \sigma'^+ \neq \xi_P$, then both $T_{\sigma}^+$ and $T_{\sigma'}^+$ are finite. So it suffices to consider the case where $\sigma^+ = \sigma'^+ = \xi_P$. Since $T_{\sigma}^+ > 5 T_{\sf h} > \tilde T$, Lemma \ref{lem.forever} implies $T_{\sigma}^+ = \infty$.
    By the first part, we have $d_{GM}(\sigma(0), \sigma'(0)) < T_{\sf h}$, and  since $t > 2 T_{\sf h}$, we have $T_{\sigma'}^+ \ge s > t - T_{\sf h} - c_0 >  T_{\sf h} - c_0 > \tilde T,$  so Lemma \ref{lem.forever} again implies
    $T_{\sigma'}^+ = \infty$.  Finally, when $T_{\sigma}^+ < \infty$, and hence $T_{\sigma'}^+ < \infty$, we can apply the same argument to the time-reversed geodesics of $\varphi_{T_{\sigma}^+}\sigma$ and $\varphi_{T_{\sigma'}^+} \sigma'$, completing the proof. 
\end{proof}

\subsection*{Proof of Proposition \ref{prop.nearpoint}}
    Fix two geodesics $\sigma, \sigma' \in \cal G$ with the same endpoints $\sigma^{\pm} = \sigma'^{\pm}$. Since the norm $\| \cdot \|_{\sigma}$ used to define $v_\sigma$ depends on the position of $\sigma(0)$, we divide the proof into cases based on the geometry of $\sigma(0)$.

    \medskip

    {\bf \noindent Case 1.} Suppose that $\sigma(0)$ lies within $5T_{\sf h}$-neighborhood of the Cayley graph of $\Ga$ in $X_{GM}$. That is, $d_{GM}(\Ga, \sigma(0)) < 5 T_{\sf h}$. By the definition of $c_0 > 0$, 
    we can find $s \in \R$ so that 
    $$
    d_{GM}(\sigma(0), \sigma'(s)) < c_0.
    $$

   Let $\ga \in \Ga$ be such that $d_{GM}(\ga \sigma(0), e) < 5 T_{\sf h}$.
   Then  both $\ga \sigma(0)$ and $\ga \sigma'(s)$ lie in the $(5T_{\sf h} + c_0)$-neighborhood of the identity. Hence the shifted geodesics $\ga \sigma$ and $\ga \varphi_s\sigma' = \varphi_s \ga \sigma'$ lie in a uniformly compact subset of $\cal G$. Therefore, there exists a uniform constant $C_1 > 0$ such that
    $$
    | \log v_{\ga \sigma} - \log v_{\ga \varphi_s \sigma'} | < C_1.
    $$
   By the equivariance formula for $v_{\ga \sigma}$ (see \eqref{eqn.vgammasigma}),
   we have 
    $$
    \begin{aligned}
        \log v_{\ga \sigma} & = \log v_{\sigma} + \psi(\beta_{\sigma^+}^{\theta}(\ga^{-1}, e)) \\
        \log v_{\ga \varphi_s \sigma'} & = \log v_{\varphi_s \sigma'} + \psi(\beta_{\sigma'^+}^{\theta}(\ga^{-1}, e)).
    \end{aligned}
    $$
    Since $\sigma^+ = \sigma'^+$, the Busemann maps in both expressions coincide and we conclude
    $$
    | \log v_{\sigma} - \log v_{\varphi_s \sigma'} | < C_1.
    $$ Choosing $C_0 > \max ( c_0, C_1)$ completes the proof in this case.

    \medskip
    {\bf \noindent Case 2.} Suppose that $d_{GM}(\Ga, \sigma(0)) > 5 T_{\sf h}$, $\sigma(0) \in H_P$, and $\sigma^+ = \xi_P$ for some $P \in \cal P^{\Ga}$. 
    In this case, we can write $\sigma = \varphi_t \sigma_0$ for some $\sigma_0 \in \partial^+ \cal G_P$ and $t > 0$. By hypothesis, $t > 5 T_{\sf h} > \tilde T$, and hence $T_{\sigma_0}^+ = \infty$ by Lemma \ref{lem.forever}. Then
    \be \label{eqn.case2_norm}
    \| \cdot \|_{\sigma} = e^{-ct} \| \cdot \|_{\sigma_0}
    \ee
    where $c > 0$ is the constant defined in \eqref{eqn.defofc}.
    
    By the definition of $c_0 > 0$,
    there exists $s \in \R$ such that $d_{GM}(\sigma'(s), \sigma(0)) < c_0$. Since $t > 5 T_{\sf h} > T_{\sf h}'$ and $T_{\sigma_0}^+ = \infty$,  Lemma \ref{lem.deepsegment} implies $\sigma'(s) \in H_P$.  So we may write  $\varphi_s \sigma' = \varphi_{t'} \sigma_0'$ for some $\sigma_0' \in \partial^+ \cal G_P$ and $t' > 0$.
    Applying Lemma \ref{lem.horoballtraveling} to $\sigma_0$ and $\sigma_0'$, we obtain 
    \be \label{eqn.closestart1}
    d_{GM}(\sigma_0(0), \sigma_0'(0)) < T_{\sf h} \quad \text{and} \quad T_{\sigma_0'}^+ = \infty.
    \ee This gives 
    \be \label{eqn.case2_norm2}
    \| \cdot \|_{\varphi_s \sigma'} = e^{-c t'} \| \cdot \|_{\sigma_0'}.
    \ee
    Combining \eqref{eqn.case2_norm} and \eqref{eqn.case2_norm2}, we compute:
    $$
    \begin{aligned}
       \log v_{\sigma} & = ct + \log v_{\sigma_0} \\
       \log v_{\varphi_s \sigma'} & = c t' + \log v_{\sigma_0'}.
    \end{aligned}
    $$
    Hence it suffices to bound $|t - t'|$ and $|\log v_{\sigma_0} - \log v_{\sigma_0'}|$. First, 
    $$
    \begin{aligned}
        t & = d_{GM}(\sigma_0(0), \sigma(0)) \\
        & \le d_{GM}(\sigma_0(0), \sigma_0'(0)) + d_{GM}(\sigma_0'(0), \sigma'(s)) + d_{GM}(\sigma'(s), \sigma(0))  < T_{\sf h} + t' + c_0.
    \end{aligned}$$
    Similarly, $t' < T_{\sf h} + t + c_0$, and hence $$|t - t'| < T_{\sf h} + c_0.$$

    Since $\sigma_0, \sigma_0' \in \partial^+ \cal G_P$ and their basepoints $\sigma_0(0)$ and $\sigma_0'(0)$ lie in the $2$-neighborhood of the Cayley graph of $\Ga$, with distance less than  $T_{\sf h}$  by \eqref{eqn.closestart1},
    we may apply Case 1 to $\sigma_0, \sigma_0'$ to obtain 
$$
    | \log v_{\sigma_0} - \log v_{\sigma_0'} | < C_2
    $$
    for some uniform constant $C_2 > 0$.
    Therefore,
    $$
    |\log v_{\sigma} - \log v_{\varphi_s \sigma'}| <  c(T_{\sf h} + c_0) + C_2.
    $$
    Taking $C_0 > \max( c_0, c(T_{\sf h} + c_0) + C_2)$ verifies the claim in this case.

    \medskip
    {\bf \noindent Case 3.} Suppose $d_{GM}(\Ga, \sigma(0)) > 5 T_{\sf h}$, $\sigma(0) \in H_P$ and $\sigma^- = \xi_P$ for some $P \in \cal P^{\Ga}$. In this case, we apply Lemma \ref{lem.forever} to the time reversal of $\sigma$, obtaining 
    $\sigma = \varphi_t \tilde \sigma_0$ for some $\tilde \sigma_0 \in \partial^- \cal G_P$ with $T_{\tilde \sigma_0}^- = - \infty$ and $t < 0$. The norm $\| \cdot \|_{\sigma}$ is given by
    $$\| \cdot \|_{\sigma} = e^{-c t} \| \cdot \|_{\tilde \sigma_0}.$$
   This case is symmetric to  Case 2 and follows by the same argument, which we omit.

    \medskip
    {\bf \noindent Case 4.} Suppose that none of Cases 1-3 applies. Then for some $P \in \cal P^{\Ga}$, $\sigma_0 \in \partial^+ \cal G_P$ with finite $T := T_{\sigma_0}^+ < \infty$, and some $t \in [5 T_{\sf h}, T - 5 T_{\sf h}]$, we have $\sigma = \varphi_t \sigma_0$. In particular, $T > 5 T_{\sf h}$ and $t \in [2 T_{\sf h}, T - 2 T_{\sf h}]$. We may assume that $P \in \cal P$ and $\sigma_0(0) = (e, 2)$.

    By definition of $c_0 > 0$, 
    there exists $s'' \in \R$ such that 
     \be \label{eqn.case4_travel}
     d_{GM}(\sigma(0), \sigma'(s'')) < c_0.
    \ee
    By Lemma \ref{lem.deepsegment}, $\sigma'(s'') \in H_P$, and hence
    \be \label{eqn.case4_proj}
    \varphi_{s''}\sigma' = \varphi_{t''} \sigma_0'\quad\text{ for some $t'' > 0$ and $\sigma_0' \in \partial^+ \cal G_P$}.
    \ee
    
    By Lemma \ref{lem.horoballtraveling}, we have $T' := T_{\sigma_0'}^+ < \infty$ and 
    \be \label{eqn.case4_proj2}
    \quad d_{GM}(\sigma_0(0), \sigma_0'(0)) < T_{\sf h} \quad \text{and} \quad d_{GM}(\sigma_0(T), \sigma_0'(T')) < T_{\sf h}.
    \ee
    In particular,
    \be \label{eqn.case4_time}
    |T - T'| < 2 T_{\sf h}.
    \ee
  Since all points $\sigma_0(0)$, $\sigma_0'(0)$, $\sigma_0(T)$, and $\sigma_0'(T')$ lie  in the $2$-neighborhood of the Cayley graph of $\Ga$, we may apply the argument of Case 1 to $\sigma_0$ and $\sigma_0'$ to obtain a uniform constant $C_3 > 0$ such that
    \be \label{eqn.case4_norm}
    | \log v_{\sigma_0} - \log v_{\sigma'_0} | < C_3 \quad \text{and} \quad | \log v_{\varphi_T \sigma_0} - \log v_{\varphi_{T'} \sigma'_0} | < C_3
    \ee

    As the norm $\| \cdot \|_{\sigma}$ is defined according to the time parameter $t$, we now proceed to subcases depending on how $t$ compares the ends of the segment $[0, T]$. 

    \medskip
    {\bf \noindent Case 4-1.} Suppose that $0 < t \le T/3$. By \eqref{eqn.case4_travel}, \eqref{eqn.case4_proj}, \eqref{eqn.case4_proj2}, and \eqref{eqn.case4_time}, we have
    $$-  (T_{\sf h} + c_0) < t - (T_{\sf h} + c_0) \le t'' \le t + (T_{\sf h} + c_0) \le \frac{T}{3} +  (T_{\sf h} + c_0) < \frac{T'}{3} +  (2T_{\sf h} + c_0).$$
    Hence, we can take $t' \in (t'' - (2 T_{\sf h} + c_0) , t'' + (T_{\sf h} + c_0))$ so that
    $$
    0 < t' < \frac{T'}{3}.
    $$
    This implies
    \be \label{eqn.case41_time}
    |t - t'| < 3 T_{\sf h} + 2 c_0 \quad \text{and}
    \ee
    \be \label{eqn.case41_dist}
    d_{GM}(\sigma(0), \sigma'_0(t')) \le d_{GM}(\sigma(0), \sigma'(s'')) + d_{GM}(\sigma_0'(t''), \sigma_0'(t')) < 2(T_{\sf h} + c_0)
    \ee
    where the last inequality follows from \eqref{eqn.case4_travel} and $|t' - t''| < 2T_{\sf h} + c_0$.
   From the construction, we have
    $$
    \| \cdot \|_{\sigma} = e^{-ct} \| \cdot \|_{\sigma_0} \quad \text{and} \quad \| \cdot \|_{\varphi_{t'}\sigma_0'} = e^{-c t'} \| \cdot \|_{\sigma_0'}.
    $$
and hence
    $$
        \log v_{\sigma} = ct + \log v_{\sigma_0} \quad \text{and} \quad  \log v_{\varphi_{t'} \sigma_0'} = ct' + \log v_{\sigma_0'}. 
    $$
   Hence, using \eqref{eqn.case41_time} and \eqref{eqn.case4_norm}, we deduce
    $$
    | \log v_{\sigma} - \log v_{\varphi_{t'}\sigma_0'} | < c(3 T_{\sf h} + 2c_0) + C_3.
    $$
  Since $\varphi_s \sigma' = \varphi_{t'} \sigma_0'$ for some $s\in \br$,
we conclude the claim in this case hold with
 $C_0 > \max ( 2(T_{\sf h} + c_0), c(3 T_{\sf h} + 2 c_0) + C_3)$.

    \medskip
    {\bf \noindent Case 4-2.} Suppose that $2T/3 \le t < T$. In this case, the norm $\| \cdot \|_{\sigma}$ is given by
    $$
    \| \cdot \|_{\sigma} = e^{c(T - t)} \| \cdot \|_{\varphi_T \sigma_0}.
    $$
This case is symmetric to Case 4-1  and follows by the same argument
using $T-t$ in place of $t$, together with \eqref{eqn.case4_time}.
We omit the details.

    \medskip
    {\bf \noindent Case 4-3.} Suppose that $T/3 < t < 2T/3$. Then from the same bounds \eqref{eqn.case4_travel}, \eqref{eqn.case4_proj}, \eqref{eqn.case4_proj2}, and \eqref{eqn.case4_time}, 
    $$
\begin{aligned}
    \frac{T'}{3} & - (2T_{\sf h} + c_0) < \frac{T}{3} - (T_{\sf h} + c_0)< t - (T_{\sf h} + c_0) \le t'' \\
    & \le t + (T_{\sf h} + c_0) \le \frac{2T}{3} + (T_{\sf h} + c_0) < \frac{2T'}{3} + (3T_{\sf h} + c_0).
    \end{aligned}
    $$
    Hence we can find $t' \in (t'' - (3 T_{\sf h} + c_0), t'' + (2 T_{\sf h} + c_0))$ so that
    $$
    \frac{T'}{3} < t' < \frac{2 T'}{3}.
    $$
  This gives
    \be \label{eqn.case43_time}
    |t - t'| < 4 T_{\sf h} + 2 c_0 \quad \text{and}
    \ee
    \be \label{eqn.case43_dist}
    d_{GM}(\sigma(0), \sigma'_0(t')) \le d_{GM}(\sigma(0), \sigma'(s'')) + d_{GM}(\sigma_0'(t''), \sigma_0'(t')) < 3 T_{\sf h} + 2 c_0
    \ee
   using again \eqref{eqn.case4_travel}  and $|t' - t''| < 3 T_{\sf h} + c_0$.

Now using the interpolation formula for the norm, we get
    $$
    \begin{aligned}
        \| \cdot \|_{\sigma} & = \| \cdot \|_{\varphi_{T/3} \sigma_0}^{2 - \tfrac{3}{T} t} \| \cdot \|_{\varphi_{2T/3} \sigma_0}^{\tfrac{3}{T} t - 1} = e^{c (2t - T)} \| \cdot \|_{\sigma_0}^{2 - \tfrac{3}{T} t} \| \cdot \|_{\varphi_T \sigma_0}^{\tfrac{3}{T} t - 1} \\
        \| \cdot \|_{\varphi_{t'} \sigma_0'} & = \| \cdot \|_{\varphi_{T'/3} \sigma_0'}^{2 - \tfrac{3}{T'} t'} \| \cdot \|_{\varphi_{2T'/3} \sigma_0'}^{\tfrac{3}{T'} t' - 1} = e^{c (2t' - T')} \| \cdot \|_{\sigma_0'}^{2 - \tfrac{3}{T'} t'} \| \cdot \|_{\varphi_{T'} \sigma_0'}^{\tfrac{3}{T'} t' - 1}.
    \end{aligned}
    $$
    Therefore,
    $$
    \begin{aligned}
    \log v_{\sigma} & = cT - 2ct + \left( 2 - \frac{3}{T} t \right) \log v_{\sigma_0} + \left( \frac{3}{T} t - 1 \right) \log v_{\varphi_T \sigma_0} \\
    & = cT - 2ct + 2 \log v_{\sigma_0} - \log v_{\varphi_T \sigma_0} + \frac{3t}{T} \left( \log v_{\varphi_T \sigma_0} - \log v_{\sigma_0}\right) \\
    \log v_{\varphi_{t'}\sigma_0'} & =  cT' - 2ct' + 2 \log v_{\sigma_0'} - \log v_{\varphi_{T'} \sigma_0'} + \frac{3t'}{T'} \left( \log v_{\varphi_{T'} \sigma_0'} - \log v_{\sigma_0'}\right).
    \end{aligned}
    $$
    Now using the triangle inequality, \eqref{eqn.case4_time}, \eqref{eqn.case43_time}, \eqref{eqn.case4_norm}, and the fact that $t' < 2T'/3$, we estimate
    $$
    \begin{aligned}
       & | \log v_{\sigma} - \log v_{\varphi_{t'} \sigma_0'} | \\
       & \le 2 c T_{\sf h} + 2c(4T_{\sf h} + 2c_0) + 2 C_3 + C_3 +\left| \frac{3t}{T} - \frac{3t'}{T'} \right| | \log v_{\varphi_T \sigma_0} - \log v_{\sigma_0} | \\  & \quad + \frac{3t'}{T'}| \log v_{\varphi_{T'} \sigma_0'} - \log v_{\varphi_T \sigma_0}| +  \frac{3t'}{T'}| \log v_{\sigma_0'} - \log v_{\sigma_0}| \\
        & \le 2c(5T_{\sf h} + 2c_0) + 3 C_3 + \left| \frac{3t}{T} - \frac{3t'}{T'} \right| | \log v_{\varphi_T \sigma_0} - \log v_{\sigma_0} |  + 4 C_3.
    \end{aligned}
    $$

    Now recall that $\sigma_0(0) = (e, 2)$ as noted earlier, and denote  $\sigma_0(T) = (\ga, 2)$ for some $\ga \in P$. Let $Q \subset X_{GM}$ denote the closed $2$-ball centered at $e$. Then  $\sigma_0(0) \in Q$ and $\sigma_0(T) \in \ga Q$. From \eqref{kappat} and \eqref{eqn.cocycle}, we have $v_{\varphi_T \sigma_0} = \frac{v_{\sigma_0}}{\kappa_T (\sigma_0)}$. In particular, 
    $$\log v_{\varphi_T \sigma_0} - \log v_{\sigma_0} = - \log \kappa_T (\sigma_0).$$
   By Lemma \ref{lem.cptpart}, there exists  $c_Q > 0$ depending only on $Q$, such that  $$
    |\log v_{\varphi_T \sigma_0} - \log v_{\sigma_0} - \psi(\mu_{\theta}(\ga)) | < c_Q.
    $$
    Therefore, 
    $$
    \begin{aligned}
    & | \log v_{\sigma} - \log v_{\varphi_{t'} \sigma_0'} | \\
    & \le 2c(5T_{\sf h} + 2c_0) + 7 C_3 + \left| \frac{3t}{T} - \frac{3t'}{T'} \right|  c_Q + \left| \frac{3t}{T} - \frac{3t'}{T'} \right| | \psi(\mu_{\theta}(\ga)) | \\
    &  \le 2c(5T_{\sf h} + 2c_0) + 7 C_3 + c_Q + \left| \frac{3t}{T} - \frac{3t'}{T'} \right| | \psi(\mu_{\theta}(\ga)) |
    \end{aligned}
    $$
    where the last inequality is from $\frac{T}{3} < t < \frac{2T}{3}$ and $\frac{T'}{3} < t' < \frac{2T'}{3}$.
 Estimate the final term:
    $$
    \begin{aligned}
    \left| \frac{3t}{T} - \frac{3t'}{T'} \right| | \psi(\mu_{\theta}(\ga)) | &\le \frac{3|t - t'|}{T} |\psi(\mu_{\theta}(\ga))| + 3 t' \left| \frac{1}{T} - \frac{1}{T'} \right| | \psi(\mu_{\theta}(\ga))| \\
    & \le \frac{3|t - t'|}{T} |\psi(\mu_{\theta}(\ga))| + 3 t' \frac{|T - T'|}{T'T} | \psi(\mu_{\theta}(\ga))| \\
    & \le (16 T_{\sf h} + 6c_0) \left| \frac{\psi(\mu_{\theta}(\ga))}{T} \right|,
    \end{aligned}
    $$
    using \eqref{eqn.case43_time}, \eqref{eqn.case4_time}, and $t' < 2T'/3$.
    It follows from $T = d_{GM}(\sigma_0(0), \sigma_0(T)) = d_{GM}((e, 2), (\ga, 2))$ that
    $$|T - d_{GM}(e, \ga)| \le 2.$$
Then, by Theorem \ref{thm.czzcartan}, there exist uniform  constants
$c_1, c_2 > 1$ such that
    $$
    c_1^{-1} T - c_2 \le \psi(\mu_{\theta}(\ga)) \le c_1 T + c_2.
    $$
    Since $T > T_{\sf h}$, we conclude:
    $$\left| \frac{\psi(\mu_{\theta}(\ga))}{T} \right| \le c_1 + \frac{c_2}{T_{\sf h}}.$$
    Combining all altogether,
    $$
    | \log v_{\sigma} - \log v_{\varphi_{t'} \sigma_0'} | \le 2c(5 T_{\sf h} + 2 c_0) + 7 C_3 + c_Q + (16 T_{\sf h} + 6 c_0) ( c_1 + c_2/T_{\sf h}).
    $$
  Since
    $\varphi_s \sigma' = \varphi_{t'} \sigma_0'$ for some $s\in \br$, and using
    \eqref{eqn.case43_dist}, the claim follows by setting $$C_0 > 2c(5 T_{\sf h} + 2 c_0) + 7 C_3 + c_Q + (16 T_{\sf h} + 6 c_0) ( c_1 + c_2/T_{\sf h}) > 3 T_{\sf h} + 2c_0.$$ This completes the proof of the first part of Proposition \ref{prop.nearpoint}.

    \medskip

    We now prove the second assertion. Let $C_0 > 0$ be the constant from the first part and let $\sigma, \sigma' \in \cal G$ be such that $\sigma^{\pm} = \sigma'^{\pm}$. Then for some $s \in \R$, we have
    $$d_{GM}(\sigma(0), \sigma'(s)) < C_0 \quad \text{and} \quad | \log v_{\sigma} - \log v_{\varphi_s \sigma'} | < C_0.$$
    Therefore,
    $$
    \log v_{\sigma} - \log v_{\sigma'} - C_0 < \log v_{\varphi_s \sigma'} - \log v_{\sigma'} < \log v_{\sigma} - \log v_{\sigma'} + C_0.
    $$
  Now,  from Theorem \ref{thm.repar}, we have
$$
\tilde \Psi(\varphi_s \sigma') = \phi_t \tilde \Psi(\sigma')
$$
for some $t$ with $a s - B \le t \le a' s + B$ if $s \ge 0$ and $a' s - B \le t \le a s + B$ if $s < 0$, where $0 < a < a'$ and $B > 0$ are constants in the theorem. Since
$$
\log v_{\varphi_s \sigma'} = t + \log v_{\sigma'},
$$
we deduce the bounds on $s$ as follows
\begin{itemize}
    \item if $s \ge 0$,
    $$
    \frac{ \log v_{\varphi_s \sigma'} - \log v_{\sigma'} - B}{a'} \le s \le \frac{ \log v_{\varphi_s \sigma'} - \log v_{\sigma'} + B}{a}.
    $$
    Therefore,
    $$
    \frac{ (\log v_{\sigma} - \log v_{\sigma'}) - C_0 - B}{a'} \le s \le \frac{ (\log v_{\sigma} - \log v_{\sigma'}) + C_0 + B}{a}.
    $$

    \item if $s < 0$,
    $$
    \frac{ \log v_{\varphi_s \sigma'} - \log v_{\sigma'} - B}{a} \le s \le \frac{ \log v_{\varphi_s \sigma'} - \log v_{\sigma'} + B}{a'}.
    $$
    Therefore,
    $$
    \frac{ (\log v_{\sigma} - \log v_{\sigma'}) - C_0 - B}{a} \le s \le \frac{ (\log v_{\sigma} - \log v_{\sigma'}) + C_0 + B}{a'}.
$$
\end{itemize}
This completes the proof.
\qed

\subsection*{Proof of Theorem \ref{thm.uniformpreimage}} Let $\sigma, \sigma' \in \cal G$ be such that $$\tilde \Psi(\sigma) = \tilde \Psi(\sigma').$$
This implies that $\sigma^{\pm} = \sigma'^{\pm}$ and $\log v_{\sigma} - \log v_{\sigma'} = 0$. By Proposition \ref{prop.nearpoint}, there exist uniform constants $a, B, C_0 > 0$ so that
$$
d_{GM}(\sigma(0), \sigma'(s)) < C_0 \text{ for some } s \in \left[ - \frac{C_0 + B}{a} , \frac{C_0 + B}{a} \right].
$$
Therefore,
$$d_{GM}(\sigma(0), \sigma'(0)) \le d_{GM}(\sigma(0), \sigma'(s)) + d_{GM}(\sigma'(s), \sigma'(0)) < C_0 + \frac{C_0 + B}{a}. $$
This finishes the proof.
\qed

\subsection*{Disjointness of $\tilde \Psi$-images of horoballs}
We deduce from Theorem \ref{thm.uniformpreimage} that $\tilde \Psi$-images of deep horoballs are disjoint. This implies that the reparameterization $\tilde \Psi : \cal G \to \tilde \Omega_{\psi}$ and $\Psi : \Ga \ba \cal G \to \Omega_{\psi}$ respectively give genuine decompositions of $\tilde \Omega_{\psi}$ and $\Omega_{\psi}$ into the non-cuspidal part and disjoint cuspidal components. 

To be precise, for each $n \ge 2$, we define the {\it depth-$n$ horoballs}, similar to the definition of open horoballs $H_P$, as follows: for $P \in \cal P$, let $H_P'(n) \subset X_{GM}$ be the subgraph induced by the vertices $\{(g, k) : g \in P, k \ge n \}$ and $\hat H_P(n) \subset X_{GM}$ be the subgraph induced by the vertices $\{(g, n) : g \in P\}$. We then set
$$H_P(n) := H_P' - \hat H_P.$$
For $\ga \in \Ga$, we set
$$H_{\ga P \ga^{-1}}(n) := \ga H_P(n).$$
This results in the collection of depth-$n$ open horoballs $\{ H_P(n) : P \in \cal P^{\Ga} \}$. Note that $H_P = H_P(2)$ for $P \in \cal P^{\Ga}$. For $P \in \cal P^{\Ga}$, we consider the set
$$\cal G_P (n) := \{ \sigma \in \cal G : \sigma(0) \in H_P(n) \}$$
which consists of bi-infinite geodesics based at $H_P(n)$. We now obtain the following disjointness:
    \begin{corollary} \label{cor.disjointimagehoroballs}
        There exists $n_0 \ge 2$ such that for $P, P' \in \cal P^{\Ga}$,
        $$
        P \neq P' \Longrightarrow \tilde \Psi(\cal G_P(n_0)) \cap \tilde \Psi(\cal G_{P'}(n_0)) = \emptyset.
        $$
    \end{corollary}

    \begin{proof}
        Let $C > 0$ be the constant given by Theorem \ref{thm.uniformpreimage}. We fix $n_0 > \frac{C}{2} + 1$ and show that the desired disjointness holds. Suppose on the contrary that for some distinct $P, P' \in \cal P^{\Ga}$, there exist $\sigma \in \cal G_P (n_0)$ and $\cal G_{P'} (n_0)$ such that $\tilde \Psi (\sigma) = \tilde \Psi(\sigma')$. Since $\sigma(0) \in H_P(n_0)$, the distance from $\sigma(0)$ to the Cayley graph of $\Ga$ is at least $n_0 - 1$. Similarly, the distance from $\sigma'(0)$ to the Cayley graph of $\Ga$ is at least $n_0 - 1$. Since two basepoints $\sigma(0)$ and $\sigma'(0)$ are contained in distinct horoballs, a geodesic segment between them must pass through the Cayley graph. Therefore, we have
        $$d_{GM}(\sigma(0), \sigma'(0)) \ge 2n_0 - 2 > C.$$
        On the other hand, since $\tilde \Psi(\sigma) = \tilde \Psi(\sigma')$, we have $d_{GM}(\sigma(0), \sigma'(0)) < C$ by Theorem \ref{thm.uniformpreimage}, which is a contradiction. This shows the claim.
    \end{proof}

\begin{remark} By the above corollary,
the reparameterization given in Corollary \ref{cor.reparquotient} gives us a thick-thin decomposition of $\Omega_\psi$ 
where the thin part is the disjoint union of $\Psi$-images of bi-infinite geodesics based at the  horoballs in $\Ga \ba X_{GM}$ corresponding to elements of $\cal P$.
\end{remark}

\section{Exponential expansion on unstable foliations}
Let $\Ga<G$ be a $\theta$-Anosov subgroup relative to $\cal P$.
Fix a $(\Ga,\theta)$-proper linear form
$\psi \in \fa_{\theta}^*$.  Recall the space $\tilde \Omega_{\psi} = \La_{\theta}^{(2)} \times \R$ equipped with the $\Ga$-action given by
$$\ga (\xi, \eta, s) = (\ga \xi, \ga \eta, s + \psi(\beta_{\xi}^{\theta}(\ga^{-1}, e)))$$
for $\ga\in \Ga$ and $(\xi, \eta, s)\in \La_{\theta}^{(2)}\times \br$, and
  $\Omega_\psi=\Ga\ba \tilde \Omega_{\psi}$ as defined in section \ref{sec.bundle}.
Recall from \eqref{ww} and \eqref{ww2} the unstable and stable foliations $W^{\pm}$ on $\Omega_{\psi}$ and their lifts $\tilde W^{\pm}$ on $\tilde{\Omega}_{\psi}$.
The goal of this section is to establish the following exponential expansion (resp. contraction) property of the flow $\{\phi_t\}$ on unstable (resp. stable) foliations.
\begin{theorem} \label{thm.nicemetrics}
We have the following:
\begin{enumerate}
    \item 
There exist a $\Ga$-invariant non-negative symmetric function $d^+:\tilde\Omega_{\psi} \times \tilde \Omega_{\psi}\to \br $ and  constants $\alpha, \alpha' > 0$ and $b \ge 1$ 
such that for $z \in \tilde \Omega_{\psi}$,
the restriction of $d^{+}$ defines a semi-metric\footnote{A semi-metric on $\cal X$ is a non-negative symmetric function $\cal X \times \cal X \to \R$ that vanishes precisely on the diagonal.} on $\tilde W^{+}(z)$ and
   for any $w_1, w_2 \in \tilde W^{+}(z)$ and $t \ge 0$,
    $$\begin{aligned}
    \frac{1}{b} e^{\alpha t} d^+(w_1, w_2) & \le d^+(\phi_t w_1, \phi_t w_2) \le b e^{\alpha' t} d^+(w_1, w_2).\\ 
    \end{aligned}
    $$

   \item  Similarly, there exists a $\Ga$-invariant non-negative symmetric function $d^-:\tilde\Omega_{\psi} \times \tilde \Omega_{\psi}\to \br$ such that for $z \in \tilde \Omega_{\psi}$,
the restriction of $d^{-}$ defines a semi-metric on $\tilde W^{-}(z)$ and
   for any $w_1, w_2 \in \tilde W^{-}(z)$ and $t \ge 0$,
   $$
    \frac{1}{b} e^{-\alpha' t} d^-(w_1, w_2) \le d^-(\phi_t w_1, \phi_t w_1) \le b e^{-\alpha t} d^-(w_1, w_2).
   $$
\item For any small enough $\varepsilon > 0$, there exists a non-negative symmetric function $d_{\varepsilon}^+:\tilde\Omega_{\psi} \times \tilde \Omega_{\psi}\to \br$ such that for $z \in \tilde \Omega_{\psi}$, the restriction of $d_{\varepsilon}^+$ defines a metric on $\tilde W^+(z)$. Moreover, for any compact subset $Q \subset \tilde{\Omega}_{\psi}$, there exists a constant $c_Q \ge 1$ such that for any $w_1, w_2 \in Q$, $$\frac{1}{c_Q} d^+(w_1, w_2)^{\varepsilon} \le d_{\varepsilon}^+(w_1, w_2) \le c_Q d^+(w_1, w_2)^{\varepsilon}.$$

   \end{enumerate}

\end{theorem}

\begin{remark}
    Even though Theorem \ref{thm.nicemetrics} states the exponential expansion and contraction for $t \ge 0$, replacing $w_1$ and $w_2$ with $\phi_{-t}w_1$ and $\phi_{-t} w_2$ implies the corresponding estimates for negative-time flow.
\end{remark}

The proof of Theorem \ref{thm.nicemetrics} is based on
our coarse reparameterization (Theorem \ref{thm.repar}) and the coarse geometry of the Groves-Manning cusp space as a Gromov hyperbolic space.

\subsection*{Groves-Manning cusp space as a Gromov hyperbolic space} 
Let $X_{GM}$ be the associated Groves-Manning cusp space of $(\Ga, \cal P)$, which is a proper geodesic Gromov hyperbolic space (\cite[Theorem 3.25]{GM_relhyp}, Theorem \ref{thm.GM_relhyp}). We refer to \cite[Chapter III.H]{Bridson1999metric}
for general facts about Gromov hyperbolic spaces.

Recall that $\cal G$ is the space of all parameterized bi-infinite geodesics in $X_{GM}$.
 We define $d^{\pm}:\cal G\times \cal G\to [0, \infty)$ as follows: for $\sigma_1, \sigma_2 \in \cal G$, 
    \be \label{eqn.defmetricsgm}
    \begin{aligned}
        d^+(\sigma_1, \sigma_2) & := \limsup_{t \to \infty} e^{d_{GM}(\sigma_1(t), \sigma_2(t)) - 2t}; \\
        d^-(\sigma_1, \sigma_2) & := \limsup_{t \to \infty} e^{d_{GM}(\sigma_1(-t), \sigma_2(-t)) - 2t}.
    \end{aligned}
    \ee
    Their well-definedness follows once we explain another formula for $d^{\pm}$ using Gromov products and Busemann functions on $X_{GM}$.
    We recall that 
for $ x, p, q \in X_{GM}$, the Gromov product of $p, q$ with respect to $x$ is  $$( p | q)_x := \frac{1}{2} (d_{GM}(x, p) + d_{GM}(x, q) - d_{GM}(p, q))\ge 0,$$ and this extends to $\partial X_{GM}$ as follows: for $\xi, \eta \in \partial X_{GM}$, we set $$(\xi | \eta)_x := \sup \liminf_{i, j \to \infty} ( p_i | q_j )_x$$
where the supremum is taken over all sequences $p_i, q_j \in X_{GM}$ such that $p_i \to \xi$ and $q_j \to \eta$ as $i, j \to \infty$. Since $X_{GM}$ is Gromov hyperbolic, there exists a uniform constant $\delta > 0$ such that for any $x\in X_{GM}$,  $\xi, \eta \in \partial X_{GM}$, and sequences $p_i, q_j \in X_{GM}$ with $\xi = \lim_{i \to \infty} p_i$ and $\eta = \lim_{j \to \infty} q_j$, we have 
\be \label{eqn.gromovineq}
(\xi | \eta)_x - \frac{\delta}{2} \le \liminf_{i, j \to \infty} (p_i | q_j)_x \le (\xi | \eta)_x.
\ee
For $\sigma \in \cal G$ and $p, q \in X_{GM}$, the following {\it Busemann function } is well-defined: $$
    \beta_{\sigma^+}(p, q)  := \lim_{t \to \infty} d_{GM}(p, \sigma(t)) - d_{GM}(q, \sigma(t)).
    $$ 
    We note that the Busemann function is defined for each geodesic $\sigma \in \cal G$, not for a point in $\partial X_{GM}$. The notation $+$ in $\beta_{\sigma^+}(p, q)$ is to indicate that the limit is taken along $t \to \infty$. Indeed, this makes the above limit well-defined since the function $f_p : \R \to \R$ defined as 
    $$
    f_p(t) = d_{GM}(p, \sigma(t)) - d_{GM}(\sigma(0), \sigma(t))
    $$
    is non-increasing and bounded from above by $d_{GM}(p, \sigma(0))$, and we have $d_{GM}(p, \sigma(t)) - d_{GM}(q, \sigma(t)) = f_p(t) - f_q(t)$.

   We have  for any $x \in X_{GM}$ that
      \be \label{eqn.anotherform}
      d^+(\sigma_1, \sigma_2)  = e^{\beta_{\sigma_1^+}(x, \sigma_1(0)) + \beta_{\sigma_2^+}(x, \sigma_2(0))}\limsup_{t \to \infty} e^{- 2 (\sigma_1(t) | \sigma_2(t))_x}.
      \ee
Since $(\sigma_1(t) | \sigma_2(t))_x\ge 0$ for all $t$, it follows that $d^+(\sigma_1, \sigma_2)<\infty$. Since 
    \be \label{eqn.minusisplusandinv}
    d^-(\sigma_1, \sigma_2) = d^+(I \sigma_1, I \sigma_2),
    \ee
  $d^-$ is well-defined as well. The definition of $d^{\pm}$ is motivated by the 
   Hamenst\"adt distance  in a negatively curved compact manifold \cite{Ham_bms}.   
    
    Since $\Ga$ acts on $X_{GM}$ by isometries, both $d^+$ and $d^-$ are $\Ga$-invariant.
The geodesic flow on $\cal G$ exponentially expand and contract $d^{+}$ and $d^-$ respectively:
    \begin{lemma} \label{lem.expcont1}
        Let $\sigma_1, \sigma_2 \in \cal G$ and $s_1, s_2 \in \R$. Then we have
        $$\begin{aligned}
        e^{-\delta} e^{s_1 + s_2} d^+(\sigma_1, \sigma_2) &\le d^+(\varphi_{s_1}\sigma_1, \varphi_{s_2}\sigma_2) \le e^{\delta} e^{s_1 + s_2} d^+(\sigma_1, \sigma_2);\\
        e^{-\delta} e^{ -(s_1 + s_2)} d^-(\sigma_1, \sigma_2) &\le d^-(\varphi_{s_1}\sigma_1, \varphi_{s_2}\sigma_2) \le e^{\delta} e^{-(s_1 + s_2)} d^-(\sigma_1, \sigma_2).
        \end{aligned}$$
    \end{lemma}
    
    \begin{proof}
    Fix $x \in X_{GM}$.
    By \eqref{eqn.anotherform} and \eqref{eqn.gromovineq}, we have
    \be \label{eqn.buseandgrom}
    \begin{aligned}
        d^+(\sigma_1, \sigma_2) & \ge e^{\beta_{\sigma_1^+}(x, \sigma_1(0)) + \beta_{\sigma_2^+}(x, \sigma_2(0))} e^{- 2 (\sigma_1^+ | \sigma_2^+)_x}; \\
        d^+(\sigma_1, \sigma_2)
        & \le e^{\delta} e^{\beta_{\sigma_1^+}(x, \sigma_1(0)) + \beta_{\sigma_2^+}(x, \sigma_2(0))} e^{- 2 (\sigma_1^+ | \sigma_2^+)_x}.
    \end{aligned}
    \ee
    By the definition of $\beta$, we have 
    \be \label{eqn.buseco1}
    \begin{aligned}
    \beta_{\sigma_1^+}(x, (\varphi_{s_1}\sigma_1)(0)) & = \beta_{\sigma_1^+}(x, \sigma_1(0)) + \beta_{\sigma_1^+}(\sigma_1(0), \sigma_1(s_1)) \\
    & = \beta_{\sigma_1^+}(x, \sigma_1(0)) + s_1,
    \end{aligned}
    \ee
    and similarly 
    \be \label{eqn.buseco2}
    \beta_{\sigma_2^+}(x, (\varphi_{s_2}\sigma_2)(0))  = \beta_{\sigma_2^+}(x, \sigma_2(0)) + s_2.
    \ee
    Since $\varphi_{s_1}\sigma_1^+ = \sigma_1^+$ and $\varphi_{s_2}\sigma_2^+ = \sigma_2^+$, it follows from \eqref{eqn.buseandgrom}, \eqref{eqn.buseco1}, and \eqref{eqn.buseco2} that
    $$e^{-\delta} e^{s_1 + s_2} d^+(\sigma_1, \sigma_2) \le d^+(\varphi_{s_1}\sigma_1, \varphi_{s_2}\sigma_2) \le e^{\delta} e^{s_1 + s_2} d^+(\sigma_1, \sigma_2).$$
    The exponential contraction of $d^-$ follows from the exponential expansion of $d^+$ shown above and \eqref{eqn.minusisplusandinv}.
    \end{proof}

    We fix a basepoint $x \in X_{GM}$. It is a standard fact about Gromov hyperbolic spaces that for $\varepsilon > 0$ small enough, there exists $0 <c_\e < 1$ and a metric $d_{\varepsilon}$ on $\partial X_{GM}$ such that 
    \be \label{eqn.visualmetric}
    c_\e e^{-2\varepsilon(\xi | \eta)_x} \le d_{\varepsilon}(\xi, \eta) \le e^{- 2 \varepsilon (\xi | \eta)_x}
    \ee
for all $\xi, \eta \in \partial X_{GM}$, with the convention that $e^{-\infty} = 0$ \cite[Proposition 3.21]{Bridson1999metric}. We fix one such $\varepsilon > 0$ and a metric $d_{\varepsilon}$ as above.

\begin{lemma} \label{lem.compactvisual0}
    For any compact subset $Q \subset \cal G$, there exists a constant $b_Q \ge 1$ such that for any $\sigma_1, \sigma_2 \in Q$, we have $$\frac{1}{b_Q} d^+(\sigma_1, \sigma_2)^{\varepsilon} \le d_{\varepsilon}(\sigma_1^+, \sigma_2^+) \le b_Q d^+(\sigma_1, \sigma_2)^{\varepsilon}.$$
\end{lemma}

\begin{proof}
    First note that for any $\sigma \in \cal G$, $$|\beta_{\sigma^+}(x, \sigma(0))| \le d_{GM}(x, \sigma(0)).$$
    Given a compact subset $Q \subset \cal G$, we set $$b' := \sup_{\sigma \in Q} d_{GM}(x, \sigma(0))<\infty .$$ Then it follows from \eqref{eqn.visualmetric} and \eqref{eqn.buseandgrom} that
    $$\begin{aligned}
        d_{\varepsilon}(\sigma_1^+, \sigma_2^+) & \le e^{-\varepsilon \left( \beta_{\sigma_1^+}(x, \sigma_1(0)) + \beta_{\sigma_2^+}(x, \sigma_2(0)) \right) } d^{+}(\sigma_1, \sigma_2)^{\varepsilon} \\
        & \le e^{2 \varepsilon b'} d^+(\sigma_1, \sigma_2)^{\varepsilon}.
    \end{aligned}$$
    Similarly, we also have $$d_{\varepsilon}(\sigma_1^+, \sigma_2^+) \ge c_\varepsilon e^{-\varepsilon(\delta + 2 b')} d^+(\sigma_1, \sigma_2)^{\varepsilon}$$
    where $0 < c_\varepsilon < 1$ is given in \eqref{eqn.visualmetric}.
    Setting $b_Q := e^{\varepsilon(\delta + 2 b')} / c_\varepsilon$ completes the proof.
\end{proof}

\subsection*{Reparameterization revisited}
Recall the reparameterization $\Psi : \Ga \ba \cal G \to \Omega_{\psi}$ in Theorem \ref{thm.repar}, which is induced from the $\Ga$-equivariant map $\tilde \Psi : \cal G \to \tilde \Omega_{\psi}$. Since $\tilde{\Psi}$ is proper and surjective, for $w_1, w_2  \in \tilde{\Omega}_{\psi}$, we define 
\be \label{eqn.defofnicemetrics}
\begin{aligned}
d^+(w_1, w_2) & := \sup_{\sigma_1 \in \tilde{\Psi}^{-1}(w_1), \ \sigma_2 \in \tilde{\Psi}^{-1}(w_2)} d^+(\sigma_1, \sigma_2); \\
d^-(w_1, w_2) & := \sup_{\sigma_1 \in \tilde{\Psi}^{-1}(w_1), \ \sigma_2 \in \tilde{\Psi}^{-1}(w_2)} d^-(\sigma_1, \sigma_2).
\end{aligned}
\ee
Since $\tilde{\Psi}$ is $\Ga$-equivariant, if $\sigma_1 \in \tilde{\Psi}^{-1}(w_1)$ and $\sigma_2 \in \tilde{\Psi}^{-1}(w_2)$, then $\ga \sigma_1 \in \tilde{\Psi}^{-1}(\ga w_1)$ and $\ga \sigma_2 \in \tilde{\Psi}^{-1}(\ga w_2)$ for all $\ga \in \Ga$. Since $d^{\pm}(\ga \sigma_1, \ga \sigma_2) = d^{\pm}(\sigma_1, \sigma_2)$ as well, we have 
\be \label{eqn.Gainvofmetric}
d^{\pm}(\ga w_1, \ga w_2) = d^{\pm}(w_1, w_2) \quad \text{for all } \ga \in \Ga.
\ee
We also have the following expansion and contraction of $d^+$ and $d^-$ via the flow $\{\phi_t\}$ respectively: 
\begin{lemma} \label{lem.expcontmain}
There exist $\alpha, \alpha' > 0$ and $b \ge 1$ such that for any $w_1, w_2 \in \tilde{\Omega}_{\psi}$ and $t \ge 0$, we have
\be \label{eqn.expconc1}
    \begin{aligned}
    \frac{1}{b} e^{\alpha t} d^+(w_1, w_2) & \le d^+(\phi_t w_1, \phi_t w_2) \le b e^{\alpha' t} d^+(w_1, w_2);\\ 
    \frac{1}{b} e^{-\alpha' t} d^-(w_1, w_2) &\le d^-(\phi_t w_1, \phi_t w_2) \le b e^{-\alpha t} d^-(w_1, w_2).
    \end{aligned}
    \ee
\end{lemma}

\begin{proof}
    Let $w_1, w_2 \in \tilde{\Omega}_{\psi}$ and $t \ge 0$. Let $\sigma_1 \in \tilde{\Psi}^{-1}(w_1)$ and $\sigma_2 \in \tilde{\Psi}^{-1}(w_2)$. By Theorem \ref{thm.repar}, there exist $s_1, s_2 \in \R$ such that $$\varphi_{s_1}\sigma_1 \in \tilde{\Psi}^{-1}(\phi_t w_1) \quad \text{and} \quad \varphi_{s_2}\sigma_2 \in \tilde{\Psi}^{-1}(\phi_t w_2),$$ and moreover, for constants $a, a' , B > 0$ in  Theorem \ref{thm.repar}, we have:
    \begin{enumerate}
        \item if $s_1 \ge 0$, then $$as_1 - B \le t \le a's_1 + B$$
        (resp. if $s_2 \ge 0$, then $as_2 - B \le t \le a's_2 + B$).
        \item if $s_1 \le 0$, then $$a's_1 - B \le t \le as_1 + B$$
        (resp. if $s_2 \le 0$, then $a's_2 - B \le t \le as_2 + B$).
    \end{enumerate}
    By Lemma \ref{lem.expcont1}, we have 
    \be \label{eqn.appexpcont1}
    e^{-\delta} e^{s_1 + s_2} d^+(\sigma_1, \sigma_2) \le d^+(\varphi_{s_1}\sigma_1, \varphi_{s_2}\sigma_2) \le e^{\delta} e^{s_1 + s_2} d^+(\sigma_1, \sigma_2).
    \ee
    
    Suppose first that $s_1, s_2 \ge 0$. Then by (1) above, we deduce from \eqref{eqn.appexpcont1} that
    $$
    d^+(\varphi_{s_1}\sigma_1, \varphi_{s_2}\sigma_2) \le e^{\delta} e^{\frac{2B}{a}} e^{\frac{2t}{a}} d^+(\sigma_1, \sigma_2) \le e^{\delta} e^{\frac{2B}{a}} e^{\frac{2t}{a}} d^+(w_1, w_2).
    $$
    Since $\sigma_1 \in \tilde{\Psi}^{-1}(w_1)$ and $\sigma_2 \in \tilde{\Psi}^{-1}(w_2)$ are arbitrary, $\varphi_{s_1}\sigma_1$ and $\varphi_{s_2}\sigma_2$ are arbitrary elements of $\tilde{\Psi}^{-1}(\phi_t w_1)$ and $\tilde{\Psi}^{-1}(\phi_t w_2)$ respectively. Hence we have
    \be \label{eqn.expineqlarge}
    d^+(\phi_t w_1, \phi_t w_2) \le e^{\delta} e^{\frac{2B}{a}} e^{\frac{2t}{a}} d^+(w_1, w_2).
    \ee
    Similarly, we deduce from (1) and \eqref{eqn.appexpcont1} that
    $$d^+(\phi_t w_1, \phi_t w_2) \ge d^+(\varphi_{s_1} \sigma_1, \varphi_{s_2} \sigma_2) \ge e^{-\delta} e^{-\frac{2B}{a'}} e^{\frac{2t}{a'}} d^+(\sigma_1, \sigma_2).$$
    Since $\sigma_1 \in \tilde{\Psi}^{-1}(w_1)$ and $\sigma_2 \in \tilde{\Psi}^{-1}(w_2)$ are arbitrary, we have 
    \be \label{eqn.expineqlarge2}
    d^+(\phi_t w_1, \phi_t w_2) \ge e^{-\delta} e^{-\frac{2B}{a'}} e^{\frac{2t}{a'}} d^+(w_1, w_2).
    \ee

    Now consider the case when at least one of $s_1$ and $s_2$ is negative. Then by (2), we must have $0 \le t \le B$, and hence we deduce from (1) and (2) that $s_1, s_2 \in [-B/a, 2B/a]$.
    It then follows from \eqref{eqn.appexpcont1} that $$d^+(\varphi_{s_1} \sigma_1, \varphi_{s_2} \sigma_2) \le e^{\delta} e^{\frac{4B}{a}} d^+(\sigma_1, \sigma_2) \le e^{\delta} e^{\frac{4B}{a}} d^+(w_1, w_2) $$
    and that
    $$
    d^+(\phi_t w_1, \phi_t w_2) \ge d^+(\varphi_{s_1}\sigma_1, \varphi_{s_2} \sigma_2) \ge e^{-\delta} e^{-\frac{2B}{a}} d^+(\sigma_1, \sigma_2).
    $$
    Again, since $\sigma_1 \in \tilde{\Psi}^{-1}(w_1)$ and $\sigma_2 \in \tilde{\Psi}^{-1}(w_2)$ are arbitrary, these imply $$
    e^{-\delta} e^{-\frac{2B}{a}} d^+(w_1, w_2) \le d^+(\phi_t w_1, \phi_t w_2) \le e^{\delta} e^{\frac{4B}{a}} d^+(w_1, w_2).
    $$
    Since $0 \le t \le B$, we in particular have
    \be \label{eqn.expineqsmall}
    e^{-\delta} e^{-\frac{2B}{a} - \frac{2B}{a'}} e^{\frac{2t}{a'}} d^+(w_1, w_2) \le d^+(\phi_t w_1, \phi_t w_2) \le  e^{\delta} e^{\frac{4B}{a}} e^{\frac{2t}{a}} d^+(w_1, w_2).
    \ee
    Combining \eqref{eqn.expineqlarge}, \eqref{eqn.expineqlarge2}, and \eqref{eqn.expineqsmall}, the inequalities for $d^+$ in \eqref{eqn.expconc1} follows. The inequalities for $d^-$ in \eqref{eqn.expconc1}  can be shown by a similar argument.
\end{proof}

For $w_1, w_2 \in \tilde{\Omega}_{\psi}$, we also define 
\be \label{eqn.pushvisual}
d_{\varepsilon}^+(w_1, w_2) := d_{\varepsilon}(\sigma_1^+, \sigma_2^+)
\ee
where $\sigma_1 \in \tilde{\Psi}^{-1}(w_1)$ and $\sigma_2 \in \tilde{\Psi}^{-1}(w_2)$. Since every elements of $\tilde{\Psi}^{-1}(w)$ has the common forward endpoint for each $w \in \tilde \Omega_{\psi}$, this is well-defined.

\begin{lemma} \label{lem.compactvisual}
    For any compact subset $Q \subset \tilde{\Omega}_{\psi}$, there exists a constant $c_Q \ge 1$ such that for any $w_1, w_2 \in Q$, we have $$\frac{1}{c_Q} d^+(w_1, w_2)^{\varepsilon} \le d_{\varepsilon}^+(w_1, w_2) \le c_Q d^+(w_1, w_2)^{\varepsilon}.$$
\end{lemma}

\begin{proof}
    Let $Q \subset \tilde{\Omega}_{\psi}$ be a compact subset. Since $\tilde{\Psi}$ is proper, it follows from Lemma \ref{lem.compactvisual0} that there exists a uniform constant $c_Q \ge 1$ such that     
    if $w_1, w_2 \in Q$ and $\sigma_1 \in \tilde{\Psi}^{-1}(w_1)$ and $\sigma_2 \in \tilde{\Psi}^{-1}(w_2)$, then $$\frac{1}{c_Q} d^+(\sigma_1, \sigma_2)^{\varepsilon} \le d_{\varepsilon}^+(w_1, w_2) \le c_Q d^+(\sigma_1, \sigma_2)^{\varepsilon} \le c_Q d^+(w_1, w_2)^{\varepsilon}.$$
    Since $\sigma_1 \in \tilde{\Psi}^{-1}(w_1)$ and $\sigma_2 \in \tilde{\Psi}^{-1}(w_2)$ are arbitrary, the claim follows.
\end{proof}

\subsection*{Proof of Theorem \ref{thm.nicemetrics}}
Let $d^{\pm} : \tilde \Omega_{\psi} \times \tilde \Omega_{\psi} \to \R$ be functions defined in \eqref{eqn.defofnicemetrics}. From the definition, $d^{\pm}$ are non-negative and symmetric. Moreover, they are $\Ga$-invariant by \eqref{eqn.Gainvofmetric}. 

Let $z \in \tilde \Omega_{\psi}$. We show that the restriction on $d^{+}$ defines a semi-metric on $\tilde W^{+}(z)$; the corresponding statement for $d^{-}$ can be shown by the same argument. It suffices to show that for $w_1, w_2 \in \tilde W^{+}(z)$, $d^{+}(w_1, w_2) = 0$ if and only if $w_1 = w_2$. Suppose first that $w_1 = w_2$. Then for any $\sigma_1, \sigma_2 \in \tilde \Psi^{-1}(w_1) =\tilde \Psi^{-1}(w_2) $, we have $\sigma_1^+ = \sigma_2^+$. This implies $(\sigma_1 | \sigma_2)_x = \infty$. Hence, by \eqref{eqn.buseandgrom}, we have $d^+(\sigma_1, \sigma_2) = 0$. Since $\sigma_1, \sigma_2 \in \tilde \Psi^{-1}(w_1) =\tilde \Psi^{-1}(w_2) $ are arbitrary, we have $d^+(w_1, w_2) = 0$. Conversely, suppose that $d^+(w_1, w_2) = 0$. Let $\sigma_1 \in \tilde \Psi^{-1}(w_1)$ and $\sigma_2 \in \tilde \Psi^{-1}(w_2)$. We then have $d^+(\sigma_1, \sigma_2) = 0$, and hence $(\sigma_1^+ | \sigma_2^+)_x = \infty$ by \eqref{eqn.buseandgrom}, from which we deduce $\sigma_1^+ = \sigma_2^+$. Since $\tilde\Psi(\sigma_1) = w_1$ and $\tilde \Psi(\sigma_2) = w_2$, it follows from $w_1, w_2 \in \tilde W^+(z)$ and Lemma \ref{lem.transbyhoro} that $w_1 = w_2$, showing the claim.

The inequalities in (1) and (2) follow from Lemma \ref{lem.expcontmain}, finishing the proofs of (1) and (2).

We now show (3). For small enough $\varepsilon > 0$, we consider the function $d_{\varepsilon}^+ : \tilde \Omega_{\psi} \times \tilde \Omega_{\psi} \to \R$ defined in \eqref{eqn.pushvisual}, that is, for $w_1, w_2 \in \tilde \Omega_{\psi}$, $$d_{\varepsilon}^+(w_1, w_2) = d_{\varepsilon}(\sigma_1^+, \sigma_2^+)$$ where $\sigma_1 \in \tilde \Psi^{-1}(w_1)$ and $\sigma_2 \in \tilde \Psi^{-1}(w_2)$, and $d_{\varepsilon}$ is the visual metric on $\partial X_{GM}$ given in \eqref{eqn.visualmetric}. Since $d_{\varepsilon}$ is a metric, $d_{\varepsilon}^+$ is symmetric and satisfies the triangle inequality. Let $z \in \tilde \Omega_{\psi}$ and $w_1, w_2 \in \tilde W^+(z)$. As discussed above, for $\sigma_1 \in \tilde \Psi^{-1}(w_1)$ and $\sigma_2 \in \tilde \Psi^{-1}(w_2)$, we have $w_1 = w_2 \Leftrightarrow \sigma_1^+ = \sigma_2^+$ since $w_1, w_2 \in \tilde W^+(z)$. Hence $d_{\varepsilon}^+(w_1, w_2) = 0$ if and only if $w_1 = w_2$, and therefore the restriction of $d_{\varepsilon}^+$ defines a metric on $\tilde W^+(z)$.
The inequality stated in (3) is proved in Lemma \ref{lem.compactvisual}. This completes the proof.
\qed

\section{Finiteness of Bowen-Margulis-Sullivan measures}
Let $\Ga<G$ be a $\theta$-Anosov subgroup relative to $\cal P$ and $X_{GM}$ the associated Groves-Manning cusp space. Let $\psi \in \fa_{\theta}^*$ be a $(\Ga, \theta)$-proper linear form tangent to the $\theta$-growth indicator $\psi_{\Ga}^{\theta}$. 
 By \cite{CZZ_relative}, there exists a unique $(\Ga, \psi)$-Patterson-Sullivan measure $\nu_{\psi}$ on $\La_{\theta}$ and a unique $(\Ga, \psi \circ \i)$-Patterson-Sullivan measure $\nu_{\psi \circ \i}$ on $\La_{\i(\theta)}$. Let $m_{\psi}$ be the Bowen-Margulis-Sullivan measure on $\Omega_{\psi}$ associated with the pair $(\nu, \nu_{\i})$ defined in \eqref{mp}.

The relatively Anosov subgroups are regarded as the higher-rank generalization of geometrically finite subgroups. Indeed, same as geometrically finite subgroups, relatively Anosov subgroups have finite Bowen-Margulis-Sullivan measures:

\begin{theorem} \label{thm.finitebms} 
    We have $$|m_\psi| :=m_{\psi}(\Omega_{\psi}) < \infty.$$
\end{theorem}

We prove this finiteness of the Bowen-Margulis-Sullivan measure as a consequence of our reparameterization theorem (Theorem \ref{thm.repar}).

\subsection*{Thick-thin decomposition of $\Omega_{\psi}$} Let $\Psi : \Ga \ba \cal G \to \Omega_{\psi}$ be the reparameterization given in Theorem \ref{thm.repar}. Via $\Psi$, the decomposition $\cal G = \cal G_{thick} \cup \cal G_{thin}$ gives the thick-thin decomposition $$\Omega_{\psi} = \Psi(\Ga \ba \cal G_{thick}) \cup \Psi(\Ga \ba \cal G_{thin})$$ into the compact thick part $\Psi(\Ga \ba \cal G_{thick})$ and the thin part $\Psi(\Ga \ba \cal G_{thin})$.

The followings are extra ingredients in the proof:
    \begin{lemma}[Shadow lemma] \cite[Lemma 7.2]{KOW_indicators} \label{lem.shadow}  For all large enough $R > 0$, there exists $c_0 = c_0 (\psi, R) \ge 1$ such that for all $\ga \in \Ga$, 
    $$
   c_0^{-1} e^{-\psi(\mu_{\theta}(\ga))}\le  \nu_\psi (O_R^{\theta}(o, \ga o)) \le c_0  e^{-\psi(\mu_{\theta}(\ga))}.$$
\end{lemma}

We denote by $0\le \delta_{\psi}(\Ga)\le\infty $ the abscissa of convergence of the Poincar\'e series $s \mapsto \sum_{\ga \in \Ga} e^{-s \psi(\mu_{\theta}(\ga))}$; this is well-defined by the $(\Ga, \theta)$-properness hypothesis on $\psi$. Indeed, the $(\Ga, \theta)$-properness implies $\delta_{\psi}(\Ga)< \infty$ as shown in \cite[Theorem 1.3]{CZZ_relative}. Since $\psi$ is tangent to $\psi_{\Ga}^{\theta}$, we furthermore have $$\delta_{\psi}(\Ga) = 1$$ \cite[Theorem 4.5]{KOW_indicators}. On the other hand, we have the following:

\begin{theorem}[Canary-Zhang-Zimmer, {\cite[Lemma 8.2, Corollary 7.2]{CZZ_relative}}] \label{thm.entdrop}
If $\psi \in \fa_{\theta}^*$ is $(\Ga, \theta)$-proper and tangent to $\psi_{\Ga}^{\theta}$, then
the Patterson-Sullivan measure $\nu_{\psi}$ is atomless and for each $P \in \cal P$, we have  $$\delta_{\psi}(P) < 1.$$
\end{theorem}

\subsection*{Proof of Theorem \ref{thm.finitebms}}
As before, we identify $\La_{\theta}$ and $\La_{\i(\theta)}$ with $\partial X_{GM}$ through the boundary maps.
Recall the norm $\| \cdot \|_{\sigma}$ on $\R_{+}$ for each $\sigma \in \cal G$ and the $\Ga$-equivariant surjective proper map $\tilde{\Psi} : \cal G \to \tilde{\Omega}_{\psi}$, $\sigma \mapsto (\sigma^+, \sigma^-, \log v_{\sigma})$, defined in the proof of Theorem \ref{thm.repar} where $v_{\sigma} \in \R_{+}$ is the unique vector such that $\|v_{\sigma}\|_{\sigma} = 1$. We then have $$\tilde{\Omega}_{\psi} = \tilde{\Psi}(\cal G_{thick}) \cup \tilde{\Psi}(\cal G_{thin}).$$ We will use this specific decomposition to show the finiteness of $m_{\psi}$. Since $\Ga$ acts cocompactly on $\tilde{\Psi}(\cal G_{thick})$, it suffices to show that the measure of thin part $m_{\psi}(\Ga \ba \tilde{\Psi}(\cal G_{thin}))$ is finite. Moreover, since $\cal G_{thin} = \Ga \cdot \bigcup_{P \in \cal P} \cal G_P$ and $\cal P$ is a finite collection, it suffices to show $m_{\psi}(P \ba \tilde{\Psi}(\cal G_{P})) < \infty$ for each $P \in \cal P$.

Let us fix $P \in \cal P$ and denote by $\xi_P \in \partial X_{GM}$ the parabolic limit point fixed by $P$. Since $\xi_P$ is bounded parabolic, we have a compact fundamental domain for the $P$-action on $\partial X_{GM} - \{\xi_P\}$, which we denote by $D$. Since $\nu_{\psi}$ and $\nu_{\psi \circ \i}$ are atomless by Theorem \ref{thm.entdrop}, we have 
\be \label{eqn.bmsestimate}
m_{\psi}(P \ba \tilde{\Psi}(\cal G_P)) = \sum_{\ga \in P} \int_{(\ga D \times D \times \R) \cap \tilde{\Psi}(\cal G_P)} e^{\psi(\langle \xi, \eta \rangle)} d \nu_{\psi}(\xi) d \nu_{\psi \circ \i}(\eta) dt.
\ee

We first estimate the integration with respect to $dt$. We claim that
there exists $C > 0$ such that for any $\ga \in P$ and $\sigma \in \cal G_P$ such that $\sigma^- \in D$ and $\sigma^+ \in \ga D$, we have 
\be \label{eqn.claimfinite}
-C \le \log v_{\sigma} \le C + \psi(\mu_{\theta}(\ga)).
\ee
Let us fix $\ga \in P$ and let $\sigma \in \cal G_P$ be such that $\sigma^+ \in \ga D$ and $\sigma^- \in D$.
Recalling that $H_P$ denotes the open horoball in $X_{GM}$ associated to $P$, this implies that the following two constants are well-defined:
$$\begin{aligned}
    s_0 & := \min \{ s < 0 : \sigma(s) \in \partial H_P \} \\
    s_1 & := \max \{ s > 0 : \sigma(s) \in \partial H_P \}.
\end{aligned}$$
In other words, $s_0$ is the first time that $\sigma$ enters into $\partial H_P$ and $s_1$ is the last time that $\sigma$ exits $\partial H_P$. We then have from \eqref{eqn.cocycle} and Theorem \ref{thm.expdecay} that $$\begin{aligned}
    v_{\varphi_{s_0}\sigma} &= \| v_{\varphi_{s_0}\sigma} \|_{\sigma} v_{\sigma} = \kappa_{-s_0}(\varphi_{s_0}\sigma) v_{\sigma} \\
    &\le b e^{as_0} v_{\sigma} \le b v_{\sigma}; \\
    v_{\varphi_{s_1}\sigma} & = \frac{1}{\|v_{\sigma} \|_{\varphi_{s_1}\sigma}} v_{\sigma} = \frac{1}{\kappa_{s_1}(\sigma)} v_{\sigma} \\
    &\ge b^{-1} e^{as_1} v_{\sigma} \ge b^{-1} v_{\sigma}. 
\end{aligned}$$
Therefore, we have 
\be \label{eqn.finite0}
- \log b + \log v_{\varphi_{s_0}\sigma} \le \log v_{\sigma} \le \log b + \log v_{\varphi_{s_1}\sigma}.
\ee

Now fix $x \in \partial H_P$. Then there exists $R > 0$ with the following property: for any $\sigma_0  \in \cal G_P$ such that $\sigma_0^- \in D$, the entering point of $\sigma_0$ into $\partial H_P$, i.e. $\sigma_0(s) \in \partial H_P$ with minimal $s$, must be contained in the $R$-ball $B_{GM}(x, R)$. Indeed, if not, then there exists a sequence $\sigma_n \in \cal G_P$ such that $\sigma_n^- \in D$ and the entering point of $\sigma_n$ into $\partial H_P$ is not contained in $B_{GM}(x, n)$ for all $n \ge 1$. However, since $\sigma_n \in \cal G_P$ and $\sigma_n^- \in D$ for all $n \ge 1$, two sequences $\sigma_n^+$ and $\sigma_n^-$ converge to two distinct points in $\partial X_{GM}$ as $n \to \infty$, after passing to a subsequence. Hence the images of the bi-infinite geodesics $\sigma_n$ intersect a single ball centered at $x$, which contradicts the choice of the sequence $\sigma_n$.

Hence we have $(\varphi_{s_0}\sigma)(0) = \sigma(s_0) \in B_{GM}(x, R)$. Since $I(\ga^{-1}\sigma) \in \cal G_P$ also satisfies that $I(\ga^{-1}\sigma)^- = \ga^{-1} \sigma^+ \in D$ and its entering point into $\partial H_P$ is given by $I(\ga^{-1}\sigma)(-s_1) = \ga^{-1}\sigma(s_1)$, we also have $\ga^{-1}\sigma(s_1) \in B_{GM}(x, R)$. In other words, we have $(\ga^{-1}\varphi_{s_0}\sigma)(s_1 - s_0) \in B_{GM}(x, R)$. Hence we can apply Lemma \ref{lem.cptpart} to $\varphi_{s_0}\sigma$ by setting $Q = \overline{B_{GM}(x, R)}$ and obtain 
\be \label{eqn.finite1}
\frac{1}{C_Q} e^{-\psi(\mu_{\theta}(\ga))} \le \kappa_{s_1 - s_0} (\varphi_{s_0}\sigma) \le C_Q e^{-\psi(\mu_{\theta}(\ga))}.
\ee
Since $$v_{\varphi_{s_1}\sigma} = \frac{1}{\| v_{\varphi_{s_0}\sigma} \|_{\varphi_{s_1}\sigma}} v_{\varphi_{s_0}\sigma} = \frac{1}{\kappa_{s_1 - s_0}(\varphi_{s_0}\sigma)} v_{\varphi_{s_0}\sigma}$$ by \eqref{eqn.cocycle}, it follows from \eqref{eqn.finite1} that $$\log v_{\varphi_{s_1}\sigma} \le \log C_Q + \log v_{\varphi_{s_0}\sigma} + \psi(\mu_{\theta}(\ga)).$$
Hence we deduce from \eqref{eqn.finite0} that 
$$
- \log b + \log v_{\varphi_{s_0}\sigma} \le \log v_\sigma \le \log (b C_Q) + \log v_{\varphi_{s_0}\sigma} + \psi(\mu_{\theta}(\ga)).
$$
Since $(\varphi_{s_0}\sigma)(0) \in B_{GM}(x, R)$ where $x$ is fixed and $R$ is determined by $x$ and $P$, the constant $\log v_{\varphi_{s_0}\sigma}$ is also uniformly bounded.
Therefore, the claim \eqref{eqn.claimfinite} follows.

\medskip

By the claim \eqref{eqn.claimfinite}, we deduce from \eqref{eqn.bmsestimate} that 
$$\begin{aligned}
m_{\psi}(P & \ba \tilde{\Psi}(\cal G_P)) \\
& \le \sum_{\ga \in P} (2C + \psi(\mu_{\theta}(\ga))) \int_{(\ga D \times D) \cap \{(\sigma^+, \sigma^-) : \sigma \in \cal G_P\}} e^{\psi(\langle \xi, \eta \rangle)} d \nu_{\psi}(\xi) d \nu_{\psi \circ \i}(\eta).
\end{aligned}$$
As we already observed, for $x \in \partial H_P$ and $R > 0$ above, we have that if $\sigma \in \cal G_P$ is such that $\sigma^- \in D$ and $\sigma^+ \in \ga D$, then the image of the bi-infinite geodesic $\sigma$ must intersect $B_{GM}(x, R)$ and $B_{GM}(\ga x, R)$. Hence it follows from Lemma \ref{lem.gromovproductbdd} that 
\be \label{eqn.finite2}
\psi(\langle \sigma^+, \sigma^- \rangle) \quad \text{is uniformly bounded.}
\ee 
Moreover, we also have that $\sigma^+ \in O_{R'}^{GM}(x, \ga x)$ for some $R' > 0$ depending on $x$ and $R$. By Proposition \ref{prop.shadowcompare}, we then have for some uniform $r > 0$ that
\be \label{eqn.finite3}
\sigma^+ \in O_r^{\theta}(o, \ga o).
\ee
By \eqref{eqn.finite2} and \eqref{eqn.finite3}, we now have 
$$m_{\psi}(P \ba \tilde{\Psi}(\cal G_P)) \ll\footnote{ The notation $f \ll g $ means that there is a constant $c>0$ such that
$f \le c g$} \sum_{\ga \in P} (2C + \psi(\mu_{\theta}(\ga))) \nu_{\psi}(O_r^{\theta}(o, \ga o)).$$
Applying the shadow lemma (Lemma \ref{lem.shadow}), we finally obtain $$m_{\psi}(P \ba \tilde{\Psi}(\cal G_P)) \ll \sum_{\ga \in P} (2C + \psi(\mu_{\theta}(\ga))) e^{-\psi(\mu_{\theta}(\ga))}.$$

Let $0 < \varepsilon < 1$.
Since $\psi$ is $(\Ga, \theta)$-proper, $\liminf_{\ga \in P} \psi(\mu_{\theta}(\ga)) = \infty$, and hence $\psi (\mu_{\theta}(\ga)) \ll e^{\varepsilon \psi(\mu_{\theta}(\ga))}$.
Hence $$m_{\psi}(P \ba \tilde{\Psi}(\cal G_P)) \ll \sum_{\ga \in P} (2C + \psi(\mu_{\theta}(\ga))) e^{-\psi(\mu_{\theta}(\ga))} \ll \sum_{\ga \in P} e^{-(1 - \varepsilon)\psi(\mu_{\theta}(\ga))}.$$ By Theorem \ref{thm.entdrop}, for $\varepsilon > 0$ sufficiently small, we have $$m_{\psi}(P \ba \tilde{\Psi}(\cal G_P)) \ll \sum_{\ga \in P} e^{-(1 - \varepsilon)\psi(\mu_{\theta}(\ga))} < \infty.$$
This completes the proof of Theorem \ref{thm.finitebms}.
\qed

\section{Unique measure of maximal entropy}

Let $\Ga$ be a relatively $\theta$-Anosov subgroup and $\psi \in \fa_{\theta}^*$ a $(\Ga, \theta)$-proper form tangent to $\psi_{\Ga}^{\theta}$. Let $m_{\psi}$ be the Bowen-Margulis-Sullivan measure on $\Omega_{\psi}$. This section is devoted to the proof of the following: by Theorem \ref{thm.finitebms}, $m_\psi$ is of finite measure.

\begin{theorem} \label{thm.maxentropy} \label{thm.mme}
    Let $m$ be a probability $\{\phi_t\}$-invariant measure on $\Omega_{\psi}$. Then the metric entropy  $h_m(\{\phi_t\}) $ is at most  $\delta_{\psi} = 1,$ and $h_m(\{\phi_t\}) =1$ if and only if $m = m_{\psi}/| m_{\psi} |$, the normalized probability measure of $m_{\psi}$.
\end{theorem}

We recall some basic notions about entropy; we refer to (\cite{KH_book}, \cite{EL_book}) for details.

\subsection*{Measurable partitions and entropy}
Let $(\cal X, \cal M, m)$ be a probability space, where $\cal M$ is a $\sigma$-algebra and $m$ is a probability measure.
By a partition $\zeta$ of $\cal X$, we mean a collection of disjoint non-empty measurable subsets of $\cal X$ whose union is $\cal X$. 
For a partition $\zeta$ of $\cal X$ and $x \in  \cal X$, we denote by $\zeta(x)$ the element of $\zeta$ containing $x$, called the {\it atom} at $x$. Let $\cal M_\zeta \subset \cal M$ be the sub $\sigma$-algebra generated by the atoms of $\zeta$. 
A partition $\zeta$ of $\cal X$ is called {\it $m$-measurable} if it admits a separation by countably many elements in $\cal M_{\zeta}$. More precisely, $\zeta$ is $m$-measurable if
there exist a $m$-conull subset $\cal Y \subset \cal X$ and a sequence $\{Y_i \in \cal M_{\zeta}:i\in \N \}$ such that for any distinct atoms $z, z' $ of $ \zeta$, there exists $i \in \N$ such that either $z \cap \cal Y \subset Y_i$ and $z' \cap \cal Y \subset \cal X - Y_i$, or $z \cap \cal Y \subset \cal X - Y_i$ and $z' \cap \cal Y \subset Y_i$.

For an $m$-measurable partition $\zeta$ and $m$-a.e. $x \in \cal X$, we denote by $m_{\zeta(x)}$ the {\it conditional measure}  on the atom $\zeta(x)$ so that the following holds \cite[Theorem 5.9]{EL_book}: for any measurable $Y \subset \cal X$, we have
\begin{itemize}
    \item $x \mapsto m_{\zeta(x)}(Y \cap \zeta(x))$ is measurable;
    \item $m(Y) = \int_{\cal X} m_{\zeta(x)}(Y \cap \zeta(x)) \ dm(x)$.
\end{itemize}

For two $m$-measurable partitions $\zeta, \zeta'$, we say that $\zeta$ is {\it finer} than $\zeta'$ and write $\zeta \succ \zeta'$ if for  $m$-a.e. $x \in \cal X$, $\zeta(x) \subset \zeta'(x)$. For a sequence of $m$-measurable partitions $\zeta_i$, we denote by $\bigvee_{i } \zeta_i$ the smallest $m$-measurable partition finer than all $\zeta_i$.

Given an $m$-measurable partition $\zeta$ and an $m$-measurable map $\varphi : \cal X\to \cal X$, the pull-back $\varphi^{-1}\zeta$ is an $m$-measurable partition with atoms $(\varphi^{-1}\zeta)(x) = \varphi^{-1}(\zeta(\varphi(x)))$. We say that $\zeta$ is {\it $\varphi$-decreasing} if $\varphi^{-1}\zeta \succ \zeta$ and  {\it $\varphi$-generating} if $\bigvee_{i \in \N} \varphi^{-i}\zeta$ is $m$-equivalent to the partition consisting of points.

Let $\varphi : \cal X \to \cal X$ be an $m$-measure-preserving transformation. For a countable partition $\zeta$, the {\it entropy of $\zeta$ relative to $m$} is $$H_m(\zeta) := \int_{\cal X} - \log m(\zeta(x)) \ dm(x)$$
with the convention that $\infty \cdot 0 = 0$. The average entropy of $\zeta$ is defined as $$H_m(\varphi, \zeta) := \lim_{n \to \infty} \frac{1}{n} H_m\left( \bigvee_{i = 0}^{n-1} \varphi^{-i}\zeta \right).$$
The {\it metric entropy} of $\varphi$ with respect to $m$ is defined as $$h_m(\varphi) := \sup H_m(\varphi, \zeta)$$
where the supremum is taken over all countable partitions $\zeta$ with $H_m(\zeta) < \infty$.
For a flow $\{\phi_t\}_{t \in \R}$ on $\cal X$, we have $h_m(\phi_t) = |t| h_m(\phi_1)$ for all $t \neq 0$. The metric entropy of the flow $\{\phi_t\}$ with respect to $m$ is defined as $$h_m(\{\phi_t\}) := h_m(\phi_1) .$$

For a $\varphi$-decreasing $m$-measurable partition $\zeta$, we also define
$$h_m(\varphi, \zeta) := \int_{\cal X} - \log m_{\zeta(x)}((\varphi^{-1}\zeta)(x)) \ dm(x) .$$ 

\subsection*{Partition realizing the entropy}
Recall the foliations $\tilde W^{\pm}$ of $\tilde \Omega_{\psi}$ and $W^{\pm}$ of $\Omega_{\psi}$ from \eqref{ww} and \eqref{ww2}. Let $m$ be a probability measure on $\Omega_{\psi}$ and $\tilde m$ the $\Ga$-invariant lift of $m$ to $\tilde \Omega_{\psi}$.
A $\Ga$-invariant partition $\tilde \zeta$ of $\tilde \Omega_{\psi}$ is called {\it $\tilde m$-measurable} if the induced partition $\zeta$ on $\Omega_{\psi}$ is $m$-measurable.
We say that an $\tilde m$-measurable partition $\tilde \zeta$ is {\it subordinated to} $\tilde W^{+}$ if for $\tilde m$-a.e. $\tilde x \in \tilde \Omega_{\psi}$, there exist precompact open neighborhoods $\tilde{\cal U}_1$ and $\tilde{\cal U}_2$ of $\tilde x$ in $\tilde W^+(\tilde x)$
such that $$\tilde{ \cal U}_1\subset \tilde \zeta(\tilde x)\subset \tilde{\cal U}_2$$

\begin{proposition} \label{prop.prop1} \label{prop.prop4}
    Let $\tau > 0$. Let $m$ be a probability measure on $\Omega_{\psi}$ which is invariant and ergodic under $\phi_\tau$ and $\tilde m$ its lift to $\tilde \Omega_{\psi}$. Then there exists a $\Ga$-invariant $\tilde m$-measurable partition $\tilde \zeta$ of $\tilde \Omega_{\psi}$ subordinated to $\tilde W^+$ such that its projection $\zeta$ is an $m$-measurable $\phi_\tau$-decreasing and generating partition of $\Omega_{\psi}$ which satisfies   $$h_m(\phi_\tau) = h_m(\phi_\tau, \zeta) < \infty.$$
\end{proposition}

The most delicate part of the proof of this proposition lies in the
 construction of the partition which is subordinated to the unstable foliation $\tilde W^+$. The exponential expansion property of the flow $\{\phi_t\}$ on $\Omega_\psi$ (Theorem \ref{thm.nicemetrics}) was obtained precisely for this purpose. Other parts of Proposition \ref{prop.prop4} can be obtained by similar argument in \cite{OP_variational}.

\subsection*{Proof of Proposition \ref{prop.prop1}}
Let $d^{\pm}$ and $d_{\varepsilon}^+$ be functions on $\tilde \Omega_{\psi} \times \tilde \Omega_{\psi}$ given in Theorem \ref{thm.nicemetrics} for some fixed $\e>0$.
Fix $u \in \tilde \Omega_{\psi}$.
For  $r > 0$, we set
$$\tilde C(u, r)=\left\{ v\in \tilde \Omega_{\psi}:    
\begin{matrix}
    \exists s \in (-r, r), w \in \tilde W^-(\phi_s u) \text{ with } d^-(\phi_s u, w) < r \\
    \text{s.t. }v \in \tilde W^+(w) \text{ and } d_{\varepsilon}^+(w, v) < r
\end{matrix}
\right\}. $$

Fix $\rho>0$ small enough so that the projection $\tilde \Omega_{\psi} \to \Omega_{\psi}$ is injective on $\tilde C (u, 4 \rho)$. 
For $0 < r < 4\rho$, we denote by $C(u, r)$ the image of $\tilde C(u, r)$ under the projection
$\tilde \Omega_{\psi} \to \Omega_{\psi}$.

We define a function $\ell : \Omega_{\psi} \to \R$ as follows:  for each $x \in C(u, \rho)$, let $\tilde x \in \tilde C(u, \rho)$ be the unique lift of $x$.
It follows from the description of $\tilde W^{\pm}$
in Lemma \ref{lem.transbyhoro} that
 there exist unique $s \in (-\rho, \rho)$ and $\tilde y\in \tilde W^-(\phi_s u)$ such that $\tilde x \in \tilde W^+(\tilde y)$, $d^-(\phi_s u, \tilde y) < \rho$ and  $d_{\varepsilon}^+(\tilde y, \tilde x) < \rho$.
We set $$\ell(x) := \max(s, d^-(\phi_s u, \tilde y), d_{\varepsilon}^+(\tilde y, \tilde x)).$$
For  $x \in \Omega_{\psi}- C(u, \rho)$,  we then set $\ell(x) := \rho$. 

For each $0 < r < \rho$, let $\tilde \zeta'_r$ be the partition of $\tilde \Omega_{\psi}$ with atoms 
$\gamma \tilde C(u, r) \cap \tilde W^+(\tilde x)$ for $\tilde x \in \tilde \Omega_{\psi}$, $\ga\in \Ga$ and $\tilde \Omega_{\psi} - \Ga \tilde C(u, r)$.  
  We then define $$\tilde \zeta_r:=\bigvee_{i = 0}^{\infty} \phi_\tau^i \tilde \zeta'_r.$$
  Let 
 $\zeta'_r$ and $\zeta_r$ be the partitions obtained by projecting $\tilde \zeta'_r$ and $\tilde \zeta_r$ to $\Omega_\psi$ respectively.
  Then $\zeta_r = \bigvee_{i = 0}^{\infty} \phi_\tau^i \zeta'_r$ since the $\Ga$-action commutes with the flow $\{\phi_t\}$.
It is clear that $\zeta_r$ is $\phi_\tau$-decreasing. In view of the construction of $\tilde \zeta$ which uses atoms $\ga \tilde C(u, r)\cap W^+(\tilde x)$,
we can verity that $\zeta_r$ is $m$-measurable by a same argument as in \cite[Proposition 1]{OP_variational}.
Denote by $\tilde m$ is the lift of $m$ to $\tilde \Omega_\psi$.
Let $\d$ be the metric on $\Omega_{\psi}$ considered in Proposition \ref{prop.admissiblemetric}. By the ergodicity of $m$, we have that for $m$-a.e. $x \in \Omega_{\psi}$, $\phi_{\tau}^k x \in C(u, r)$ for infinitely many $k \in \N$, and hence $\zeta_r'(\phi_{\tau}^k x)$ is contained in a uniformly bounded set $C(u, r) \cap W^+(\phi_{\tau}^k x)$ with respect to $\d$. Since $(\phi_{\tau}^{-k} \zeta_r)(x) \subset \phi_{\tau}^{-k}(\zeta_r'(\phi_{\tau}^k x))$, it follows from Proposition \ref{prop.admissiblemetric} that $\zeta_r$ is $\phi_{\tau}$-generating. Similarly, for $\tilde m$-a.e. $\tilde x \in \tilde \Omega_{\psi}$, we have $\phi_{\tau}^{-k} \tilde x \in \ga \tilde C(u, r)$ for some $k \in \N$ and $\ga \in \Ga$. Hence we have $\tilde \zeta_r (\tilde x) \subset \phi_{\tau}^k (\tilde \zeta_r'(\phi_{\tau}^{-k} \tilde x)) \subset \phi_{\tau}^k \ga \tilde C(u, r) \cap \tilde W^+(\tilde x)$, and therefore $\tilde \zeta_r(\tilde x)$ is a precompact subset of $\tilde W^+(\tilde x)$.

We now show the most delicate part
of the proof that we can take $r > 0$ so that $\tilde \zeta_r(\tilde x)$ contains an open neighborhood of $\tilde x$ in $W^+(\tilde x)$ for $\tilde m$-a.e. $\tilde x \in \tilde \Omega_{\psi}$. We use
Theorem \ref{thm.nicemetrics} in a crucial way.

Consider the push-forward $\ell_*m$ of the measure $m$ by $\ell$, which is a probability measure on  $[0, \rho] \subset \R$.
For any $\varepsilon_0 \in (0, 1)$, we have
that $$\Leb\left( \left\{ r \in (0, \rho) : \sum_{k = 0}^{\infty} (\ell_*m)([r - \varepsilon_0^k, r + \varepsilon_0^k]) < \infty \right\} \right) = \rho$$
by \cite[Proposition 3.2]{LS_entropy}. Since $m$ is $\phi_\tau$-invariant, this is same to say that $$\Leb\left( \left\{ r \in (0, \rho) : \sum_{k = 0}^{\infty} m( \{x : |\ell(\phi_\tau^{-k}x) - r| <  \varepsilon_0^k \}) < \infty \right\} \right) = \rho.$$
We fix a constant $ e^{-\varepsilon \alpha \tau} <\varepsilon_0 <1$ where $\alpha > 0$ is a constant given in Theorem \ref{thm.nicemetrics}. 
We can therefore choose $0 < r < \rho/2$ so that $m(\partial C(u, r)) = 0$ and  that $$\sum_{k = 0}^{\infty} m( \{x : |\ell(\phi_\tau^{-k}x) - r| <  \varepsilon_0^k \}) < \infty.$$

 Let $\Omega_\psi'$ be the set of all $x\in  \Omega_\psi -
 \bigcup_{k = 0}^{\infty} \phi^k  \partial C(u, r) $ satisfying 
 that for some $N_0=N_0(x) > 0$,  we have \be\label{ll} \ell(\phi_\tau^{-k}x) < r - \varepsilon_0^k \quad \text{or} \quad \ell(\phi_\tau^{-k}x) > r + \varepsilon_0^k\ee 
  for all $k \ge N_0$.
Since $m(\partial C(u, r)) = 0$, it follows from the classical Borel-Cantelli lemma that $m(\Omega_\psi ')=1$.
Let $x \in \Omega_\psi' $  be an arbitrary point and corresponding $N_0=N_0(x)$.  We fix a lift $\tilde x \in \tilde \Omega_{\psi}$ of $x$.

For  $\tilde y \in \tilde \Omega_\psi$, we write $y$ for its projection to $\Omega_\psi$. 
Fix a compact subset $Q \subset \tilde \Omega_{\psi}$ containing $$\bigcup_{v_0 \in \tilde C (u, \rho)} \{v \in \tilde W^+(v_0) : d^+(v, v_0) \le b \}$$
where $b \ge 1$ is the constant given in Theorem \ref{thm.nicemetrics}.

We set
$$r_1 := \min \left(\frac{1}{2},   \frac{1}{b(2 c)^{1/\varepsilon}} \right) >0 $$
where  $c=c_Q \ge 1$ is as given in Theorem \ref{thm.nicemetrics}(3).
Let $$\tilde {\cal U}=  \{ \tilde{y} \in \tilde W^+(\tilde x): d^+(\tilde{x}, \tilde{y}) < r_1\};$$ this is a precompact neighborhood of $\tilde x$ in $\tilde W^+(\tilde x)$. Let $\cal U$ be the image of $\tilde{\cal U}$ in $\Omega_\psi$.
We claim that  for each $k \ge N_0$, either
\be\label{c1}  
\phi_{\tau}^{-k} (\tilde{\cal U}) \subset \ga^{-1} \tilde C(u, r) \text{ for some } \gamma \in \Ga \quad \text{or} \quad \phi_{\tau}^{-k} (\tilde{\cal U}) \cap \Ga \tilde C(u, r) = \emptyset.
\ee 

Fix $k\ge N_0$. Recall that $x$ satisfies either $\ell(\phi_\tau^{-k}x) < r - \varepsilon_0^k$
or $\ell(\phi_\tau^{-k}x) > r + \varepsilon_0^k$.
 Consider the first case.  This implies that there exists $\ga \in \Ga$ such that $\ga \phi_{\tau}^{-k} \tilde x \in \tilde C(u, r - \varepsilon_0^k)$. We then have
$$d^+(\ga \phi_{\tau}^{-k} \tilde x, \ga \phi_{\tau}^{-k} \tilde y) = d^+(\phi_{\tau}^{-k} \tilde x, \phi_{\tau}^{-k} \tilde y) \le b e^{-\alpha \tau k }d^+(\tilde x, \tilde y).$$
by \eqref{eqn.Gainvofmetric} and Theorem \ref{thm.nicemetrics}(1).
In particular, we have $\ga \phi_{\tau}^{-k} \tilde y \in Q$ and hence 
\be \label{eqn.apr081153}
d_{\varepsilon}^+(\ga \phi_{\tau}^{-k} \tilde x, \ga \phi_{\tau}^{-k} \tilde y) \le c d^+(\ga \phi_{\tau}^{-k} \tilde x, \ga \phi_{\tau}^{-k} \tilde y)^{\varepsilon} \le c b^{\varepsilon} e^{-\varepsilon \alpha \tau k} d^+(\tilde x, \tilde y)^{\varepsilon}
\ee
by Theorem \ref{thm.nicemetrics}(3). Let $\tilde{y} \in \tilde{\cal U}$, and hence  $d^+(\tilde{x}, \tilde{y}) < r_1$. Since $ e^{-\varepsilon \alpha \tau} < \varepsilon_0$, we then have $$d_{\varepsilon}^+(\ga \phi_{\tau}^{-k} \tilde x, \ga \phi_{\tau}^{-k} \tilde y) < \varepsilon_0^k$$
by \eqref{eqn.apr081153}, and therefore $\ga \phi_{\tau}^{-k} \tilde y \in \tilde C(u, r)$. Hence $$\phi_{\tau}^{-k} (\tilde{\cal U}) \subset \ga^{-1} \tilde C(u, r),$$
proving \eqref{c1} in this case.

Now consider the case when $\ell(\phi_{\tau}^{-k}x) > r + \varepsilon_0^k$. 
In this case, we claim that 
$\phi_{\tau}^{-k} (\tilde{\cal U})  \cap \Ga \tilde C(u, r) =\emptyset.$ Suppose not. Then
there exists $\ga\in \Ga$ and some
$\tilde{y} \in \tilde W^+(\tilde x)$ such that
$ d^+(\tilde{x}, \tilde{y}) < r_1$
and
$\ga \phi_{\tau}^{-k} \tilde y \in \tilde C(u, r)$. By the same argument as above, $\ga \phi_{\tau}^{-k} \tilde x \in Q$ and hence $$d_{\varepsilon}^+(\ga \phi_{\tau}^{-k} \tilde x, \ga \phi_{\tau}^{-k} \tilde y) \le c b^{\varepsilon} e^{-\varepsilon \alpha \tau k} d^+(\tilde x, \tilde y)^{\varepsilon}.$$
Since $d^+(\tilde x, \tilde y) < r_1$, we have  $\ga \phi_{\tau}^{-k}\tilde x \in \tilde C(u, r + \varepsilon_0^k)$. This is a contradiction since $\ell(\phi_{\tau}^{-1}x) > r + \varepsilon_0^k$, proving the claim.

The claim \eqref{c1} implies that  $\phi_{\tau}^{-k} (\tilde{\cal U})$ lies in a single atom of $\tilde \zeta'_r$ for each $k \ge N_0$. 

 Since $\phi_\tau^{-k}\tilde x \notin \partial \gamma^{-1} \tilde C(u, r)$ for all $k\in \N$ and $\gamma\in \Ga$,  we can find a small neighborhood $\tilde{\cal U}'\subset \tilde{\cal U}$ of $\tilde x$ in $\tilde W^+(\tilde x)$ such that  $\phi_\tau^{-k}(\tilde{\cal U}')$ is entirely contained in some $\gamma^{-1}\tilde C(u, r)$, $\gamma\in \Ga$ or disjoint from $ \Gamma\overline{C(u, r)}$ for each $0\le k\le N_0$.
 Therefore $\phi_\tau^{-k} (\tilde{\cal U}')$ is contained in a single atom of $\tilde \zeta'_r$ for all $k\in \N$. 
 This proves that the atom of $\tilde \zeta_r$ containing $\tilde x$ also contains $\tilde{\cal U}'$. Since $x \in \Omega_\psi'$ is arbitrary,  $\tilde \zeta_r$ is subordinated to $\tilde W^+$.

The rest of the argument is a similar entropy computation as in the deduction of \cite[Proposition 4]{OP_variational} from \cite[Proposition 1]{OP_variational}.
\qed

\subsection*{Proof of Theorem \ref{thm.maxentropy}}
The deduction of Theorem \ref{thm.maxentropy} from Proposition \ref{prop.prop4} can be done similarly to \cite{OP_variational}. 

First, note that $\delta_\psi = 1$ since $\psi$ is tangent to $\psi_{\Ga}^{\theta}$ (\cite[Theorem 10.1]{CZZ_relative}, \cite[Theorem 4.5]{KOW_indicators}). For $g \in G$ such that $[g] \in \tilde{\Omega}_{\psi}$, we consider the measure $\mu_{\tilde W^+([g])}$ on $\tilde{W}^+([g])$ given by $$d \mu_{\tilde W^+([g])}([gn]) = e^{\psi(\beta_{(gn)^+}^{\theta}(e, gn))} d\nu ((gn)^+)$$ for $n \in N_{\theta}^+$. 
It follows from  the definition that for all $a \in A_{\theta}$, we have 
\be \label{eqn.leafrn}
\frac{d a_* \mu_{\tilde W^+([g])}}{ d \mu_{ \tilde W^+([ga])}}(x) = e^{-\psi(\log a)}.
\ee

We write $m^{pr}$ for the normalized probability measure $m_{\psi} / |m_{\psi}|$. Denote by $\tilde m^{pr}$ its lift to $\tilde \Omega_\psi$.
The following can be obtained by directly checking the condition for conditional measures:

\begin{lemma} \label{lem.bmscond}
    Let $\tilde \zeta$ be an $\tilde m^{pr}$-measurable partition of $\tilde \Omega_{\psi}$ subordinated to $\tilde W^+$. Then the family of conditional measures of $\tilde m^{pr}$ with respect to $\tilde \zeta$ is given by $$d \tilde m^{pr}_{\tilde \zeta(\tilde x)}(w) := \frac{ \mathbbm{1}_{\tilde \zeta(\tilde x)}(w)}{\mu_{\tilde W^+(\tilde x)}(\tilde \zeta(\tilde x))} d \mu_{\tilde W^+(\tilde x)}(w)  \quad\text{for $\tilde x \in \tilde \Omega_{\psi}$.} $$
\end{lemma}

By Theorem \ref{thm.finitebms}, $m_{\psi}$ is finite, and hence it follows from Theorem \ref{thm.ergodic} that $m^{pr}$ is $\{\phi_t\}$-ergodic. It is a general fact that $m^{pr}$ is ergodic for the transformation $\phi_t$ for uncountably many $t$ \cite[Lemma 7]{OP_variational}. 
Fix $\tau > 0$ so that $m^{pr}$ is $\phi_{\tau}$-ergodic.
Now let $m$ be a probability $\{\phi_t\}$-invariant measure on $\Omega_{\psi}$. Considering the ergodic decomposition of $m$, we may assume that $m$ is $\phi_{\tau}$-ergodic without loss of generality \cite[(3.5a)]{EL_book}.

We now consider the partition $\tilde \zeta$ given by Proposition \ref{prop.prop4} for the measure $m$, its lift $\tilde m$, and the transformation $\phi_{\tau}$. Since $\tilde \zeta$ is subordinated to $\tilde W^+$, the measure 
$$d \tilde m_{\tilde \zeta(\tilde x)}^{pr}(w) := \frac{ \mathbbm{1}_{\tilde \zeta(\tilde x)}(w)}{\mu_{\tilde W^+(\tilde x)}(\tilde \zeta(\tilde x))} d \mu_{\tilde W^+(\tilde x)}(w)$$ and the function $$\tilde G( \tilde x) := - \log \mu_{\tilde W^+(\tilde x)}(\tilde \zeta(\tilde x))$$ are well-defined for $\tilde m$-a.e. $\tilde x \in \tilde \Omega_{\psi}$. Note that since $\tilde \zeta$ is a partition for the measure $\tilde m$, it may not be $\tilde m^{pr}$-measurable and hence  Lemma \ref{lem.bmscond} does not apply to $\tilde \zeta$.
 It follows from \eqref{eqn.leafrn} that for $\tilde m$-a.e. $\tilde x \in \tilde \Omega_{\psi}$, we have 
\be \label{eqn.taucompute}
- \log \tilde m_{\tilde \zeta(\tilde x)}^{pr} ( (\phi_\tau^{-1} \tilde \zeta)(\tilde x)) = \tau + (\tilde G \circ \phi_\tau)(\tilde x) - \tilde G(\tilde x).
\ee
This implies 
$$\tilde G \circ \phi_\tau - \tilde G \ge - \tau$$ $\tilde m$-a.e. Since $\tilde G$ is $\Ga$-invariant, it induces the function $G : \Omega_{\psi} \to \R$. By \cite[Lemme 8]{OP_variational}, we have $ \int G \circ \phi_\tau - G \ dm = 0$ and therefore 
\be \label{eqn.keyidentity}
\int - \log m_{\zeta(x)}^{pr} ((\phi_\tau^{-1}\zeta)(x))\ dm(x) = \tau.
\ee
where $m_{\zeta(x)}^{pr}$ is the measure on $\zeta(x)$ induced by $\tilde m_{\tilde \zeta(\tilde x)}^{pr}$.

We can now show $h_{m^{pr}}(\{\phi_t\}) = 1$. Indeed, if we consider the special case that $m = m^{pr}$, then the partition $\zeta$ becomes an $m^{pr}$-measurable partition given by Proposition \ref{prop.prop4}. Hence by Lemma \ref{lem.bmscond}, the measure $m_{\zeta(x)}^{pr}$ forms the family of conditional measure for $m^{pr}$. Therefore the above identity \eqref{eqn.keyidentity}
yields $$h_{m^{pr}}(\phi_\tau) = h_{m^{pr}}(\phi_\tau, \zeta) = \int - \log m_{\zeta(x)}^{pr} ((\phi_\tau^{-1} \zeta)(x)) \ dm(x) = \tau.$$
Hence
$$h_{m^{pr}}(\{\phi_t\}) = h_{m^{pr}}(\phi_{\tau}) / \tau = 1.$$

It remains to show that for a general $m$, $h_{m}(\{\phi_t\}) \le 1$ and that $h_m(\{\phi_t\}) = 1$ implies $m = m^{pr}$.
We define the following function: for $m$-a.e. $x \in \Omega_{\psi}$, $$
F(x)  := \frac{ m_{\zeta(x)}^{pr}((\phi_\tau^{-1} \zeta)(x))}{m_{\zeta(x)} ((\phi_\tau^{-1}\zeta)(x))} \quad  \text{if} \quad m_{\zeta(x)} ((\phi_\tau^{-1}\zeta)(x)) > 0,$$
and $F(x)  := 0$ otherwise.
By \cite[Fait 9]{OP_variational}, both functions $F$ and $\log F$ are $m$-integrable and $\int F \ dm \le 1$. Since $$\int \log F \ dm = - \tau + h_m(\phi_\tau, \zeta) = - \tau + h_m(\phi_{\tau}) = - \tau + \tau h_m(\{\phi_t\})$$
by \eqref{eqn.keyidentity} and the choice of $\zeta$, we apply Jensen's inequality and obtain $$-\tau + \tau h_m(\{\phi_t\}) \le \log \left( \int F \ dm \right) \le 0.$$
This proves $$h_m(\{\phi_t\}) \le 1.$$

Now suppose that $h_m(\{\phi_t\}) = 1$.
This implies that the equality holds in Jensen's inequality, that is,
$ \log \left( \int F \ dm \right) = 0$,
 which means that $F = 1$ $m$-a.e. It follows that the two conditional measures $m_{\zeta(x)}^{pr}$ and $m_{\zeta(x)}$ coincide on the $\sigma$-algebra generated by $(\phi_\tau^{-1}\zeta)(x)$ for $m$-a.e. $x$. Since this holds after replacing $\phi_\tau$ with $\phi_\tau^k$ for any $k \in \N$ and the partition $\zeta$ is $\phi_\tau$-generating, we have
 $$m_{\zeta(x)}^{pr} = m_{\zeta(x)}\quad\text{ for $m$-a.e. $x \in \Omega_{\psi}$.}$$

Then the equality between measures $m = m^{pr}$ follows from the Hopf argument. Indeed, let $f : \Omega_{\psi} \to \R$ be a compactly supported continuous function. By the Birkhoff ergodic theorem, the set $$\cal Z := \left\{x \in \Omega_{\psi} : \lim_{t \to \infty} \frac{1}{t} \int_0^t f( \phi_s x) ds = m^{pr}(f) \right\}$$ has a full $m^{pr}$-measure. Then $\cal Z$ is invariant under the flow $\{\phi_t\}$ and moreover, since $f$ is uniformly continuous, $x \in \cal Z$ implies $W^-(x) \subset \cal Z$ by Proposition \ref{prop.admissiblemetric}. By the quasi-product structure of the BMS measure $m^{pr}$, this implies that for all $x\in \Omega_\psi$,  $\cal Z \cap W^+(x)$ has full $\mu_{W^+(x)}$-measure.
Hence $\cal Z \cap \zeta(x)$ has full $m_{\zeta(x)}^{pr}$-measure for $m$-a.e. $x \in \Omega_{\psi}$ by the definition of $m_{\zeta(x)}^{pr}$. Hence $\cal Z \cap \zeta(x)$ has full $m_{\zeta(x)}$-measure for $m$-a.e. $x \in \Omega_{\psi}$. Since $m_{\zeta(x)}$ is a conditional measure for $m$, this implies $m(\cal Z) = 1$, and therefore $m(f) = m^{pr}(f)$ by applying the Birkhoff ergodic theorem again to $m$. This finishes the proof.
\qed

\end{document}